\documentclass{article}
\usepackage{amsmath}
\usepackage{amssymb}
\usepackage{amsthm}
\usepackage{cite}
\usepackage{tikz}
\usepackage{stmaryrd}
\SetSymbolFont{stmry}{bold}{U}{stmry}{m}{n}
\usepackage{slashed} 
\usepackage{euscript}
\usepackage[arrow,curve,matrix,arc,2cell]{xy}
\usepackage[utf8]{inputenc}
\usepackage[unicode]{hyperref}
\usepackage{mathtools}
\usepackage{enumitem}
\DeclareFontFamily{U}{rsfs}{} 
\DeclareFontShape{U}{rsfs}{n}{it}{<->
rsfs10}{} \DeclareSymbolFont{mscr}{U}{rsfs}{n}{it}
\DeclareSymbolFontAlphabet{\scr}{mscr}
\def\mathscr{\scr}
\begin{document}
\def\e#1\e{\begin{equation}#1\end{equation}}
\def\ea#1\ea{\begin{align}#1\end{align}}
\def\eq#1{{\rm(\ref{#1})}}
\theoremstyle{plain}
\newtheorem{thm}{Theorem}[section]
\newtheorem{lem}[thm]{Lemma}
\newtheorem{prop}[thm]{Proposition}
\newtheorem{cor}[thm]{Corollary}
\newtheorem{quest}[thm]{Question}
\theoremstyle{definition}
\newtheorem{dfn}[thm]{Definition}
\newtheorem{ex}[thm]{Example}
\newtheorem{rem}[thm]{Remark}
\numberwithin{equation}{section}
\numberwithin{figure}{section}
\def\dim{\mathop{\rm dim}\nolimits}
\def\codim{\mathop{\rm codim}\nolimits}
\def\Im{\mathop{\rm Im}\nolimits}
\def\det{\mathop{\rm det}\nolimits}
\def\Ker{\mathop{\rm Ker}}
\def\Coker{\mathop{\rm Coker}}
\def\Flag{\mathop{\rm Flag}\nolimits}
\def\FlagSt{\mathop{\rm FlagSt}\nolimits}
\def\Iso{\mathop{\rm Iso}\nolimits}
\def\Aut{\mathop{\rm Aut}}
\def\End{\mathop{\rm End}\nolimits}
\def\Spec{\mathop{\rm Spec}\nolimits}
\def\Ho{\mathop{\rm Ho}}
\def\GL{\mathop{\rm GL}}
\def\SL{\mathop{\rm SL}}
\def\SO{\mathop{\rm SO}}
\def\SU{\mathop{\rm SU}}
\def\Sp{\mathop{\rm Sp}}
\def\Spin{\mathop{\rm Spin}\nolimits}
\def\U{{\mathbin{\rm U}}}
\def\vol{\mathop{\rm vol}}
\def\inc{\mathop{\rm inc}}
\def\ind{{\rm ind}}
\def\irr{{\rm irr}}
\def\red{{\rm red}}
\def\db{\bar\partial}
\def\rk{\mathop{\rm rk}}
\def\rank{\mathop{\rm rank}\nolimits}
\def\Hom{\mathop{\rm Hom}\nolimits}
\def\id{{\mathop{\rm id}\nolimits}}
\def\Id{{\mathop{\rm Id}\nolimits}}
\def\TopSta{{\mathop{\bf TopSta}\nolimits}}
\def\Map{{\mathop{\rm Map}\nolimits}}
\def\Ad{\mathop{\rm Ad}}
\def\irr{{\rm irr}}
\def\coh{{\rm coh}}
\def\cs{{\rm cs}}
\def\Top{{\mathop{\bf Top}\nolimits}}
\def\Stab{\mathop{\rm Stab}\nolimits}
\def\ul{\underline}
\def\bs{\boldsymbol}
\def\ge{\geqslant}
\def\le{\leqslant\nobreak}
\def\boo{{\mathbin{\mathbf 1}}}
\def\O{{\mathcal O}}
\def\bA{{\mathbin{\mathbb A}}}
\def\bG{{\mathbin{\mathbb G}}}
\def\bL{{\mathbin{\mathbb L}}}
\def\P{{\mathbin{\mathbb P}}}
\def\bT{{\mathbin{\mathbb T}}}
\def\K{{\mathbin{\mathbb K}}}
\def\R{{\mathbin{\mathbb R}}}
\def\Z{{\mathbin{\mathbb Z}}}
\def\Q{{\mathbin{\mathbb Q}}}
\def\N{{\mathbin{\mathbb N}}}
\def\C{{\mathbin{\mathbb C}}}
\def\CP{{\mathbin{\mathbb{CP}}}}
\def\A{{\mathbin{\cal A}}}
\def\B{{\mathbin{\cal B}}}
\def\ovB{{\mathbin{\smash{\,\overline{\!\mathcal B}}}}}\def\G{{\mathbin{\cal G}}}
\def\M{{\mathbin{\cal M}}}
\def\cB{{\mathbin{\cal B}}}
\def\cC{{\mathbin{\cal C}}}
\def\cD{{\mathbin{\cal D}}}
\def\cE{{\mathbin{\cal E}}}
\def\cF{{\mathbin{\cal F}}}
\def\cG{{\mathbin{\cal G}}}
\def\cH{{\mathbin{\cal H}}}
\def\E{{\mathbin{\cal E}}}
\def\F{{\mathbin{\cal F}}}
\def\cG{{\mathbin{\cal G}}}
\def\cH{{\mathbin{\cal H}}}
\def\cI{{\mathbin{\cal I}}}
\def\cJ{{\mathbin{\cal J}}}
\def\cK{{\mathbin{\cal K}}}
\def\cL{{\mathbin{\cal L}}}
\def\cN{{\mathbin{\cal N}\kern .04em}}
\def\cP{{\mathbin{\cal P}}}
\def\cQ{{\mathbin{\cal Q}}}
\def\cR{{\mathbin{\cal R}}}
\def\cS{{\mathbin{\cal S}}}
\def\T{{{\cal T}\kern .04em}}
\def\cW{{\mathbin{\cal W}}}
\def\cX{{\cal X}}
\def\cY{{\cal Y}}
\def\cZ{{\cal Z}}
\def\cV{{\cal V}}
\def\cW{{\cal W}}
\def\g{{\mathfrak g}}
\def\h{{\mathfrak h}}
\def\m{{\mathfrak m}}
\def\u{{\mathfrak u}}
\def\su{{\mathfrak{su}}}
\def\al{\alpha}
\def\be{\beta}
\def\ga{\gamma}
\def\de{\delta}
\def\io{\iota}
\def\ep{\epsilon}
\def\la{\lambda}
\def\ka{\kappa}
\def\th{\theta}
\def\ze{\zeta}
\def\up{\upsilon}
\def\vp{\varphi}
\def\si{\sigma}
\def\om{\omega}
\def\De{\Delta}
\def\Ka{{\rm K}}
\def\La{\Lambda}
\def\Om{\Omega}
\def\Ga{\Gamma}
\def\Si{\Sigma}
\def\Th{\Theta}
\def\Up{\Upsilon}
\def\Chi{{\rm X}}
\def\Tau{{T}}
\def\Nu{{\rm N}}
\def\pd{\partial}
\def\ts{\textstyle}
\def\st{\scriptstyle}
\def\sst{\scriptscriptstyle}
\def\w{\wedge}
\def\sm{\setminus}
\def\lt{\ltimes}
\def\bu{\bullet}
\def\sh{\sharp}
\def\di{\diamond}
\def\he{\heartsuit}
\def\op{\oplus}
\def\od{\odot}
\def\ot{\otimes}
\def\bt{\boxtimes}
\def\ov{\overline}
\def\bigop{\bigoplus}
\def\bigot{\bigotimes}
\def\iy{\infty}
\def\es{\emptyset}
\def\ra{\rightarrow}
\def\rra{\rightrightarrows}
\def\Ra{\Rightarrow}
\def\Longra{\Longrightarrow}
\def\ab{\allowbreak}
\def\longra{\longrightarrow}
\def\hookra{\hookrightarrow}
\def\dashra{\dashrightarrow}
\def\lb{\llbracket}
\def\rb{\rrbracket}
\def\ha{{\ts\frac{1}{2}}}
\def\t{\times}
\def\ci{\circ}
\def\ti{\tilde}
\def\d{{\rm d}}
\def\md#1{\vert #1 \vert}
\def\ms#1{\vert #1 \vert^2}
\def\bmd#1{\big\vert #1 \big\vert}
\def\bms#1{\big\vert #1 \big\vert^2}
\def\an#1{\langle #1 \rangle}
\def\ban#1{\bigl\langle #1 \bigr\rangle}
\def\o{\operatorname{o}}
\def\D{\mathbin{\slashed{\operatorname{D}}}}
\def\O{\mathcal{O}}
\def\Eig{\operatorname{Eig}}
\title{On spin structures and orientations for gauge-theoretic moduli spaces}
\author{Dominic Joyce and Markus Upmeier}
\date{}
\maketitle

\begin{abstract}
 Let $X$ be a compact manifold, $G$ a Lie group, $P\ra X$ a principal $G$-bundle, and $\B_P$ the infinite-dimensional moduli space of connections on $P$ modulo gauge, as a topological stack. For a real elliptic operator $E_\bu$ we previously studied orientations on the real determinant line bundle over $\B_P$, twisting $E_\bu$ by connections $\nabla_{\Ad(P)}$. These are used to construct orientations in the usual sense on smooth gauge theory moduli spaces, and have been extensively studied since the work of Donaldson~\cite{Dona1,Dona2,DoKr}.
 
 Here we consider {\it complex} elliptic operators $F_\bu$ and introduce the idea of {\it spin structures}, square roots of the complex determinant line bundle of $F_\bu$ twisted over $\B_P$. These may be used to construct spin structures in the usual sense on smooth complex gauge theory moduli spaces. We study the existence and classification of such spin structures.
 
 Our main result identifies spin structures on $X$ with orientations on $X\t\cS^1$.  Thus, if $P\ra X$ and $Q\ra X\t\cS^1$ are principal $G$-bundles with $Q\vert_{X\t\{1\}}\cong P$, we relate spin structures on $(\B_P,F_\bu)$ to orientations on $(\B_Q,E_\bu)$ for a certain class of operators $F_\bu$ on $X$ and $E_\bu$ on $X\t\cS^1$.
 
 Combined with \cite{JoUp1}, we obtain canonical spin structures for positive Diracians on spin $6$-manifolds and gauge groups $G=\U(m), \SU(m)$. In a sequel \cite{JoUp2} we will apply this to define canonical {\it orientation data\/} for all Calabi--Yau 3-folds $X$ over $\C$, as in Kontsevich and Soibelman \cite[\S 5.2]{KoSo1}, solving a long-standing problem in Donaldson--Thomas theory.
\end{abstract}

\setcounter{tocdepth}{2}
\tableofcontents

\section{Introduction}
\label{os1}

Let $X$ be a compact manifold, $E_\bu=\bigl(D:\Ga^\iy(E_0)\ra\Ga^\iy(E_1)\bigr)$ a real elliptic operator on $X$, $G$ a Lie group, $P\ra X$ a principal $G$-bundle, and $\B_P$ the infinite-dimensional moduli space of all connections $\nabla_P$ on $P$ modulo gauge, as a topological stack. For each $[\nabla_P]\in\B_P$, we consider the twisted real elliptic operator $D^{\nabla_{\Ad(P)}}:\Ga^\iy(\Ad(P)\ot E_0)\ra\Ga^\iy(\Ad(P)\ot E_1)$ on $X$. As this is a continuous family of real elliptic operators over the base $\B_P$, it has a real determinant line bundle and associated orientation bundle $O^{E_\bu}_P$, a principal $\Z_2$\nobreakdash-bundle parametrizing orientations of $\Ker D^{\nabla_{\Ad(P)}}\op \Coker D^{\nabla_{\Ad(P)}}$ at each $[\nabla_P]$. An {\it orientation\/} on $(\B_P,E_\bu)$ is a trivialization $O^{E_\bu}_P\cong\B_P\t\Z_2$. Orientations were studied in the recent series \cite{CGJ,JTU,JoUp1} by the authors, Yalong Cao and Jacob Gross, and by previous authors such as  Donaldson~\cite{Dona1,Dona2,DoKr}.

In gauge theory one studies moduli spaces $\M_P^{\rm ga}$ of connections $\nabla_P$ on $P$ satisfying some curvature condition, such as anti-self-dual instantons on Riemannian 4-manifolds $(X,g)$. Under good conditions $\M_P^{\rm ga}$ is a smooth manifold, and orientations on $(\B_P,E_\bu)$, where $E_\bu$ is determined by the curvature condition, pull back to orientations on the manifold $\M_P^{\rm ga}$ in the usual sense under the inclusion $\M_P^{\rm ga}\hookra\B_P$. This is important in areas such as Donaldson theory \cite{Dona1,Dona2,DoKr}, where one needs an orientation on $\M_P^{\rm ga}$ to define enumerative invariants.\smallskip

This paper will define and study a {\it complex} version of orientations, which we call {\it spin structures}. Now, let $F_\bu=\bigl(D:\Ga^\iy(F_0)\ra\Ga^\iy(F_1)\bigr)$ be a complex elliptic operator on $X$ and let $G,P,\B_P$ be as above. For each $[\nabla_P]\in\B_P$, the twisted operator $D^{\nabla_{\Ad(P)}}:\Ga^\iy(\Ad(P)\ot F_0)\ra\Ga^\iy(\Ad(P)\ot F_1)$ is now complex linear. As this is a continuous family of complex elliptic operators over $\B_P$, it has a complex determinant line bundle $K^{F_\bu}_P\ra\B_P$. 

A {\it spin structure\/} $\si_P^{F_\bu}$ on $(\B_P,F_\bu)$ is a choice of square root line bundle $(K^{F_\bu}_P)^{1/2}$ on $\B_P$, up to isomorphism, as in Definition~\ref{os3def2}. We are interested in questions such as: Do spin structures exist for $(\B_P,F_\bu)$? What is the family of spin structures? Can we choose a canonical spin structure for $(\B_P,F_\bu)$? Can we relate spin structures for different moduli spaces $\B_P,\B_Q,\B_R$ with $X,F_\bu$ fixed? Can we relate spin structures for $X$ with orientations for~$X\t\cS^1$?

To justify the name, note that if $Y$ is an (almost) complex manifold with canonical bundle $K_Y\ra Y$, then spin structures on $Y$ correspond to isomorphism classes of square roots of $K_Y$. We can think of $\B_P$ as an infinite-dimensional complex manifold, and $K^{F_\bu}_P\ra\B_P$ as its canonical bundle.

Under good conditions in complex gauge theory problems (for example, moduli spaces of Hermitian--Einstein connections on a compact K\"ahler manifold), we may form moduli spaces $\M_P^{\rm ga}$ which are (almost) complex manifolds such that $K^{F_\bu}_P\ra\B_P$ pulls back to the canonical bundle $K_{\M_P^{\rm ga}}$ under the inclusion $\M_P^{\rm ga}\hookra\B_P$. Thus, a spin structure on $(\B_P,F_\bu)$ induces a spin structure on the manifold $\M_P^{\rm ga}$ in the usual sense of differential geometry.

The authors' motivation for studying spin structures stems from the notion of {\it orientation data\/} for Calabi--Yau 3-folds, as introduced by Kontsevich and Soibelman \cite[\S 5.2]{KoSo1}, which are important in generalized Donaldson--Thomas theory (see for instance \cite{BBJ,BBDJS,BJM,Joyc2,KoSo2}). Let $X$ be a Calabi--Yau 3-fold over $\C$ and $\bs\M$ the derived moduli stack of (complexes of) coherent sheaves on $X$. Write $K_{\bs\M}\ra\bs\M$ for the determinant line bundle of the cotangent complex $\bL_{\bs\M}$. Then {\it orientation data\/} for $X$ is basically a choice of (isomorphism class of) square root $K_{\bs\M}^{1/2}$, satisfying certain compatibilities with exact sequences.

Orientation data is an algebro-geometric version of our notion of spin structure for $(\B_P,F_\bu)$. In the sequel \cite{JoUp2} we will apply the results of this paper to prove that canonical orientation data exists for any Calabi--Yau 3-fold over $\C$, solving a long-standing problem in Donaldson--Thomas theory.\smallskip

This will be based on the main theorem of this paper, Theorem \ref{os5thm3}, or more precisely, its consequence Theorem \ref{os5thm5}, obtained in conjunction with our previous result \cite{JoUp1} on $G_2$-instantons. This identifies spin structures on $X$ with orientations on $X\t\cS^1$.  Thus, if $P\ra X$ and $Q\ra X\t\cS^1$ are principal $G$-bundles with $Q\vert_{X\t\{1\}}\cong P$, we relate spin structures on $(\B_P,F_\bu)$ to orientations on $(\B_Q,E_\bu)$ for a certain class of operators $F_\bu$ on $X$ and $E_\bu$ on $X\t\cS^1$. In particular, if $(X,g)$ is a compact, oriented, spin Riemannian 6-manifold and $F_\bu$ is the positive Dirac operator on $X$, then we construct in Theorem~\ref{os5thm5} canonical spin structures on $(\B_P,F_\bu)$ for all principal $\U(m)$- and $\SU(m)$-bundles $P\ra X$.\smallskip

We begin in \S\ref{os2} with background on moduli spaces and orientations on them, taken mostly from \cite[\S 2]{JTU}. Section \ref{os3} introduces spin structures. Section \ref{os4} gives some elementary results, similar to those in \S\ref{os2} for orientations. In \S\ref{os5} we state our deeper theorems about spin structures and prove Theorem~\ref{os5thm5}. The proofs of Theorems \ref{os5thm1}, \ref{os5thm2}, and \ref{os5thm3} are postponed to \S\ref{os6}--\S\ref{os8}, followed by Appendix \ref{osA} on elliptic boundary problems, used in the proof of Theorem~\ref{os5thm2}.\medskip

\noindent{\it Acknowledgements.} This research was partly funded by a Simons Collaboration Grant on `Special Holonomy in Geometry, Analysis and Physics'.  The second author was partly funded by DFG grant UP 85/2-1 of the DFG priority program SPP 2026 `Geometry at Infinity', and by the Centre for Quantum Geometry of Moduli Spaces at Aarhus University. The authors would like to thank Yalong Cao, Jacob Gross, Bernhard Hanke, and Yuuji Tanaka for helpful conversations.

\section{Connection moduli spaces and orientations}
\label{os2}

We begin with background material on orientations on moduli spaces, following \cite{CGJ,JTU,JoUp1,Upme}. Sections \ref{os21}--\ref{os26} summarize Joyce--Tanaka--Upmeier~\cite[\S 1--\S 2]{JTU}. The reader interested only in spin structures may initially restrict to \S\ref{os21}--\S\ref{os22} and refer back to the notation set up in \S\ref{os23}--\S\ref{os26} in the discussion of the analogous properties for spin structures in Section~\ref{os4}.

\subsection{\texorpdfstring{Connection moduli spaces $\A_P,\B_P,\ovB_P$}{Connection moduli spaces 𝒜ᵨ,ℬᵨ}}
\label{os21}

\begin{dfn}
\label{os2def1}
Suppose we are given the following data:
\begin{itemize}
\setlength{\itemsep}{0pt}
\setlength{\parsep}{0pt}
\item[(a)] A compact, connected manifold $X$.
\item[(b)] A Lie group $G$, with $\dim G>0$, and centre $Z(G)\subseteq G$, and Lie algebra $\g$.
\item[(c)] A principal $G$-bundle $\pi:P\ra X$. We write $\Ad(P)\ra X$ for the vector bundle with fibre $\g$ defined by $\Ad(P)=(P\t\g)/G$, where $G$ acts on $P$ by the principal bundle action, and on $\g$ by the adjoint action.
\end{itemize}

Write $\A_P$ for the set of connections $\nabla_P$ on the principal bundle $P\ra X$. This is a real affine space modelled on the infinite-dimensional vector space $\Ga^\iy(\Ad(P)\ot T^*X)$, and we make $\A_P$ into a topological space using the $C^\iy$ topology on $\Ga^\iy(\Ad(P)\ot T^*X)$. Here if $E\ra X$ is a vector bundle then $\Ga^\iy(E)$ denotes the vector space of smooth sections of $E$. Note that $\A_P$ is contractible. 

Write $\G_P=\Aut(P)$ for the infinite-dimensional Lie group of $G$-equivariant diffeomorphisms $\ga:P\ra P$ with $\pi\ci\ga=\pi$. Then $\G_P$ acts on $\A_P$ by gauge transformations, and the action is continuous for the topology on~$\A_P$. 

There is an inclusion $Z(G)\hookra\G_P$ mapping $z\in Z(G)$ to the principal bundle action of $z$ on $P$. This maps $Z(G)$ into the centre $Z(\G_P)$ of $\G_P$, so we may take the quotient group $\G_P/Z(G)$. The action of $Z(G)\subset\G_P$ on $\A_P$ is trivial, so the $\G_P$-action on $\A_P$ descends to a $\G_P/Z(G)$-action. 

Each $\nabla_P\in\A_P$ has a (finite-dimensional) {\it stabilizer group\/} $\Stab_{\G_P}(\nabla_P)\subset\G_P$ under the $\G_P$-action on $\A_P$, with $Z(G)\subseteq\Stab_{\G_P}(\nabla_P)$. As $X$ is connected, $\Stab_{\G_P}(\nabla_P)$ is isomorphic to a closed Lie subgroup $H$ of $G$ with $Z(G)\subseteq H$. As in \cite[p.~133]{DoKr} we call $\nabla_P$ {\it irreducible\/} if $\Stab_{\G_P}(\nabla_P)=Z(G)$, and {\it reducible\/} otherwise. Write $\A_P^\irr,\A_P^\red$ for the subsets of irreducible and reducible connections in $\A_P$. Then $\A_P^\irr$ is open and dense in $\A_P$, and $\A_P^\red$ is closed and of infinite codimension in the infinite-dimensional affine space $\A_P$.

We write $\B_P=[\A_P/\G_P]$ for the moduli space of gauge equivalence classes of connections on $P$, considered as a {\it topological stack\/} in the sense of Metzler \cite{Metz} and Noohi \cite{Nooh1,Nooh2}. Write $\B_P^\irr=[\A_P^\irr/\G_P]$ for the substack $\B_P^\irr\subseteq\B_P$ of irreducible connections. 

\end{dfn}

\begin{rem}
The stacks $\B_P,\ab\B_P^\irr$ have variations $\ovB_P=[\A_P/(\G_P/Z(G))]$, $\ovB_P^\irr=[\A_P^\irr/(\G_P/Z(G))]$. Then $\ovB_P$ is a topological stack, but as $\G_P/Z(G)$ acts freely on $\A_P^\irr$, we may consider $\ovB_P^\irr$ as a topological space, an example of a topological stack. There are natural morphisms $\Pi_P:\B_P\ra\ovB_P$, $\Pi_P^\irr:\B^\irr_P\ra\ovB^\irr_P$.

The inclusions $\B^\irr_P\hookra\B_P$, $\ovB^\irr_P\hookra\ovB_P$ are weak homotopy equivalences of topological stacks in the sense of \cite{Nooh2}. Also $\Pi_P:\B_P\ra\ovB_P$ is a fibration with connected fibre $[*/Z(G)]$. Therefore, for the algebraic topological questions that concern us, working on $\ovB^\irr_P$ and on $\B_P$ are essentially equivalent, so we could just consider the topological space $\ovB^\irr_P$, and not worry about topological stacks at all, following most other authors in the area.

The main reason we do not do this in \cite{JTU} is that to relate orientations on different moduli spaces we consider direct sums of connections, which give a morphism $\Phi_{P,Q}:\B_P\t\B_Q\ra\B_{P\op Q}$, but this and similar morphisms do not make sense for the spaces $\B_P^\irr,\ovB_P,\ovB_P^\irr$, so we prefer to work with the~$\B_P$.	
\end{rem}

\subsection{Real elliptic operators and orientations}
\label{os22}

\begin{dfn}
\label{os2def2}
Let $X$ be a compact manifold. Suppose we are given real vector bundles $E_0,E_1\ra X$, of the same rank $r$, and a linear elliptic partial differential operator $D:\Ga^\iy(E_0)\ra\Ga^\iy(E_1)$, of degree $d$. As a shorthand we write $E_\bu=(E_0,E_1,D)$. With respect to connections $\nabla_{E_0}$ on $E_0\ot\bigot^iT^*X$ for $0\le i<d$, when $e\in\Ga^\iy(E_0)$ we may write
\e
D(e)=\ts\sum_{i=0}^d a_i\cdot \nabla_{E_0}^ie,
\label{os2eq1}
\e
where $a_i\in \Ga^\iy(E_0^*\ot E_1\ot S^iTX)$ for $i=0,\ldots,d$. The condition that $D$ is {\it elliptic\/} is that $a_d\vert_x\cdot\ot^d\xi:E_0\vert_x\ra E_1\vert_x$ is an isomorphism for all $x\in X$ and $0\ne\xi\in T_x^*X$, and the {\it symbol\/} $\si(D)$ of $D$ is defined using~$a_d$.

The {\it index\/} of $D$ is $\ind_\R D=\dim_\R\Ker D-\dim_\R\Coker D$. It can be computed using the Atiyah--Singer Index Theorem~\cite{AtSi1,AtSe,AtSi3,AtSi4,AtSi5}.

Now suppose we are given Euclidean metrics $h_{E_0},h_{E_1}$ on the fibres of $E_0,E_1$ and a volume form $\d V$ on $X$. Then there is a unique {\it adjoint operator\/} $D^*:\Ga^\iy(E_1)\ra\Ga^\iy(E_0)$, which is also a linear elliptic partial differential operator of degree $d$, satisfying for all $e_0\in\Ga^\iy(E_0)$, $e_1\in\Ga^\iy(E_1)$
\e
\int_X h_{E_1}(D(e_0),e_1)\d V=\int_X h_{E_0}(e_0,D^*(e_1))\d V.
\label{os2eq2}
\e
We call $D$ {\it self-adjoint\/} (or {\it anti-self-adjoint\/}), if $E_0=E_1$, $h_{E_0}=h_{E_1}$, and $D=D^*$ (or $D=-D^*$). For example, Dirac operators and Laplacians are self-adjoint. If $D$ is self-adjoint then $\Ker D=\Coker D$ and~$\ind_\R D=0$.
\end{dfn}

We define orientation bundles $O^{E_\bu}_P,\bar O^{E_\bu}_P$ on moduli spaces~$\B_P,\ovB_P$:

\begin{dfn} 
\label{os2def3}
Suppose $X,G,P,\A_P,\B_P,\ovB_P$ are as Definition \ref{os2def1}, and $E_\bu$ is a real elliptic operator on $X$ as in Definition \ref{os2def2}. Let $\nabla_P\in\A_P$. Then $\nabla_P$ induces a connection $\nabla_{\Ad(P)}$ on the vector bundle $\Ad(P)\ra X$. Thus we may form the twisted elliptic operator
\e
\begin{split}
D^{\nabla_{\Ad(P)}}&:\Ga^\iy(\Ad(P)\ot E_0)\longra\Ga^\iy(\Ad(P)\ot E_1),\\
D^{\nabla_{\Ad(P)}}&:e\longmapsto \ts\sum_{i=0}^d (\id_{\Ad(P)}\ot a_i)\cdot \nabla_{\Ad(P)\ot E_0}^ie,
\end{split}
\label{os2eq3}
\e
where $\nabla_{\Ad(P)\ot E_0}$ are the connections on $\Ad(P)\ot E_0\ot\bigot^iT^*X$ for $0\le i<d$ induced by $\nabla_{\Ad(P)}$ and~$\nabla_{E_0}$.

Since $D^{\nabla_{\Ad(P)}}$ is a linear elliptic operator on a compact manifold $X$, it has finite-dimensional kernel $\Ker(D^{\nabla_{\Ad(P)}})$ and cokernel $\Coker(D^{\nabla_{\Ad(P)}})$. The {\it determinant\/} $\det_\R(D^{\nabla_{\Ad(P)}})$ is the 1-dimensional real vector space
\e
\det_\R(D^{\nabla_{\Ad(P)}})=\det_\R\Ker(D^{\nabla_{\Ad(P)}})\ot\bigl(\det_\R\Coker(D^{\nabla_{\Ad(P)}})\bigr)^*,
\label{os2eq4}
\e
where if $V$ is a finite-dimensional real vector space then $\det_\R V=\La_\R^{\dim V}V$.

These operators $D^{\nabla_{\Ad(P)}}$ vary continuously with $\nabla_P\in\A_P$, so they form a family of elliptic operators over the base topological space $\A_P$. Thus as in Atiyah--Singer \cite{AtSi4}, Knudsen--Mumford~\cite{KnMu}, and Quillen~\cite{Quil} there is a natural real line bundle $\hat L{}^{E_\bu}_P\ra\A_P$ with fibre $\hat L{}^{E_\bu}_P\vert_{\nabla_P}=\det_\R(D^{\nabla_{\Ad(P)}})$ at each $\nabla_P\in\A_P$. It is equivariant under the actions of $\G_P$ and $\G_P/Z(G)$ on $\A_P$, and so pushes down to real line bundles $L^{E_\bu}_P\ra\B_P$, $\bar L^{E_\bu}_P\ra\ovB_P$ on the topological stacks $\B_P,\ovB_P$, with $L^{E_\bu}_P\cong\Pi_P^*(\bar L_P^{E_\bu})$. We call $L^{E_\bu}_P,\bar L^{E_\bu}_P$ the {\it determinant line bundles\/} of $\B_P,\ovB_P$. The restriction $\bar L^{E_\bu}_P\vert_{\ovB_P^\irr}$ is a topological real line bundle in the usual sense on the topological space~$\ovB_P^\irr$.

Define the {\it orientation bundle\/} $O^{E_\bu}_P$ of $\B_P$ by $O^{E_\bu}_P=(L^{E_\bu}_P\sm 0(\B_P))/(0,\iy)$. That is, we take the complement $L^{E_\bu}_P\sm 0(\B_P)$ of the zero section $0(\B_P)$ in $L^{E_\bu}_P$, and quotient by $(0,\iy)$ acting on the fibres of $L^{E_\bu}_P\sm 0(\B_P)\ra\B_P$ by multiplication. Then $L^{E_\bu}_P\ra\B_P$ descends to $\pi:O^{E_\bu}_P\ra\B_P$, which is a bundle with fibre $(\R\sm\{0\})/(0,\iy)\cong\{1,-1\}=\Z_2$, since $L^{E_\bu}_P\ra\B_P$ is a fibration with fibre $\R$. That is, $\pi:O^{E_\bu}_P\ra\B_P$ is a {\it principal\/ $\Z_2$-bundle}, in the sense of topological stacks.

Similarly we define a $\G_P$-equivariant principal $\Z_2$-bundle $\hat O{}^{E_\bu}_P\ra \A_P$ and a principal $\Z_2$-bundle $\bar\pi:\bar O^{E_\bu}_P\ra\ovB_P$ from $\bar L^{E_\bu}_P$, and as $L^{E_\bu}_P\cong\Pi_P^*(\bar L_P^{E_\bu})$ we have a canonical isomorphism $O^{E_\bu}_P\cong\Pi_P^*(\bar O_P^{E_\bu})$. The fibres of $O^{E_\bu}_P\ra\B_P$, $\bar O^{E_\bu}_P\ra\ovB_P$ are orientations on the real line fibres of $L^{E_\bu}_P\ra\B_P$, $\bar L^{E_\bu}_P\ra\ovB_P$. The restriction $\bar O^{E_\bu}_P\vert_{\ovB^\irr_P}$ is a principal $\Z_2$-bundle on the topological space $\ovB^\irr_P$, in the usual sense.

We say that $\B_P$ is {\it orientable\/} if $O^{E_\bu}_P$ is isomorphic to the trivial principal $\Z_2$-bundle $\B_P\t\Z_2\ra\B_P$. An {\it orientation\/} $\om$ on $\B_P$ is an isomorphism $\om:O^{E_\bu}_P\,{\buildrel\cong\over\longra}\,\B_P\t\Z_2$ of principal $\Z_2$-bundles. We make the same definitions for $\ovB_P$ and $\bar O^{E_\bu}_P$. Since $\Pi_P:\B_P\ra\ovB_P$ is a fibration with fibre $[*/Z(G)]$, which is connected and simply-connected, and $O^{E_\bu}_P\cong\Pi_P^*(\bar O_P^{E_\bu})$, we see that $\B_P$ is orientable if and only if $\ovB_P$ is, and orientations of $\B_P$ and $\ovB_P$ correspond. As $\B_P$ is connected, if $\B_P$ is orientable it has exactly two orientations.

We also define the {\it normalized orientation bundle}, or {\it n-orientation bundle}, a principal $\Z_2$-bundle $\check O_P^{E_\bu}\ra\B_P$, by
\e
\check O_P^{E_\bu}=O_P^{E_\bu}\ot_{\Z_2}O_{X\t G}^{E_\bu}\vert_{[\nabla^0]}.
\label{os2eq5}
\e
That is, we tensor the $O_P^{E_\bu}$ with the orientation torsor $O_{X\t G}^{E_\bu}\vert_{[\nabla^0]}$ of the trivial principal $G$-bundle $X\t G\ra X$ at the trivial connection $\nabla^0$. A {\it normalized orientation}, or {\it n-orientation}, of $\B_P$ is an isomorphism $\check\om:\check O^{E_\bu}_P\,{\buildrel\cong\over\longra}\,\B_P\t\Z_2$. There is a natural n-orientation of $\B_{X\t G}$ at~$[\nabla^0]$.

N-orientations and orientations are equivalent once we choose an isomorphism $O_{X\t G}^{E_\bu}\vert_{[\nabla^0]}\cong\Z_2$. N-orientations behave nicely under the Excision Theorem in Upmeier \cite[Th.~2.13]{Upme}, and in examples there are often canonical choices of n-orientations, so we use them in preference to orientations.
\end{dfn}

\begin{rem}{\bf(i)} Up to continuous isotopy, and hence up to isomorphism, $L^{E_\bu}_P,O^{E_\bu}_P,\check O_P^{E_\bu}$ in Definition \ref{os2def3} depend on the elliptic operator $D:\Ga^\iy(E_0)\ra\Ga^\iy(E_1)$ up to continuous deformation amongst elliptic operators, and thus only on the {\it symbol\/} $\si(D)$ of $D$ (essentially, the highest order coefficients $a_d$ in \eq{os2eq1}), up to deformation.
\smallskip

\noindent{\bf(ii)} For orienting moduli spaces of `instantons' in gauge theory, as in \S\ref{os28}, we usually start not with an elliptic operator on $X$, but with an {\it elliptic complex\/}
\e
\smash{\xymatrix@C=28pt{ 0 \ar[r] & \Ga^\iy(E_0) \ar[r]^{D_0} & \Ga^\iy(E_1) \ar[r]^(0.55){D_1} & \cdots \ar[r]^(0.4){D_{k-1}} & \Ga^\iy(E_k) \ar[r] & 0. }}
\label{os2eq6}
\e
If $k>1$ and $\nabla_P$ is an arbitrary connection on a principal $G$-bundle $P\ra X$ then twisting \eq{os2eq6} by $(\Ad(P),\nabla_{\Ad(P)})$ as in \eq{os2eq3} may not yield a complex (that is, we may have $D^{\nabla_{\Ad(P)}}_{i+1}\ci D^{\nabla_{\Ad(P)}}_i\ne 0$), so the definition of $\det_\R(D_\bu^{\nabla_{\Ad(P)}})$ does not work, though it does work if $\nabla_P$ satisfies the appropriate instanton-type curvature condition. To get round this, we choose metrics on the $E_i$ and a volume form $\d V$ on X, so that we can take adjoints $D_i^*$, and replace \eq{os2eq6} by the elliptic operator
\e
\smash{\xymatrix@C=90pt{ \Ga^\iy\bigl(\bigop_{0\le i\le k/2}E_{2i}\bigr) \ar[r]^(0.48){\sum_i(D_{2i}+D_{2i-1}^*)} & \Ga^\iy\bigl(\bigop_{0\le i< k/2}E_{2i+1}\bigr), }}
\label{os2eq7}
\e
and then Definition \ref{os2def3} works with \eq{os2eq7} in place of~$E_\bu$.
\smallskip

\noindent{\bf(iii)} If $D$ is self-adjoint then $D^{\nabla_{\Ad(P)}}$ is too, so $\Ker(D^{\nabla_{\Ad(P)}})=\Coker(D^{\nabla_{\Ad(P)}})$, and in \eq{os2eq4} we have a canonical isomorphism $\det_\R(D^{\nabla_{\Ad(P)}})\cong\R$. Surprisingly, this does {\it not\/} imply that $L^{E_\bu}_P,O^{E_\bu}_P,\check O_P^{E_\bu}$ are trivial --- they may not be --- since these isomorphisms $\det_\R(D^{\nabla_{\Ad(P)}})\cong\R$ do not vary continuously with~$\nabla_P$.
\label{os2rem1}	
\end{rem}

\subsection{\texorpdfstring{Natural n-orientations when $G$ is abelian}{Natural n-orientations when G is abelian}}
\label{os23}

In the situation of Definition \ref{os2def3}, suppose the Lie group $G$ is abelian, for example $G=\U(1)$. Then as in \cite[\S 2.2.3]{JTU} the line bundles $\hat L{}^{E_\bu}_P\ra\A_P$, $L^{E_\bu}_P\ra\B_P$, $\bar L^{E_\bu}_P\ra\ovB_P$ are all canonically trivial, with fibre 
\begin{equation*}
(\det_\R D)^{\ot^{\dim\g}}\ot(\La_\R^{\dim\g}\g)^{\ot^{\ind_\R D}}.
\end{equation*}
This implies that the n-orientation bundle $\check O^{E_\bu}_P\ra\B_P$ is canonically trivial.

\subsection{N-orientations on products of moduli spaces}
\label{os24}

This section comes from \cite[\S 2.2.7]{JTU}. Let $X$ and $E_\bu$ be fixed, and suppose $G,H$ are Lie groups, and $P\ra X$, $Q\ra X$ are principal $G$- and $H$-bundles. Then $P\t_XQ$ is a principal $G\t H$ bundle over $X$. There is a natural 1-1 correspondence between pairs $(\nabla_P,\nabla_Q)$ of connections $\nabla_P,\nabla_Q$ on $P,Q$, and connections $\nabla_{P\t_XQ}$ on $P\t_XQ$. This induces an isomorphism of topological stacks $\La_{P,Q}:\B_P\t\B_Q\ra\B_{P\t_XQ}$.

For $(\nabla_P,\nabla_Q)$ and $\nabla_{P\t_XQ}$ as above, there are also natural isomorphisms
\begin{align*}
\Ker(D^{\nabla_{\Ad(P)}})\op \Ker(D^{\nabla_{\Ad(Q)}})&\cong\Ker(D^{\nabla_{\Ad(P\t_XQ)}}),\\
\Coker(D^{\nabla_{\Ad(P)}})\op \Coker(D^{\nabla_{\Ad(Q)}})&\cong\Coker(D^{\nabla_{\Ad(P\t_XQ)}}).\end{align*}
With the orientation conventions of \cite[\S 3]{Upme}, these induce a natural isomorphism
\e
\det(D^{\nabla_{\Ad(P)}})\ot \det(D^{\nabla_{\Ad(Q)}})\cong\det(D^{\nabla_{\Ad(P\t_XQ)}}),
\label{os2eq8}
\e
which is the fibre at $(\nabla_P,\nabla_Q)$ of an isomorphism of line bundles on~$\B_P\t\B_Q$:
\e
L_P^{E_\bu}\bt L_Q^{E_\bu}\cong\La_{P,Q}^*(L_{P\t_XQ}^{E_\bu}).
\label{os2eq9}
\e
This induces an isomorphism of n-orientation bundles on~$\B_P\t\B_Q$:
\e
\check\la^{E_\bu}_{P,Q}:\check O_P^{E_\bu}\bt\check O_Q^{E_\bu}\longra
\La_{P,Q}^*(\check O_{P\t_XQ}^{E_\bu}).
\label{os2eq10}
\e
Hence n-orientations on $\B_P,\B_Q$ induce an n-orientation on $\B_{P\t_XQ}$.

\begin{rem}
\label{os2rem2}
Equation \eq{os2eq8} is defined using an orientation convention as in \cite[\S 3.1--\S 3.2]{Upme} and depends on the order of $P$ and $Q$. Exchanging $P$ and $Q$ modifies \eq{os2eq8}--\eq{os2eq10} by signs as in \cite[(2.3)]{JTU}. Write $\ind^{E_\bu}_P=\ind_\R D^{\nabla_{\Ad(P)}}$. Under the natural isomorphisms $\B_P\t\B_Q\cong\B_Q\t\B_P$, $\check O_P^{E_\bu}\bt\check O_Q^{E_\bu}\cong \check O_Q^{E_\bu}\bt\check O_P^{E_\bu}$,
\e
\check\la_{P,Q}^{E_\bu}=(-1)^{(\ind^{E_\bu}_P+\ind^{E_\bu}_{X\t G})\cdot(\ind^{E_\bu}_Q+\ind^{E_\bu}_{X\t H})}\cdot\check\la_{Q,P}^{E_\bu}
\label{os2eq11}
\e
holds for the n-orientation bundles. On the other hand, given a third principal $K$-bundle $R\to X$, \eq{os2eq8}--\eq{os2eq10} are associative without any sign corrections.
\end{rem}

\subsection{\texorpdfstring{Relating moduli spaces for discrete quotients $G\twoheadrightarrow H$}{Relating moduli spaces for discrete quotients G→H}}
\label{os25}

This section comes from  \cite[\S 2.2.8]{JTU}. Suppose $G$ is a Lie group, $K\subset G$ a discrete normal subgroup, and set $H=G/K$ for the quotient Lie group. Let $X,E_\bu$ be fixed. If $P\ra X$ is a principal $G$-bundle, then $Q:=P/K$ is a principal $H$-bundle over $X$. Each $G$-connection $\nabla_P$ on $P$ induces a natural $H$-connection $\nabla_Q$ on $Q$, and the map $\nabla_P\mapsto\nabla_Q$ induces a natural morphism $\De_{P,Q}:\B_P\ra\B_Q$ of topological stacks, which is a bundle with fibre $[*/K]$. If $\nabla_P,\nabla_Q$ are as above then the natural isomorphism $\g\cong\h$ induces an isomorphism $\Ad(P)\cong\Ad(Q)$ of vector bundles on $X$, which identifies the connections $\nabla_{\Ad(P)},\nabla_{\Ad(Q)}$. Hence the twisted elliptic operators $D^{\nabla_{\Ad(P)}},D^{\nabla_{\Ad(Q)}}$ are naturally isomorphic, and so are their determinants \eq{os3eq2}. This gives a canonical isomorphism
\e
L_P^{E_\bu}\cong\De_{P,Q}^*(L_Q^{E_\bu}),
\label{os2eq13}
\e
which induces an isomorphism of n-orientation bundles
\e
\check\de_{P,Q}^{E_\bu}:\check O_P^{E_\bu}\,{\buildrel\cong\over\longra}\,\De_{P,Q}^*(\check O_Q^{E_\bu}).
\label{os2eq14}
\e
Hence n-orientations on $\B_Q$ pull back to n-orientations on~$\B_P$.
 
\begin{ex}
\label{os2ex1}
Take $G=\SU(m)\t\U(1)$, and define $K\subset G$ by
\begin{equation*}
K=\bigl\{(e^{2\pi ik/m}\Id_m,e^{-2\pi ik/m}):k=1,\ldots,m\bigr\}\cong\Z_m.
\end{equation*}
Then $K$ lies in the centre $Z(G)$, so is normal in $G$, and $H=G/K\cong\U(m)$. 

For fixed $X,E_\bu$, let $P\ra X$ be a principal $\SU(m)$-bundle, and $P'=X\t\U(1)\ra X$ be the trivial principal $\U(1)$-bundle. Write $P''=P\t\U(1)=P\t_XP'\ra X$ for the associated principal $\SU(m)\t\U(1)$-bundle, with trivial connection $\nabla^0$, and define $Q=(P\t\U(1))/\Z_m\cong(P\t\U(m))/\SU(m)$ to be the quotient principal $\U(m)$-bundle. Define a morphism $\Ka_{P,Q}:\B_P\ra\B_Q$ by the commutative diagram
\e
\begin{gathered}
\xymatrix@!0@C=141pt@R=35pt{ 
*+[r]{\B_P} \ar[d]^{(\id_{\B_P},[\nabla^0])} \ar@/^1pc/[drr]^(0.6){\Ka_{P,Q}}
\\
*+[r]{\B_P\t\B_{P'}} \ar[r]^(0.45){\La_{P,P'}} & \B_{P\t\U(1)}=\B_{P''} \ar[r]^(0.45){\De_{P'',Q}} & *+[l]{\B_Q,} }
\end{gathered}
\label{os2eq15}
\e
where $\La_{P,P'}$ is as in \S\ref{os24}. Define an isomorphism of n-orientation bundles $\check\ka_{P,Q}^{E_\bu}:\check O_P^{E_\bu}\ra \Ka_{P,Q}^*(\check O_Q^{E_\bu})$ by the commutative diagram
\e
\begin{gathered}
\!\!\!\!\xymatrix@!0@C=141pt@R=45pt{ 
*+[r]{\check O_P^{E_\bu}} \ar[d]^(0.45){\id_{\check O_P^{E_\bu}}\ot\io} \ar[rr]_{\check\ka_{P,Q}^{E_\bu}} && *+[l]{\Ka_{P,Q}^*(\check O_Q^{E_\bu})} \ar@{=}[d]
\\
*+[r]{\begin{subarray}{l}\ts (\id_{\B_P},[\nabla^0])^* \\ \ts (\check O_P^{E_\bu}\bt \check O_{P'}^{E_\bu})\end{subarray}} \ar[r]^{\begin{subarray}{l}(\id_{\B_P},[\nabla^0])^* \\ \;\>(\check\la_{P,P'}^{E_\bu}) \end{subarray}} &
*+[l]{\begin{subarray}{l}\ts (\id_{\B_P},[\nabla^0])^* \\ \ts \ci\La_{P,P'}^*(\check O_{P''}^{E_\bu})\end{subarray}}  \ar[r]^(0.3){\begin{subarray}{l} (\id_{\B_P},[\nabla^0])^*\ci{} \\ \La_{P,P'}^*(\de_{P'',Q}^{E_\bu})\end{subarray}} & *+[l]{\begin{subarray}{l}\ts \qquad\,\,\,\, (\id_{\B_P},[\nabla^0])^*\ci{} \\ \ts (\De_{P'',Q}\!\ci\!\La_{P,P'})^*(\check O_Q^{E_\bu}).\end{subarray}\;\>}\!\!\!\!
}
\end{gathered}
\label{os2eq16}
\e
Here $\io:\Z_2\ra\check O_{P'}^{E_\bu}\vert_{[\nabla^0]}$ is the natural isomorphism, and $\check\la_{P,P'}^{E_\bu},\de_{P'',Q}^{E_\bu}$ are as in \eq{os2eq10}, \eq{os2eq14}. Hence n-orientations on $\B_Q$ pull back to n-orientations on~$\B_P$.
\end{ex}

\subsection{\texorpdfstring{Relating moduli spaces for Lie subgroups $G\subset H$}{Relating moduli spaces for Lie subgroups G⊂H}}
\label{os26}

This section comes from \cite[\S 2.2.9]{JTU}. Let $X,E_\bu$ be fixed, and let $H$ be a Lie group and $G\subset H$ a Lie subgroup, with Lie algebras $\g\subset\h$. If $P\ra X$ is a principal $G$-bundle, then $Q:=(P\t H)/G$ is a principal $H$-bundle over $X$. Each $G$-connection $\nabla_P$ on $P$ induces a natural $H$-connection $\nabla_Q$ on $Q$, and the map $\nabla_P\mapsto\nabla_Q$ induces a natural morphism $\Xi_{P,Q}:\B_P\ra\B_Q$ of topological stacks. Thus, we can try to compare the line bundles $L_P^{E_\bu},\Xi_{P,Q}^*(L_Q^{E_\bu})$ on~$\B_P$.

Write $\m=\h/\g$, and $\rho_\R:G\ra\Aut(\m)$ for the real representation induced by the adjoint representation of $H\supset G$. Then we have an exact sequence 
\e
\xymatrix@C=30pt{ 0 \ar[r] & \Ad(P) \ar[r] & \Ad(Q) \ar[r] & \rho_\R(P)=(P\t\m)/G \ar[r] & 0 }	
\label{os2eq17}
\e
of vector bundles on $X$, induced by $0\ra\g\ra\h\ra\m\ra 0$. If $\nabla_P,\nabla_Q$ are as above, we have connections $\nabla_{\Ad(P)},\nabla_{\Ad(Q)},\nabla_{\rho_\R(P)}$ on $\Ad(P),\Ad(Q),\rho_\R(P)$ compatible with \eq{os2eq17}. Twisting $E_\bu$ by $\Ad(P),\ab\Ad(Q),\ab\rho_\R(P)$ and their connections and taking determinants, as in Remark \ref{os2rem2} we define an isomorphism
\begin{equation*}
\det_\R(D^{\nabla_{\Ad(P)}})\ot_\R \det_\R(D^{\nabla_{\rho_\R(P)}})\cong \det_\R(D^{\nabla_{\Ad(Q)}}),
\end{equation*}
which is the fibre at $\nabla_P$ of an isomorphism of line bundles on $\B_P$,
\e
L_P^{E_\bu}\ot_\R L_{P,\rho_\R}^{E_\bu}\cong\Xi_{P,Q}^*(L_Q^{E_\bu}),
\label{os2eq18}
\e
where $L_{P,\rho_\R}^{E_\bu}\ra\B_P$ is the determinant line bundle associated to the family of real elliptic operators $\nabla_P\mapsto D^{\nabla_{\rho_\R(P)}}$ on $\B_P$.

Now suppose that we can give $\m$ the structure of a {\it complex\/} vector space $\m^\C$, such that $\rho_\C=\rho_\R:G\ra\Aut(\m^\C)$ is complex linear. This happens if $H/G$ has an (almost) complex structure homogeneous under $H$. Then we can regard $\nabla_P\mapsto D^{\nabla_{\rho_\R(P)}}=D^{\nabla_{\rho_\C(P)}}$ as a family of {\it complex\/} elliptic operators over $\B_P$, so they have a {\it complex\/} determinant line bundle $K_{P,\rho_\C}^{E_\bu}\ra\B_P$. There is a natural isomorphism $L_{P,\rho_\R}^{E_\bu}\cong\La_\R^2K_{P,\rho_\C}^{E_\bu}$. As $K_{P,\rho_\C}^{E_\bu}$ is complex, $L_{P,\rho_\R}^{E_\bu}$ has a natural orientation. Thus, taking (n-)orientation bundles in \eq{os2eq18}, the contribution from $L_{P,\rho_\R}^{E_\bu}$ is trivial, and we obtain an isomorphism of n-orientation bundles
\e
\check\xi_{P,Q}^{E_\bu}:\check O_P^{E_\bu}\longra\Xi_{P,Q}^*(\check O_Q^{E_\bu}).
\label{os2eq19}
\e
Hence n-orientations on $\B_Q$ induce n-orientations on $\B_P$.

The next example is a kind of converse to Example~\ref{os2ex1}.

\begin{ex} Define an inclusion $\U(m)\hookra\SU(m+1)$ by mapping
\begin{equation*}
A\longmapsto \begin{pmatrix} A & \begin{matrix} 0 \\ \vdots \\ 0 \end{matrix} \\ \begin{matrix} 0 & \cdots & 0 \end{matrix} & (\det A)^{-1} \end{pmatrix},\qquad A\in\U(m).	
\end{equation*}
There is an isomorphism $\m=\su(m+1)/\u(m)\cong\C^m$, such that $A\in\U(m)$ acts on $\m\cong\C^m$ by $A:\bs x\mapsto \det A\cdot A\bs x$, which is complex linear on~$\m$. 

For fixed $X,E_\bu$, let $Q\ra X$ be a principal $\U(m)$-bundle, and $R=(Q\t\SU(m+1))/\U(m)$ its $\SU(m+1)$-bundle. Then n-orientations on $\B_R$ pull back to n-orientations on $\B_Q$.
\label{os2ex2}	
\end{ex}

\begin{ex} We have an inclusion $G=\U(m_1)\t\U(m_2)\subset\U(m_1+m_2)=H$ for $m_1,m_2\ge 1$, with $\u(m_1+m_2)/(\u(m_1)\op\u(m_2))=\m\cong\C^{m_1}\ot_\C\ov{\C^{m_2}}$, where $G=\U(m_1)\t\U(m_2)$ acts on $\C^{m_1}\ot_\C\ov{\C^{m_2}}$ via the usual representations of $\U(m_1),\U(m_2)$ on $\C^{m_1},\C^{m_2}$, with $\ov{\C^{m_2}}$ the complex conjugate of $\C^{m_2}$, so the representation $\rho_\R=\rho_\C$ is complex linear.

For fixed $X,E_\bu$, suppose $P_1\ra X$, $P_2\ra X$ are principal $\U(m_1)$- and $\U(m_2)$-bundles. Define a principal $\U(m_1+m_2)$-bundle $P_1\op P_2\ra X$ by 
\e
P_1\op P_2=(P_1\t_XP_2\t\U(m_1+m_2))/(\U(m_1)\t\U(m_2)).
\label{os2eq20}
\e
Then combining the material of \S\ref{os24} for the product of $\U(m_1),\U(m_2)$ with the above, we have a morphism
\e
\Phi_{P_1,P_2}:=\Xi_{P_1\t_XP_2,P_1\op P_2}\ci\La_{P_1,P_2}:\B_{P_1}\t\B_{P_2}\longra\B_{P_1\op P_2},
\label{os2eq21}
\e
and a natural isomorphism of n-orientation bundles on~$\B_{P_1}\t\B_{P_2}$:
\e
\begin{split}
&\check\phi^{E_\bu}_{P_1,P_2}=\La_{P_1,P_2}^*(\check\xi_{P_1\t_XP_2,P_1\op P_2}^{E_\bu})\ci\check\la^{E_\bu}_{P_1,P_2}:\\
&\qquad \check O_{P_1}^{E_\bu}\bt_{\Z_2}\check O_{P_2}^{E_\bu}\,{\buildrel\cong\over\longra}\,
\Phi_{P_1,P_2}^*(\check O_{P_1\op P_2}^{E_\bu})
\end{split}
\label{os2eq22}
\e
Thus n-orientations for $\B_{P_1}$ and $\B_{P_2}$ determine an n-orientation for $\B_{P_1\op P_2}$.
\label{os2ex3}	
\end{ex}

\subsection{Example theorems on orientability and orientations}
\label{os27}

The next theorem summarizes results from {\bf(a)\rm,\bf(b)} Joyce--Tanaka--Upmeier \cite[\S 4]{JTU}, {\bf(a)} Taubes~\cite[\S 2]{Taub}, {\bf(b)} Donaldson \cite[II.4]{Dona1}, \cite[\S 3(d)]{Dona2}, \cite[\S 5.4]{DoKr}, {\bf(c)} Walpuski \cite[Prop.~6.3]{Walp}, Joyce--Upmeier \cite[Th.~1.2]{JoUp1}, and {\bf(d)} Cao--Gross--Joyce \cite[Th.~1.11]{CGJ}.

\begin{thm}
\label{os2thm1}
Let\/ $X$ be a compact, oriented\/ $n$-manifold, supposed spin for {\bf(c)\rm,\bf(d)\rm,}\/ $E_\bu$ be a real elliptic operator on\/ $X$,\/ $G$ a Lie group, and\/ $P\ra X$ a principal\/ $G$-bundle. Then\/ $O_P^{E_\bu}\ra\B_P$ is \textup(n-\textup)orientable in the following cases:
\begin{itemize}
\setlength{\itemsep}{0pt}
\setlength{\parsep}{0pt}
\item[{\bf(a)}] $n=2$ or\/ $3$ and\/ $E_\bu$ is\/ $\d+\d^*:\Ga^\iy(\La^{\rm even}T^*X)\ra \Ga^\iy(\La^{\rm odd}T^*X);$
\item[{\bf(b)}] $n=4$ and\/ $E_\bu$ is\/ $\d+\d_+^*:\Ga^\iy(\La^0T^*X\op\La^2_+T^*X)\ra \Ga^\iy(\La^1T^*X);$
\item[{\bf(c)}] $n=7,$ $E_\bu$ is the Dirac operator\/ $\slashed{D},$ and\/ $G=\SU(m)$ or\/ $\U(m);$ and
\item[{\bf(d)}] $n=8,$ $E_\bu$ is the positive Dirac operator\/ $\slashed{D}_+,$ and\/ $G=\SU(m)$ or\/~$\U(m)$.
\end{itemize}

As in {\rm\S\ref{os28},} these are relevant to orienting gauge theory moduli spaces of\/ {\bf(a)} flat connections on\/ $2$- and\/ $3$-manifolds, {\bf(b)} anti-self-dual instantons on\/ $4$\nobreakdash-manifolds, {\bf(c)}\/ $G_2$-instantons on Riemannian\/ $7$-manifolds with holonomy\/ $G_2,$ and\/ {\bf(d)}\/ $\Spin(7)$-instantons on Riemannian\/ $8$-manifolds with holonomy\/ $\SU(4)$ or\/ $\Spin(7)$. In cases {\bf(a)\rm--\bf(c)\rm,} after choosing a small amount of extra data\/ $\scr D$ on\/ $X,$ we can construct canonical \textup(n-\textup)orientations on\/ $\B_P$ for all such\/~$P$.
\end{thm}

These are proved using a wide variety of techniques from algebraic and differential topology and index theory. The difficulty of the proofs generally increases with the dimension $n$. The extra data $\scr D$ needed to define (n-)orientations can be subtle, for example in case (c) it includes a {\it flag structure}, a curious algebro-topological structure on 7-manifolds introduced in \cite[\S 3.1]{Joyc3}. See \S\ref{os53} below.

\subsection{Applications to orienting gauge theory moduli spaces}
\label{os28}

In gauge theory one studies moduli spaces $\M_P^{\rm ga}$ of (irreducible) connections $\nabla_P$ on a principal bundle $P\ra X$ (perhaps plus some extra data, such as a Higgs field) satisfying a curvature condition. Under suitable genericity conditions, these moduli spaces $\M_P^{\rm ga}$ will be smooth manifolds, and the ideas of \cite{JTU} can often be used to prove $\M_P^{\rm ga}$ is orientable, and construct a canonical orientation on $\M_P^{\rm ga}$. These orientations are important in defining enumerative invariants such as Casson invariants, Donaldson invariants, and Seiberg--Witten invariants. 
We illustrate this with the example of instantons on 4-manifolds,~\cite{DoKr}:

\begin{ex}
\label{os2ex4}	
Let $(X,g)$ be a compact, oriented Riemannian 4-manifold, and $G$ a Lie group (e.g.\ $G=\SU(2)$), and $P\ra X$ a principal $G$-bundle. For each connection $\nabla_P$ on $P$, the curvature $F^{\nabla_P}$ is a section of $\Ad(P)\ot\La^2T^*X$. We have $\La^2T^*X=\La^2_+T^*X\op\La^2_-T^*X$, where $\La^2_\pm T^*X$ are the subbundles of 2-forms $\al$ on $X$ with $*\al=\pm\al$. Thus $F^{\nabla_P}=F^{\nabla_P}_+\op F^{\nabla_P}_-$, with $F^{\nabla_P}_\pm$ the component in $\Ad(P)\ot\La^2_\pm T^*X$. We call $(P,\nabla_P)$ an ({\it anti-self-dual\/}) {\it instanton\/} if $F^{\nabla_P}_+=0$.

Write $\M_P^{\rm asd}$ for the moduli space of gauge isomorphism classes $[\nabla_P]$ of irreducible instanton connections $\nabla_P$ on $P$, modulo $\G_P/Z(G)$. The deformation theory of $[\nabla_P]$ in $\M_P^{\rm asd}$ is governed by the Atiyah--Hitchin--Singer complex \cite{AHS}: 
\e
\begin{gathered} 
\xymatrix@C=6pt@R=4pt{ 0 \ar[rr] && 
\Ga^{\iy} ( \Ad(P) \ot \La^0T^*X ) \ar[rrr]^{\d^{\nabla_P}} &&& 
\Ga^{\iy} ( \Ad(P) \ot \La^1T^*X  ) \\
&& {\qquad\qquad\qquad} \ar[rrr]^{\d^{\nabla_P}_+} &&&
\Ga^{\iy} ( \Ad(P) \ot \La^2_+T^*X  ) \ar[rr] && 0, } 
\end{gathered}
\label{os2eq23}
\e
where $\d^{\nabla_P}_+\ci\d^{\nabla_P}=0$ as $F^{\nabla_P}_+=0$. Write $\cH^0,\cH^1,\cH^2_+$ for the cohomology groups of \eq{os2eq23}. Then $\cH^0$ is the Lie algebra of $\Aut(\nabla_P)$, so $\cH^0=Z(\g)$, the Lie algebra of the centre $Z(G)$ of $G$, as $\nabla_P$ is irreducible. Also $\cH^1$ is the Zariski tangent space of $\M_P^{\rm asd}$ at $[\nabla_P]$, and $\cH^2_+$ is the obstruction space. If $g$ is generic then as in \cite[\S 4.3]{DoKr}, for non-flat connections $\cH^2_+=0$ for all $[\nabla_P]\in\M_P^{\rm asd}$, and $\M_P^{\rm asd}$ is a smooth manifold, with tangent space $T_{[\nabla_P]}\M_P^{\rm asd}=\cH^1$. Note that $\M_P^{\rm asd}\subset\ovB_P$ is a subspace of the topological stack $\ovB_P$ from Definition~\ref{os2def1}.

Take $E_\bu$ to be the elliptic operator on $X$
\e
D=\d+\d_+^*:\Ga^\iy(\La^0T^*X\op\La^2_+T^*X)\longra\Ga^\iy(\La^1T^*X).	
\label{os2eq24}
\e
Turning the complex \eq{os2eq23} into a single elliptic operator as in Remark \ref{os2rem1}(ii) yields the twisted operator $D^{\nabla_{\Ad(P)}}$ from \eq{os2eq3}. Hence the line bundle $\bar L^{E_\bu}_P\ra\ovB_P$ in Definition \ref{os2def3} has fibre at $[\nabla_P]$ the determinant line of \eq{os2eq23}, which (after choosing an isomorphism $\det_\R Z(\g)\cong\R$) is $\det_\R(\cH^1)^*=\det_\R T^*_{[\nabla_P]}\M_P^{\rm asd}$. It follows that $\bar O_P^{E_\bu}\vert_{\M_P^{\rm asd}}$ is the orientation bundle of the manifold $\M_P^{\rm asd}$.

Thus an orientation on $\ovB_P$ in Definition \ref{os2def3} restricts to an orientation on the manifold $\M_P^{\rm asd}$ in the usual sense of differential geometry. As in \cite{CGJ,JTU,JoUp1}, we can often use differential and algebraic topology techniques to construct orientations on $\ovB_P$, and hence induce orientations on $\M_P^{\rm asd}$. This is a very useful way of defining orientations on $\M_P^{\rm asd}$, first used by Donaldson~\cite{Dona1,Dona2,DoKr}.
\end{ex}

There are several other important classes of gauge-theoretic moduli spaces $\M_P^{\rm ga}$ which have elliptic deformation theory, and so are generically smooth manifolds, for which orientations can be defined by pullback from $\ovB_P$. Examples are given in \cite[\S 4]{JTU}, \cite[Cor.~1.4]{JoUp1}, and~\cite[\S 1.3]{CGJ}. 

\begin{rem} If we omit the genericness/transversality conditions, gauge theory moduli spaces $\M_P^{\rm ga}$ are generally not smooth manifolds. However, as long as their deformation theory is given by an elliptic complex similar to \eq{os2eq23} whose cohomology is constant except at the second and third terms, $\M_P^{\rm ga}$ will still be a {\it derived smooth manifold\/} ({\it d-manifold}, or {\it m-Kuranishi space\/}) in the sense of Joyce \cite{Joyc1,Joyc4,Joyc5,Joyc6}. Orientations for derived manifolds are defined and well behaved, and we can define orientations on $\M_P^{\rm ga}$ by pullback of orientations on $\ovB_P$ exactly as in the case when $\M_P^{\rm ga}$ is a manifold.
\label{os2rem3}
\end{rem}

\subsection{Applications in complex (derived) algebraic geometry}
\label{os29}

Cao, Gross, and Joyce \cite{CGJ} apply orientations on moduli spaces $\B_P$ to solve an orientation problem in complex algebraic geometry. Suppose $X$ is a Calabi--Yau $m$-fold over $\C$, and write $\bs\M$ for the derived moduli $\C$-stack of complexes of coherent sheaves on $X$, in the sense of To\"en, Vaqui\'e, and Vezzosi \cite{ToVa,ToVe1,ToVe2}, and $\bs\M_{\rm vect}\subset\bs\M_{\rm coh}\subset\bs\M$ for the open substacks of algebraic vector bundles, and coherent sheaves. Then $\bs\M$ has a $(2-m)$-shifted symplectic structure $\om$ in the sense of Pantev, To\"en, Vaqui\'e, and Vezzosi~\cite{PTVV}. 

Borisov and Joyce \cite[\S 2.4]{BoJo} define a notion of orientation for such $(\bs\M,\om)$ for even $m$. When $m=4$, they construct virtual cycles for proper, oriented $-2$-shifted symplectic derived $\C$-schemes, and propose to use these to define Donaldson--Thomas type `DT4 invariants' counting derived moduli schemes $\bs\M_\al^{\rm ss}(\tau)$ of $\tau$-semistable coherent sheaves on Calabi--Yau 4-folds. An important ingredient in this programme is a choice of orientation on~$(\bs\M_\al^{\rm ss}(\tau),\om)$.

Let $X$ be a Calabi--Yau $m$-fold for $m=4k$. Very roughly, Cao, Gross, and Joyce \cite{CGJ} define a continuous map $\Phi:\bs\M_{\rm vect}^{\rm top}\ra \coprod_{\text{$\U(n)$-bundles $P\ra X$}}\B_P^{\rm top}$, where $\bs\M_{\rm vect}^{\rm top},\B_P^{\rm top}$ are the `topological realizations' of $\bs\M_{\rm vect},\B_P$, such that the orientation bundles $O^{E_\bu}_P\ra\B_P$ of \S\ref{os22} for $E_\bu$ the positive Dirac operator $\slashed{D}_+$ on $X$ pull back under $\Phi$ to the principal $\Z_2$-bundle of Borisov--Joyce orientations on $(\bs\M_{\rm vect},\om)$. Hence orientations on $\B_P$ for all principal $\U(n)$-bundles $P\ra X$ induce orientations on $(\bs\M_{\rm vect},\om)$. For Calabi--Yau 4-folds, they show that all $\B_P$ are orientable, so $(\bs\M_{\rm vect},\om)$ is orientable, and then they deduce that $(\bs\M_{\rm coh},\om)$ and $(\bs\M,\om)$ are orientable by a `group completion' argument.

In general, one should expect that `orientations' on complex algebraic moduli spaces of vector bundles, coherent sheaves or complexes, or Higgs bundles, on a smooth projective $\C$-scheme $X$, can be pulled back from orientations on $\B_P$ for principal $\U(n)$-bundles $P\ra X$. This is useful as orientations on $\B_P$ are much easier to understand than orientations on algebro-geometric moduli spaces.

\section{Spin structures on moduli spaces}
\label{os3}

In \S\ref{os22}--\S\ref{os29}, given a real elliptic operator $E_\bu$ on $X$, we defined and discussed `orientations' on moduli spaces $\A_P,\B_P,\ovB_P$ for principal $G$-bundles $P\ra X$. Now, given a complex elliptic operator $F_\bu$ on $X$, we will define and discuss `spin structures' on $\A_P,\B_P,\ovB_P$. So far as the authors know this is a new idea, though as we explain in \S\ref{os33}, it is connected to the notion of `orientation data' in Donaldson--Thomas theory of Calabi--Yau 3-folds \cite[\S 5.2]{KoSo1}. Our main results in \S\ref{os4} relate spin structures on $\B_P$ for principal $G$-bundles $P\ra X$ to orientations on $\B_Q$ for principal $G$-bundles $Q\ra X\t\cS^1$ with $Q\vert_{X\t\{\text{point}\}}\cong P$.

\subsection{Complex elliptic operators and spin structures}
\label{os31}

We discuss complex elliptic operators, generalizing Definition \ref{os2def2}.

\begin{dfn}
\label{os3def1}
Let $X$ be a compact manifold. Suppose we are given complex vector bundles $F_0,F_1\ra X$, of the same rank $r$, and a complex linear elliptic partial differential operator $D:\Ga^\iy(F_0)\ra\Ga^\iy(F_1)$, of degree $d$. As a shorthand we write $F_\bu=(F_0,F_1,D)$. We may write $F_\bu$ in terms of connections as in \eq{os2eq1}.

Now suppose we are given Hermitian metrics $h_{F_0},h_{F_1}$ (that is, Euclidean metrics on $F_0,F_1$ compatible with the complex structures) on the fibres of $F_0,F_1$, and a volume form $\d V$ on $X$. Then as in Definition \ref{os2def2} there is a unique adjoint operator $D^*:\Ga^\iy(F_1)\ra\Ga^\iy(F_0)$ satisfying \eq{os2eq2}. It is complex anti-linear in $D$. We call $D$ {\it Hermitian self-adjoint\/} if $F_0=F_1$, $h_{F_0}=h_{F_1}$, and $D^*=D$. This is the obvious notion of self-adjointness for complex elliptic operators.

However, later we will need a different notion of self-adjointness. Write $\bar F_0,\bar F_1$ for the complex conjugate vector bundles of $F_0,F_1$ (the same real vector bundles, but the complex structures change sign), and $\bar D:\Ga^\iy(\bar F_0)\ra\Ga^\iy(\bar F_1)$ for the complex conjugate operator (as real vector spaces and operators $\Ga^\iy(\bar F_j)=\Ga^\iy(F_j)$ and $\bar D=D$). Then we call $D$ {\it antilinear self-adjoint\/} if $F_0=\bar F_1$, and $h_{F_0}=h_{F_1}$, and $D^*=\bar D$. For example, if $(X,g)$ is a spin Riemannian manifold of dimension $8n+6$ then the positive Dirac operator $\slashed{D}_+:\Ga^\iy(S_+)\ra \Ga^\iy(S_-)$ is antilinear self-adjoint.
\end{dfn}

Here is one of the central definitions of this paper, which may be new.

\begin{dfn} 
\label{os3def2}
Suppose $X,G,P,\A_P,\B_P,\ovB_P$ are as Definition \ref{os2def1}, and $F_\bu$ is a complex elliptic operator on $X$ as in Definition \ref{os3def1}. Let $\nabla_P\in\A_P$. Then $\nabla_P$ induces a connection $\nabla_{\Ad(P)}$ on the vector bundle $\Ad(P)\ra X$. Thus we may form the twisted complex elliptic operator
\e
\begin{split}
D^{\nabla_{\Ad(P)}}&:\Ga^\iy(\Ad(P)\ot_\R F_0)\longra\Ga^\iy(\Ad(P)\ot_\R F_1),\\
D^{\nabla_{\Ad(P)}}&:f\longmapsto \ts\sum_{i=0}^d (\id_{\Ad(P)}\ot a_i)\cdot \nabla_{\Ad(P)\ot F_0}^if,
\end{split}
\label{os3eq1}
\e
where $\nabla_{\Ad(P)\ot F_0}$ are the connections on $\Ad(P)\ot F_0\ot\bigot^iT^*X$ for $0\le i<d$ induced by $\nabla_{\Ad(P)}$ and~$\nabla_{F_0}$.

Since $D^{\nabla_{\Ad(P)}}$ is a complex linear elliptic operator on a compact manifold $X$, it has finite-dimensional kernel $\Ker(D^{\nabla_{\Ad(P)}})$ and cokernel $\Coker(D^{\nabla_{\Ad(P)}})$. The {\it determinant\/} $\det_\C(D^{\nabla_{\Ad(P)}})$ is the 1-dimensional complex vector space
\e
\det_\C(D^{\nabla_{\Ad(P)}})=\det_\C\Ker(D^{\nabla_{\Ad(P)}})\ot\bigl(\det_\C\Coker(D^{\nabla_{\Ad(P)}})\bigr)^*,
\label{os3eq2}
\e
where if $V$ is a finite-dimensional complex vector space then $\det_\C V=\La_\C^{\dim_\C V}V$.

These operators $D^{\nabla_{\Ad(P)}}$ vary continuously with $\nabla_P\in\A_P$, so they form a family of elliptic operators over the base topological space $\A_P$. Thus as in \cite{AtSi4} there is a natural complex line bundle $\hat K{}^{F_\bu}_P\ra\A_P$ with fibre $\hat K{}^{F_\bu}_P\vert_{\nabla_P}=\det_\C(D^{\nabla_{\Ad(P)}})$ at each $\nabla_P\in\A_P$. It is equivariant under the actions of $\G_P$ and $\G_P/Z(G)$ on $\A_P$, and so pushes down to complex line bundles $K^{F_\bu}_P\ra\B_P$, $\bar K^{F_\bu}_P\ra\ovB_P$ on the topological stacks $\B_P,\ovB_P$, with $K^{F_\bu}_P\cong\Pi_P^*(\bar K_P^{F_\bu})$. We call $K^{F_\bu}_P,\bar K^{F_\bu}_P$ the {\it determinant line bundles\/} of $\B_P,\ovB_P$. The restriction $\bar K^{F_\bu}_P\vert_{\ovB_P^\irr}$ is a topological complex line bundle in the usual sense on the topological space~$\ovB_P^\irr$.

A {\it spin structure\/} $\si_P^{F_\bu}$ on $\B_P$ is a choice of square root line bundle $(K^{F_\bu}_P)^{1/2}$ for $K^{F_\bu}_P\ra\B_P$, up to isomorphism. That is, $\si_P^{F_\bu}=[L,\io]$ is an equivalence class of pairs $(L,\io)$, where $L\ra\B_P$ is a topological complex line bundle on $\B_P$, and $\io:L^{\ot^2}\ra K^{F_\bu}_P$ is an isomorphism, and pairs $(L,\io),(L',\io')$ are equivalent if there exists an isomorphism $\jmath:L\ra L'$ with~$\io'\ci(\jmath\ot\jmath)=\io:L^{\ot^2}\ra K^{F_\bu}_P$.

We can similarly define a {\it spin structure on\/} $\A_P$ to be a $\G_P$-equivariant square root $(\hat K^{F_\bu}_P)^{1/2}$ for $\hat K^{F_\bu}_P\ra\A_P$, up to $\G_P$-equivariant isomorphism, and a {\it spin structure on\/} $\ovB_P$ to be a square root $(\bar K^{F_\bu}_P)^{1/2}$ for $\bar K^{F_\bu}_P\ra\ovB_P$, up to isomorphism. Then spin structures on $\A_P$ and $\B_P$ are equivalent.

In the general case, spin structures on $\A_P,\B_P$ and $\ovB_P$ are {\it not\/} quite equivalent. Under the morphism $\Pi_P:\B_P\ra\ovB_P$, spin structures $(\bar K^{F_\bu}_P)^{1/2}$ on $\ovB_P$ pull back to spin structures $(K^{F_\bu}_P)^{1/2}=\Pi_P^*((\bar K^{F_\bu}_P)^{1/2})$ on $\B_P$, and this gives an injective map from spin structures on $\ovB_P$ to those on $\B_P$. A spin structure $(K^{F_\bu}_P)^{1/2}$ lies in the image of this map if and only if for any (equivalently, for all) $[\nabla_P]$ in $\B_P$, the centre
$Z(G)\subseteq\Stab_{\G_P}(\nabla_P)=\Iso_{\B_P}([\nabla_P])$ acts trivially on the fibre $(K^{F_\bu}_P)^{1/2}\vert_{[\nabla_P]}$. This happens automatically if all group morphisms $Z(G)\ra\{\pm 1\}$ are the identity, and then spin structures on $\A_P,\B_P,\ovB_P$ agree.
\end{dfn}

\begin{rem}
\label{os3rem2}
Here is why we chose the term `spin structure'. Suppose $Y$ is a complex $m$-manifold. The {\it canonical bundle\/} $K_Y=\La_\C^mT^*Y$ is a complex line bundle $K_Y\ra Y$. Then spin structures on $Y$ are in natural 1--1 correspondence with isomorphism classes of square roots $K_Y^{1/2}$, see~\cite[App.~D]{LaMi}. We can think of $\A_P,\B_P,\ovB_P$ as infinite-dimensional complex manifolds, and $\hat K{}^{F_\bu}_P,K^{F_\bu}_P,\bar K^{F_\bu}_P$ as their canonical bundles. We will see in \S\ref{os32} that spin structures on $\ovB_P$ are related to spin structures in the usual sense on smooth gauge theory moduli spaces.
\end{rem}

\subsection{Applications to gauge theory moduli spaces}
\label{os32}

In \S\ref{os28} we explained that for many interesting gauge theory moduli spaces $\M_P^{\rm ga}$ of instanton-type connections on a principal bundle $P\ra X$, the deformation theory of $\M_P^{\rm ga}$ is controlled by a real elliptic operator $E_\bu$ on $X$, and if $\M_P^{\rm ga}$ is smooth then orientations on $\ovB_P$ restrict to orientations on $\M_P^{\rm ga}$ under the inclusion $\M_P^{\rm ga}\hookra\ovB_P$. In some cases $E_\bu$ is actually the real operator underlying a complex elliptic operator $F_\bu$. Then $\M_P^{\rm ga}$ may be a complex manifold (or almost complex manifold) whose canonical bundle $K_{\M_P^{\rm ga}}$ is the restriction of $\bar K^{F_\bu}_P\ra\ovB_P$. Hence a spin structure for $\ovB_P$, that is, a square root for $\bar K^{F_\bu}_P$, restricts to a square root for $K_{\M_P^{\rm ga}}$, which as in Remark \ref{os3rem2} is equivalent to a spin structure on the manifold $\M_P^{\rm ga}$ in the usual sense.

\begin{ex}
\label{os3ex1}
In Example \ref{os2ex4}, suppose the oriented Riemannian 4-manifold $(X,g)$ is a K\"ahler surface. Then the elliptic operator $E_\bu$ in \eq{os2eq24} may be rewritten as a complex elliptic operator $F_\bu$ given by
\e
D=\db+\db^*:\Ga^\iy(\La^{0,0}T^*X\op\La^{0,2}T^*X)\longra\Ga^\iy(\La^{0,1}T^*X).	
\label{os3eq3}
\e
If an instanton moduli space $\M_P^{\rm asd}$ is unobstructed, it is a complex manifold, with complex tangent space $\Coker D^{\nabla_{\Ad(P)}}$ at each $[\nabla_P]\in\M_P^{\rm asd}$, and its canonical bundle $K_{\M_P^{\rm asd}}$ is the restriction of the complex line bundle $\bar K^{F_\bu}_P\ra\ovB_P$ under the inclusion $\M_P^{\rm asd}\hookra\ovB_P$. Hence a spin structure for $\ovB_P$ restricts to a spin structure on the manifold $\M_P^{\rm asd}$.
\end{ex}

As in Remark \ref{os2rem3}, in enumerative invariant problems such as Donaldson theory and Seiberg--Witten theory, even if gauge theory moduli spaces $\M_P^{\rm ga}$ are not smooth, one hopes to make them into compact, oriented derived manifolds, which have bordism classes $[\M_P^{\rm ga}]_{\rm bord}$. If the gauge theory problem comes from a complex elliptic operator $F_\bu$ (even up to deformation), we should lift $[\M_P^{\rm ga}]_{\rm bord}$ to a unitary bordism class. Given a spin structure on $\ovB_P$, we should be able to lift this to a spin (unitary) bordism class. So considering spin structures as in \S\ref{os31} might lead to refined versions of enumerative invariants. 

See Sasahira \cite{Sasa} for a refined version of Seiberg--Witten invariants involving spin bordism classes $[\M^{\rm SW}]_{\rm bord}$ of Seiberg--Witten moduli spaces $\M^{\rm SW}$, with an essentially arbitrary choice of spin structure on $\M^{\rm SW}$. We expect that our methods could be used to choose the spin structure canonically.

\subsection{Applications in complex (derived) algebraic geometry}
\label{os33}

As in \S\ref{os29}, by Pantev, To\"en, Vaqui\'e, and Vezzosi \cite{PTVV,ToVa,ToVe1,ToVe2}, if $X$ is a Calabi--Yau $m$-fold over $\C$ and $\bs\M$ is the derived moduli stack of complexes of coherent sheaves on $X$, then $\bs\M$ has a $(2-m)$-shifted symplectic structure $\om$. When $m$ is even, Borisov and Joyce \cite[\S 2.4]{BoJo} define `orientations' on $(\bs\M,\om)$. When $m=4k$ these are related to orientations on $\B_P$ for principal $\U(n)$-bundles $P\ra X$, and when $m=4$ they are important in defining DT4 invariants~\cite{BoJo}.

It turns out that there is a parallel story for Calabi--Yau $m$-folds when $m$ is odd, related to spin structures on $\B_P$ as in \S\ref{os32}, though the usual term used in the literature is `orientation data' rather than `spin structure'.

In 2008, Kontsevich and Soibelman \cite[\S 5]{KoSo1} introduced the notion of `orientation data' on an odd Calabi--Yau category $\cC$, such as $\coh(X)$ or $D^b\coh(X)$ for $X$ a Calabi--Yau $m$-fold with $m$ odd. Oversimplifying a bit, if $\bs\M$ is the derived moduli stack of objects in $\cC$ (which is $(2-m)$-shifted symplectic), $\bL_{\bs\M}$ is its cotangent complex, and $K_{\bs\M}=\det(\bL_{\bs\M})\ra\bs\M$ its determinant line bundle, then orientation data is a choice of square root $K_{\bs\M}^{1/2}$ satisfying some compatibility conditions under direct sums in $\cC$, where the compatibility conditions need the odd Calabi--Yau assumption to state. Nekrasov and Okounkov \cite[\S 6]{NeOk} gave a simple argument showing that square roots $K_{\bs\M}^{1/2}$ must exist for odd Calabi--Yau categories $\cC$, but whether there is a canonical choice, and whether we can choose $K_{\bs\M}^{1/2}$ to satisfy the compatibility conditions, appears to be unknown.

Kontsevich and Soibelman \cite{KoSo1} needed orientation data for their motivic Donaldson--Thomas invariants of Calabi--Yau 3-folds (see also \cite{BJM}). Later, orientation data was found to be necessary in other generalizations of Donaldson--Thomas theory for 3-Calabi--Yau categories, including Kontsevich and Soibelman's Cohomological Hall Algebras \cite{KoSo2}, and categorification of Donaldson--Thomas theory using perverse sheaves by Ben-Bassat, Brav, Bussi, Dupont, Joyce, and Szendr\H oi~\cite{BBJ,BBDJS,Joyc2}. 

In the sequel \cite{JoUp2}, using a similar argument to Cao, Gross, and Joyce \cite{CGJ} for Calabi--Yau $4k$-folds, we will show that if $X$ is a Calabi--Yau $m$-fold for $m=4k+3$, then orientation data on $\bs\M$ can be pulled back from a differential-geometric notion of orientation data on $X$, involving choices of spin structures $\si_P^{F_\bu}$ on $\B_P$ for all principal $\U(m)$-bundles $P\ra X$, satisfying compatibility conditions, with $F_\bu$ the positive Dirac operator on $X$. When $m=3$, we show there is a canonical choice of such differential-geometric orientation data, and hence construct canonical orientation data on $D^b\coh(X)$ in the sense of \cite{KoSo1,KoSo2}. This was the authors' motivation for writing this paper.

In the situation of \S\ref{os31}, for a complex elliptic operator $F_\bu$ on $X$ such as $\db+\db^*:\Ga^\iy(\La^{0,\rm even}T^*X)\ra \Ga^\iy(\La^{0,\rm odd}T^*X)$ as in \eq{os3eq3}, the analogue of the `$4k$-Calabi--Yau' condition is that $F_\bu\cong\bar F_\bu$, that is, $F_\bu=E_\bu\ot_\R\C$ for $E_\bu$ a real elliptic operator on $X$, and the analogue of the `$(4k+3)$-Calabi--Yau' condition is that $F^*_\bu\cong\bar F_\bu$, that is, $F_\bu$ is antilinear self-adjoint as in Definition \ref{os3def1}. It is significant that much of this paper needs $F_\bu$ to be antilinear self-adjoint.

\section{Elementary results on spin structures}
\label{os4}

We now develop some basic theory for spin structures on moduli spaces $\B_P$. Sections \ref{os42}--\ref{os45} are analogues of \S\ref{os23}--\S\ref{os26} for (n-)orientations.

\subsection{\texorpdfstring{Square roots of line bundles on products $X\t Y$}{Square roots of line bundles on products X x Y}}
\label{os41}

\begin{prop}
\label{os4prop1}
Suppose\/ $X,Y,Z$ are nonempty topological spaces or topological stacks with\/ $X,Y$ connected,\/ $I\ra X,$\/ $J\ra Y,$\/ $K\ra Z$ and\/ $L\ra X\t Y$ are topological complex line bundles,\/ $\Phi:X\t Y\ra Z$ is continuous, and\/ $\phi:(I\bt J)\ot L^{\ot^2}\ra \Phi^*(K)$ is an isomorphism of line bundles on\/ $X\t Y,$ where\/ {\rm`$\bt$'} is the external tensor product. Then for each square root\/ $K^{1/2}$ of\/ $K,$ there exist square roots\/ $I^{1/2},J^{1/2}$ of\/ $I,J,$ unique up to isomorphism, with an isomorphism\/ $\phi^{1/2}:(I^{1/2}\bt J^{1/2})\ot L\ra \Phi^*(K^{1/2})$ on\/ $X\t Y$ such that\/~$(\phi^{1/2})^{\ot^2}=\phi$.
\end{prop}

\begin{proof} In the situation of the proposition, fix a square root $K^{1/2}$ for $K$. Let $x\in X$ and $y\in Y$, and choose square roots $I\vert_x^{1/2},J\vert_y^{1/2}$ for the 1-dimensional complex vector spaces $I\vert_x,J\vert_y$. Define line bundles $I^{1/2}\ra X$, $J^{1/2}\ra Y$ by
\e
\begin{split}
I^{1/2}&=\bigl(\Phi^*(K^{1/2})\ot L^*\bigr)\vert_{X\t\{y\}}\ot (J\vert_y^{1/2})^*,\\
J^{1/2}&=\bigl(\Phi^*(K^{1/2})\ot L^*\bigr)\vert_{\{x\}\t Y}\ot (I\vert_x^{1/2})^*,
\end{split}
\label{os4eq1}
\e
identifying $X\cong X\t\{y\},$ $Y\cong\{x\}\t Y$. Then we have canonical isomorphisms
\begin{align*}
I^{1/2}\ot I^{1/2}&\cong \bigl(\Phi^*(K)\ot (L^*)^{\ot^2}\bigr)\vert_{X\t\{y\}}\ot (J\vert_y)^*\\
&\,\,{\buildrel\phi\over\cong}\, \bigl((I\bt J)\ot L^{\ot^2}\ot (L^*)^{\ot^2}\bigr)\vert_{X\t\{y\}}\ot (J\vert_y)^*\cong I,
\end{align*}
so $I^{1/2}$ is a square root of $I$, and similarly $J^{1/2}$ is a square root of $J$. By restricting $\phi^{1/2}:(I^{1/2}\bt J^{1/2})\ot L\ra \Phi^*(K^{1/2})$ to $X\t\{y\},$ $Y\cong\{x\}$ we see that any $I^{1/2},J^{1/2}$ satisfying the conditions of the proposition must be isomorphic to those in \eq{os4eq1}. Hence $I^{1/2},J^{1/2}$ are unique up to isomorphism if they exist.

Now $I^{1/2}\bt J^{1/2}$ and $\Phi^*(K^{1/2})\ot L^*$ are square roots of $I\bt J$ on $X\t Y$, and by \eq{os4eq1} we have isomorphisms compatible with the square roots
\e
\begin{split}
\bigl(I^{1/2}\bt J^{1/2}\bigr)\vert_{X\t\{y\}}&\cong\bigl(\Phi^*(K^{1/2})\ot L^*\bigr)\vert_{X\t\{y\}},\\
\bigl(I^{1/2}\bt J^{1/2}\bigr)\vert_{\{x\}\t Y}&\cong\bigl(\Phi^*(K^{1/2})\ot L^*\bigr)\vert_{\{x\}\t Y}.\end{split}
\label{os4eq2}
\e
Two square roots of the same line bundle on $X\t Y$ differ by a principal $\Z_2$-bundle $Q\ra X\t Y$. Since $X,Y$ are connected, by taking monodromy $Q$ is equivalent to a group morphism $q:\pi_1(X\t Y)\ra\Z_2=\{\pm 1\}$, where $\pi_1(X\t Y)\cong\pi_1(X)\t\pi_1(Y)$. Equation \eq{os4eq2} implies that $q\vert_{\pi_1(X)\t\{1\}}\equiv 1$ and $q\vert_{\{1\}\t\pi_1(Y)}\equiv 1$, so $q\equiv 1$, and $Q$ is trivial, which forces $I^{1/2}\bt J^{1/2}\cong\Phi^*(K^{1/2})\ot L^*$ as square roots of $I\bt J$. The proposition follows.
\end{proof}

\begin{rem}
\label{os4rem1}
We will apply this as follows: suppose in the situation of \S\ref{os31} we have moduli spaces $\B_P,\B_Q,\B_R$, and we construct a morphism $\Phi:\B_P\t\B_Q\ra\B_R$ and an isomorphism for some line bundle~$L\ra\B_P\t\B_Q$:
\begin{equation*}
\phi:(K^{F_\bu}_P\bt K^{F_\bu}_Q)\ot L^{\ot^2}\ra \Phi^*(K^{F_\bu}_R).
\end{equation*}
Then Proposition \ref{os4prop1} implies that a spin structure $(K^{F_\bu}_R)^{1/2}$ for $\B_R$ determines unique spin structures $(K^{F_\bu}_P)^{1/2},(K^{F_\bu}_Q)^{1/2}$ for $\B_P,\B_Q$.

Note that this is stronger than the analogous fact for orientations: if in the situation of \S\ref{os2} we had an analogue of $\phi$ for real line bundles $L^{E_\bu}_P,L^{E_\bu}_Q,L^{E_\bu}_R$, then an orientation for $\B_R$ implies $\B_P,\B_Q$ are orientable, but does not determine unique orientations on $\B_P,\B_Q$. This makes the theory of spin structures simpler than that of orientations in some ways (see Theorem \ref{os5thm1} for instance).
\end{rem}

\subsection{\texorpdfstring{Natural spin structures when $G$ is abelian}{Natural spin structures when G is abelian}}
\label{os42}

In the situation of Definition \ref{os3def2}, suppose the Lie group $G$ is abelian, for example $G=\U(1)$. Then as in \S\ref{os23} the line bundles $\hat K{}^{F_\bu}_P\ra\A_P$, $K^{F_\bu}_P\ra\B_P$, $\bar K^{F_\bu}_P\ra\ovB_P$ are all canonically trivial, with fibre 
\begin{equation*}
(\det_\C D)^{\ot_\C^{\dim_\R\g}}\ot_\R(\La_\R^{\dim_\R\g}\g)^{\ot_\R^{\ind_\C D}}.
\end{equation*}
Hence they have trivial square roots, canonical up to isomorphism, giving natural spin structures on $\A_P,\B_P,\ovB_P$.

\subsection{Spin structures on products of moduli spaces}
\label{os43}

Let $X$ and $F_\bu$ be fixed, and suppose $G,H$ are Lie groups, and $P\ra X$, $Q\ra X$ are principal $G$- and $H$-bundles, so that $P\t_XQ$ is a principal $G\t H$ bundle over $X$. As in \S\ref{os24} we have an isomorphism of topological stacks $\La_{P,Q}:\B_P\t\B_Q\ra\B_{P\t_XQ}$. By the analogue of \eq{os2eq8}--\eq{os2eq9} we have a natural isomorphism of line bundles on~$\B_P\t\B_Q$:
\e
\la_{P,Q}^{F_\bu}:K_P^{F_\bu}\bt K_Q^{F_\bu}\longra
\La_{P,Q}^*(K_{P\t_XQ}^{F_\bu}).
\label{os4eq3}
\e
Since $\B_P,\B_Q$ are connected, we see from Proposition \ref{os4prop1} that \eq{os4eq3} induces a bijection between pairs $\bigl((K_P^{F_\bu})^{1/2},(K_Q^{F_\bu})^{1/2}\bigr)$ of spin structures on $\B_P,\B_Q$, and spin structures $(K_{P\t_XQ}^{F_\bu})^{1/2}$ on $\B_{P\t_XQ}$.

\begin{rem}
\label{os4rem2}
In Remark \ref{os2rem2} we noted that isomorphisms such as \eq{os4eq3} depend on an orientation convention and that exchanging $P$ and $Q$ introduces a sign as in \eq{os2eq11} for n-orientation bundles. Similarly, the sign change for \eq{os4eq3} is given by $(-1)^{\ind_\C D^{\nabla_{\Ad(P)}}\cdot\ind_\C D^{\nabla_{\Ad(Q)}}}$. For spin structures, this issue will not arise, for two reasons. Firstly, a sign change in an isomorphism $K_1\cong K_2$ between complex line bundles $K_1,K_2$ has no effect in the identification between square roots $K_1^{1/2},K_2^{1/2}$. Secondly, our main results concern {\it self-adjoint\/} elliptic operators, which have index zero, so the sign changes are 1 anyway.
\end{rem}

\subsection{\texorpdfstring{Relating moduli spaces for discrete quotients $G\twoheadrightarrow H$}{Relating moduli spaces for discrete quotients G→H}}
\label{os44}

Suppose $G$ is a Lie group, $K\subset G$ a discrete normal subgroup, and set $H=G/K$ for the quotient Lie group. Let $X,F_\bu$ be fixed. If $P\ra X$ is a principal $G$-bundle, then $Q:=P/K$ is a principal $H$-bundle over $X$. As in \S\ref{os25} we have a natural morphism $\De_{P,Q}:\B_P\ra\B_Q$ of topological stacks, and the analogue of \eq{os2eq13} is a canonical isomorphism
\e
\de_{P,Q}^{F_\bu}:K_P^{F_\bu}\,{\buildrel\cong\over\longra}\,\De_{P,Q}^*(K_Q^{F_\bu}).
\label{os4eq4}
\e
Hence spin structures $(K_Q^{F_\bu})^{1/2}$ on $\B_Q$ pull back to spin structures $(K_P^{F_\bu})^{1/2}=\De_{P,Q}^*((K_Q^{F_\bu})^{1/2})$ on $\B_P$. Here is the analogue of Example~\ref{os2ex1}.

\begin{ex} Let $G=\SU(m)\t\U(1)$, $K=\Z_m\subset G$ and $H=G/K\cong\U(m)$ be as in Example \ref{os2ex1}. For fixed $X,F_\bu$, let $P\ra X$ be a principal $\SU(m)$-bundle, and $P'=X\t\U(1)\ra X$ be the trivial principal $\U(1)$-bundle. Write $P''=P\t\U(1)=P\t_XP'\ra X$ for the associated principal $\SU(m)\t\U(1)$-bundle, and define $Q=(P\t\U(1))/\Z_m\cong(P\t\U(m))/\SU(m)$ to be the quotient principal $\U(m)$-bundle. Define $\Ka_{P,Q}:\B_P\ra\B_Q$ as in \eq{os2eq15}. As for \eq{os2eq16}, choose an isomorphism $\io:\C\ra K_{P'}^{F_\bu}\vert_{[\nabla^0]}$, and define an isomorphism of line bundles $\ka_{P,Q}^{F_\bu}:K_P^{F_\bu}\ra \Ka_{P,Q}^*(K_Q^{F_\bu})$ by the commutative diagram
\begin{equation*}
\!\!\!\!\xymatrix@!0@C=160pt@R=45pt{ 
*+[r]{K_P^{F_\bu}} \ar[d]^(0.45){\id_{K_P^{F_\bu}}\ot\io} \ar[rr]_{\ka_{P,Q}^{F_\bu}} && *+[l]{\Ka_{P,Q}^*(K_Q^{F_\bu})} \ar@{=}[d]
\\
*+[r]{\begin{subarray}{l}\ts (\id_{\B_P},[\nabla^0])^* \\ \ts (K_P^{F_\bu}\bt K_{P'}^{F_\bu})\end{subarray}} \ar[r]^{\begin{subarray}{l}(\id_{\B_P},[\nabla^0])^* \\ \;\>(\la_{P,P'}^{F_\bu}) \end{subarray}} &
{\;\>\begin{subarray}{l}\ts (\id_{\B_P},[\nabla^0])^* \\ \ts \ci\La_{P,P'}^*(K_{P''}^{F_\bu})\end{subarray}}  \ar[r]^(0.36){\begin{subarray}{l} (\id_{\B_P},[\nabla^0])^*\ci{} \\ \La_{P,P'}^*(\de_{P'',Q}^{F_\bu})\end{subarray}} & *+[l]{\begin{subarray}{l}\ts \qquad\,\,\,\, (\id_{\B_P},[\nabla^0])^*\ci{} \\ \ts (\De_{P'',Q}\!\ci\!\La_{P,P'})^*(K_Q^{F_\bu}).\end{subarray}\;\>}\!\!\!\!
}
\end{equation*}
Hence spin structures $(K_Q^{F_\bu})^{1/2}$ on $\B_Q$ pull back to spin structures $(K_P^{F_\bu})^{1/2}=\Ka_{P,Q}^*((K_Q^{F_\bu})^{1/2})$ on $\B_P$. This identification is independent of the choice of~$\io$.
\label{os4ex1}	
\end{ex}

\subsection{\texorpdfstring{Relating moduli spaces for Lie subgroups $G\subset H$}{Relating moduli spaces for Lie subgroups G⊂H}}
\label{os45}

We follow \S\ref{os26}, with modifications. Let $X,F_\bu$ be fixed, and let $H$ be a Lie group and $G\subset H$ a Lie subgroup, with Lie algebras $\g\subset\h$. If $P\ra X$ is a principal $G$-bundle, then $Q:=(P\t H)/G$ is a principal $H$-bundle over $X$. As in \S\ref{os26} we have a natural morphism $\Xi_{P,Q}:\B_P\ra\B_Q$ of topological stacks. Thus, we can try to compare the line bundles $K_P^{F_\bu},\Xi_{P,Q}^*(K_Q^{F_\bu})$ on~$\B_P$.

Write $\m=\h/\g$, and $\rho_\R:G\ra\Aut(\m)$ for the real representation induced by the adjoint representation of $H\supset G$. Then as for \eq{os2eq18} we have an isomorphism of line bundles on $\B_P$
\e
\xi_P^Q:K_P^{F_\bu}\ot_\C K_{P,\rho_\R}^{F_\bu}\longra\Xi_{P,Q}^*(K_Q^{F_\bu}),
\label{os4eq5}
\e
where $K_{P,\rho_\R}^{F_\bu}\ra\B_P$ is the determinant line bundle associated to the family of complex elliptic operators $\nabla_P\mapsto D^{\nabla_{\rho_\R(P)}}$ on $\B_P$.

Suppose as in \S\ref{os26} that we can give $\m$ the structure of a {\it complex\/} vector space $\m^\C$, such that $\rho_\C=\rho_\R:G\ra\Aut(\m^\C)$ is complex linear. We also have a complex conjugate vector space $\bar\m^\C$ and complex conjugate representation $\bar\rho_\C:G\ra\Aut(\bar\m^\C)$. Then $\m\ot_\R\C=\m^\C\op\ov\m^\C$, so that $\rho_\R(P)\ot_\R\C=\rho_\C(P)\op\bar\rho_\C(P)$ in complex vector bundles on $\B_P$. We have
\begin{gather*}
D^{\nabla_{\rho_\R(P)}}\!=\!D^{\nabla_{\rho_\C(P)}}\!\op\! D^{\nabla_{\bar\rho_\C(P)}}\!:\!
\Ga^\iy\bigl(\rho_\R(P)\!\ot_\R \!F_0\bigr)\!=\!\Ga^\iy\bigl((\rho_\C(P)\!\op\!\bar\rho_\C(P))\!\ot_\C \!F_0\bigr)\\
\longra\Ga^\iy\bigl(\rho_\R(P)\ot_\R F_1\bigr)=\Ga^\iy\bigl((\rho_\C(P)\op\bar\rho_\C(P))\ot_\C F_1\bigr).
\end{gather*}
So taking determinant line bundles gives
\e
K_{P,\rho_\R}^{F_\bu}\cong K_{P,\rho_\C}^{F_\bu}\ot_\C K_{P,\bar\rho_\C}^{F_\bu},
\label{os4eq6}
\e
with $K_{P,\rho_\C}^{F_\bu},K_{P,\bar\rho_\C}^{F_\bu}\ra\B_P$ the determinants of $\nabla_P\mapsto D^{\nabla_{\rho_\C(P)}},D^{\nabla_{\bar\rho_\C(P)}}$.

Next suppose that the complex elliptic operator $F_\bu$ is {\it antilinear self-adjoint\/} in the sense of Definition \ref{os3def1}. Then $\bar F_\bu\cong F_\bu^*$, so twisting by $\rho_\C(P),\bar\rho_\C(P)$ for $[\nabla_P]\in\B_P$ shows that $\ov{D^{\nabla_{\rho_\C(P)}}}\cong(D^{\nabla_{\bar\rho_\C(P)}})^*$, so taking determinant line bundles gives an isomorphism $\overline{K_{P,\rho_\C}^{F_\bu}}\cong(K_{P,\bar\rho_\C}^{F_\bu})^*$.

Now complex line bundles are equivalent to principal $\C^*$-bundles, where $\C^*\cong(0,\iy)\t\U(1)$. Thus we may canonically write 
\e
K_{P,\rho_\C}^{F_\bu}=R_{P,\rho_\C}^{F_\bu}\ot_\R L_{P,\rho_\C}^{F_\bu},\quad
K_{P,\bar\rho_\C}^{F_\bu}=R_{P,\bar\rho_\C}^{F_\bu}\ot_\R L_{P,\bar\rho_\C}^{F_\bu},
\label{os4eq7}
\e
where $R_{P,\rho_\C}^{F_\bu},R_{P,\bar\rho_\C}^{F_\bu}$ are the real line bundles associated to the principal $(0,\iy)$-bundles in $K_{P,\rho_\C}^{F_\bu},K_{P,\bar\rho_\C}^{F_\bu}$, and $L_{P,\rho_\C}^{F_\bu},L_{P,\bar\rho_\C}^{F_\bu}$ the complex line bundles associated to the principal $\U(1)$-bundles in $K_{P,\rho_\C}^{F_\bu},K_{P,\bar\rho_\C}^{F_\bu}$. As $\overline{K_{P,\rho_\C}^{F_\bu}}\cong(K_{P,\bar\rho_\C}^{F_\bu})^*$ we have
\e
R_{P,\rho_\C}^{F_\bu}\cong (R_{P,\bar\rho_\C}^{F_\bu})^*\quad \text{and}\quad
L_{P,\rho_\C}^{F_\bu}\cong L_{P,\bar\rho_\C}^{F_\bu}.
\label{os4eq8}
\e
Combining \eq{os4eq6}--\eq{os4eq8}, we may rewrite \eq{os4eq5} as an isomorphism on $\B_P$:
\e
\xi_{P,Q}^{F_\bu}:K_P^{F_\bu}\ot_\C (L_{P,\rho_\C}^{F_\bu})^{\ot^2}\longra\Xi_{P,Q}^*(K_Q^{F_\bu})
\label{os4eq9}
\e
We deduce:

\begin{prop}
\label{os4prop2}
In the situation above, with\/ $\h=\g\op\m$ for\/ $\m$ a complex\/ $G$-representation and\/ $F_\bu$ antilinear self-adjoint, if\/ $(K_Q^{F_\bu})^{1/2}$ is a spin structure on\/ $\B_Q$ then\/ $(K_P^{F_\bu})^{1/2}:=\Xi_{P,Q}^*((K_Q^{F_\bu})^{1/2})\ot_\C(L_{P,\rho_\C}^{F_\bu})^*$ is a spin structure on\/ $\B_P$ by \eq{os4eq9}. So spin structures on\/ $\B_Q$ pull back under\/ $\Xi_{P,Q}$ to spin structures on\/~$\B_P$.
\end{prop}

The next example is a kind of converse to Example~\ref{os4ex1}.

\begin{ex} As in Example \ref{os2ex2}, we have an inclusion $\U(m)\hookra\SU(m+1)$ with an isomorphism $\m=\su(m+1)/\u(m)\cong\C^m$, such that the action of $\U(m)$ on $\m$ is complex linear. For fixed $X,F_\bu$ with $F_\bu$ antilinear self-adjoint, let $Q\ra X$ be a principal $\U(m)$-bundle, and $R=(Q\t\SU(m+1))/\U(m)$ its $\SU(m+1)$-bundle. Then by Proposition \ref{os4prop2}, spin structures for $\B_R$ pull back to spin structures for~$\B_Q$.
\label{os4ex2}	
\end{ex}

\begin{ex} As in Example \ref{os2ex3}, we have an inclusion $\U(m_1)\t\U(m_2)\hookra\U(m_1+m_2)$ with an isomorphism $\m=\u(m_1+m_2)/(\u(m_1)\op\u(m_2))\cong\C^{m_1m_2}$, such that the action of $\U(m_1)\t\U(m_2)$ on $\m$ is complex linear. For fixed $X,F_\bu$ with $F_\bu$ antilinear self-adjoint, suppose $P_1\ra X$, $P_2\ra X$ are principal $\U(m_1)$- and $\U(m_2)$-bundles. Define a principal $\U(m_1+m_2)$-bundle $P_1\op P_2\ra X$ as in \eq{os2eq20}, and a morphism $\Phi_{P_1,P_2}:\B_{P_1}\t\B_{P_2}\longra\B_{P_1\op P_2}$ as in \eq{os2eq21}. Then combining the material of \S\ref{os43} for the product of $\U(m_1),\U(m_2)$ with the above, we have a natural isomorphism of line bundles on~$\B_{P_1}\t\B_{P_2}$:
\e
\begin{split}
&\phi^{F_\bu}_{P_1,P_2}=\La_{P_1,P_2}^*(\xi_{P_1\t_XP_2,P_1\op P_2}^{F_\bu})\ci\la_{P_1,P_2}^{F_\bu}:\\
&\quad (K_{P_1}^{F_\bu}\bt K_{P_2}^{F_\bu})\ot_\C (\La_{P_1,P_2}^*(L_{P_1\t_XP_2,\rho_\C}^{F_\bu}))^{\ot^2}\,{\buildrel\cong\over\longra}\,
\Phi_{P_1,P_2}^*(K_{P_1\op P_2}^{F_\bu}).
\end{split}
\label{os4eq10}
\e
Proposition \ref{os4prop1} now implies that a spin structure for $\B_{P_1\op P_2}$ determines unique spin structures for $\B_{P_1}$ and $\B_{P_2}$.
\label{os4ex3}	
\end{ex}

\subsection{\texorpdfstring{Stabilization and spin structures for $\U(m)$-bundles}{Stabilization and spin structures for U(m)-bundles}}
\label{os46}

Let $X$ be a compact, connected $n$-manifold, and $P_1\ra X$, $P_2\ra X$ be principal $\U(m_1)$- and $\U(m_2)$-bundles, so that Example \ref{os4ex3} defines a morphism $\Phi_{P_1,P_2}:\B_{P_1}\t\B_{P_2}\ra\B_{P_1\op P_2}$. Fix a connection $\nabla_{P_2}$ on $P_2$, and consider the morphism $\Phi_{P_1,P_2}(-,[\nabla_{P_2}]):\B_{P_1}\ra\B_{P_1\op P_2}$. Now as in Noohi \cite{Nooh2}, `hoparacompact' topological stacks (which include our $\B_P$) have a well-behaved homotopy theory, so we can consider their homotopy groups (which are just homotopy groups of an associated topological space called the `classifying space'). As in \cite[\S 2.3.3]{JTU}, by facts about `stabilization' of $\U(m)$-bundles, it is known that
\e
\begin{gathered}
\pi_k(\Phi_{P_1,P_2}(-,[\nabla_{P_2}])):\pi_k(\B_{P_1})\longra\pi_k(\B_{P_1\op P_2})\\
\text{is bijective if $0\le k\le 2m_1-n,$ and surjective if $k=2m_1-n+1$.}\\
\end{gathered}
\label{os4eq11}
\e
In other words, $\Phi_{P_1,P_2}(-,[\nabla_{P_2}])$ is an $(2m_1-n+1)$-equivalence as defined in \cite[Def.~6.68]{KiDa}. Following \cite[\S 2.3.3]{JTU}, the proof of \eqref{os4eq11} makes use of the `group completed' version of $\B_{P_1},$ where $\Phi_{P_1,P_2}(-,[\nabla_{P_2}])$ has a homotopy inverse and to which $\B_{P_1}$ and $\B_{P_1\op P_2}$ admit $(2m_1-n+1)$-connected and  $(2m_1+2m_2-n+1)$-connected maps.


Now if a morphism $f:X\ra Y$ of topological spaces or stacks is $3$-connected, so is bijective on $\pi_0,\pi_1, \pi_2$ and surjective on $\pi_3,$ then pullback $f^*$ induces an equivalence between the categories of topological complex line bundles on $Y$ and $X$. Indeed, as in \cite[Cor.~6.69]{KiDa}, the Hurewicz theorem implies the vanishing of the homology of the mapping cone of $f$ in degrees $0\leqslant *\leqslant 3.$ The long exact sequence in cohomology of the mapping cylinder then gives that $f^*\colon H^*(Y)\to H^*(X)$ is an isomorphism in degrees $0\leqslant *\leqslant 2.$ Hence our functor is essentially surjective, being bijective on isomorphism classes of complex line bundles. Similarly, our functor is fully faithful, using that the automorphism groups $\Aut(L)=[X,\C^*]$ are isomorphic to $H^1(X).$ We deduce:

\begin{prop}
\label{os4prop3}
Let\/ $X$ be a compact, connected\/ $n$-manifold and\/ $F_\bu$ be an antilinear self-adjoint elliptic operator on\/ $X.$ Suppose\/ $P_1\ra X,$\/ $P_2\ra X$ are principal\/ $\U(m_1)$- and\/ $\U(m_2)$-bundles, where\/ $2m_1\ge n+2$. Then the map defined in Example\/ {\rm\ref{os4ex3}} from spin structures on\/ $\B_{P_1\op P_2}$ to spin structures on\/ $\B_{P_1}$ is a {\rm 1--1} correspondence. In particular, a spin structure\/ $(K_{P_1}^{F_\bu})^{1/2}$ on\/ $\B_{P_1}$ determines a unique spin structure\/ $(K_{P_1\op P_2}^{F_\bu})^{1/2}$ on\/ $\B_{P_1\op P_2}$.
\end{prop}

As in \cite[\S 2.3.3 \& Prop.~2.24]{JTU}, if $P\ra X$ is a principal $\U(m)$-bundle with $2m\ge n+1$ then there is a canonical isomorphism $\pi_1(\B_P)\cong K^1(X)$, for $K^1(X)$ the odd complex K-theory group of $X$. Now two square roots of a complex line bundle $K_P^{F_\bu}\ra\B_P$ differ by a principal $\Z_2$-bundle on $\B_P$, and as $\B_P$ is connected, by taking monodromy, principal $\Z_2$-bundles are in 1-1 correspondence with group morphisms $\pi_1(\B_P)\ra\Z_2$. This yields:

\begin{prop}
\label{os4prop4}
Let\/ $X$ be a compact, connected\/ $n$-manifold and\/ $F_\bu$ be a complex elliptic operator on\/ $X$. Suppose\/ $P\ra X$ is a principal\/ $\U(m)$-bundle, where\/ $2m\ge n+1$. Then there is a canonical isomorphism\/ $\pi_1(\B_P)\cong K^1(X)$. If\/ $\B_P$ admits a spin structure\/ $(K_P^{F_\bu})^{1/2},$ then the family of all spin structures on\/ $\B_P$ is a torsor over\/ $\Hom(K^1(X),\Z_2)$.
\end{prop}

\section{The main results}
\label{os5}

\subsection[\texorpdfstring{Natural spin structures for all $\U(m)$- and $\SU(m)$-bundles}{Natural spin structures for all U(m)- and SU(m)-bundles}]{Natural spin structures for all $\U(m)$-, $\SU(m)$-bundles}
\label{os51}

The next theorem, proved in \S\ref{os6} using the material of \S\ref{os44}--\S\ref{os46}, shows that if $F_\bu$ is antilinear self-adjoint, and we can choose a spin structure $\hat\si_Q^{F_\bu}$ on $\B_Q$ for any {\it one\/} $\U(N)$-bundle $Q\ra X$ with $N\gg 0$ (for simplicity we take $Q=X\t\U(N)$ to be the trivial $\U(N)$-bundle), then we can construct natural spin structures on $\B_P$ for {\it all\/} $\U(m)$- or $\SU(m)$-bundles $P\ra X$. The analogue of Theorem \ref{os5thm1} does not work for orientations in \S\ref{os2} (though \cite[Prop.~2.24]{JTU} gives a partial analogue for {\it orientability\/} of moduli spaces $\B_P$), for the reasons in Remark~\ref{os4rem1}.

\begin{thm}
\label{os5thm1}
Let\/ $X$ be a compact, connected\/ $n$-manifold and\/ $F_\bu$ be an antilinear self-adjoint complex elliptic operator on\/ $X$. Suppose that for some\/ $N$ with\/ $2N\ge n+2,$ we are given a spin structure\/ $\hat\si_{X\t\U(N)}^{F_\bu}$ for the trivial principal\/ $\U(N)$-bundle\/ $X\t\U(N)\ra X$. Then, whenever\/ $P\ra X$ is a principal\/ $\U(m)$- or\/ $\SU(m)$-bundle for any\/ $m\ge 0,$ there is a unique spin structure\/ $\si_P^{F_\bu}$ on\/ $\B_P,$ such that the family of all\/ $\si_P^{F_\bu}$ satisfy:
\begin{itemize}
\setlength{\itemsep}{0pt}
\setlength{\parsep}{0pt}
\item[{\bf(a)}] $\si_{X\t\U(N)}^{F_\bu}=\hat\si_{X\t\U(N)}^{F_\bu}$.
\item[{\bf(b)}] Suppose\/ $P\ra X,$ $P'\ra X$ are principal\/ $\U(m)$- or $\SU(m)$-bundles and\/ $\rho:P\ra P'$ is an isomorphism. This induces isomorphisms\/ $\B_P\cong\B_{P'}$ and\/ $K_P^{F_\bu}\cong K_{P'}^{F_\bu},$ both of which are independent of the choice of\/ $\rho$. These identify the canonical spin structures\/~$\si_P^{F_\bu}\cong\si_{P'}^{F_\bu}$.
\item[{\bf(c)}] Let\/ $P_1\ra X,$ $P_2\ra X$ be principal\/ $\U(m_1)$- and\/ $\U(m_2)$-bundles, so that\/ $P_1\op P_2\ra X$ is a principal\/ $\U(m_1+m_2)$-bundle. Then Example\/ {\rm\ref{os4ex3}} defines a map from spin structures on\/ $\B_{P_1\op P_2}$ to pairs of spin structures on\/ $\B_{P_1}$ and\/ $\B_{P_2},$ and this should map\/ $\si_{P_1\op P_2}^{F_\bu}\mapsto(\si_{P_1}^{F_\bu},\si_{P_2}^{F_\bu})$.
\item[{\bf(d)}] Let\/ $P\ra X$ be a principal\/ $\SU(m)$-bundle, and\/ $Q=(P\t\U(1))/\Z_m$ be the associated principal\/ $\U(m)$-bundle. Then Example\/ {\rm\ref{os4ex1}} maps spin structures on\/ $\B_Q$ to spin structures on\/ $\B_P,$ and this should map\/ $\si_Q^{F_\bu}\mapsto\si_P^{F_\bu}$.
\end{itemize}
The family of possible choices for\/ $\hat\si_{X\t\U(N)}^{F_\bu}$ is a torsor over\/~$\Hom(K^1(X),\Z_2)$.
\end{thm}

\subsection{Spin structures, orientations, and loop spaces}
\label{os52}

The next two definitions set up notation for Theorems \ref{os5thm2} and~\ref{os5thm3}.

\begin{dfn}
\label{os5def1}
Let $X$ be a compact, connected $n$-manifold, and write the circle $\cS^1$ as $\cS^1=\bigl\{e^{i\th}:\th\in[0,2\pi)\bigr\}\subset\C$, so $\th$ is a (periodic) coordinate on $\cS^1$, and $1\in\cS^1$ is a basepoint. Then $X\t\cS^1$ is a compact, connected $(n+1)$-manifold.

Let $G$ be a Lie group, and $P\ra X$, $Q\ra X\t\cS^1$ be principal $G$-bundles with an isomorphism $Q\vert_{X\t\{1\}}\cong P$ over $X\t\{1\}\cong X$. Then we also have (non-canonical) isomorphisms $Q\vert_{X\t\{e^{i\th}\}}\cong P$ for all $e^{i\th}\in\cS^1$.

For each $e^{i\th}\in\cS^1$, define a morphism of topological stacks $\Ga_{Q,P}^{e^{i\th}}:\B_Q\ra\B_P$ to map $\Ga_{Q,P}^{e^{i\th}}:[\nabla_Q]\mapsto[\nabla_Q\vert_{X\t\{e^{i\th}\}}]$. Here $\nabla_Q$ is a connection on $Q\ra X\t\cS^1$ (modulo the gauge group $\G_Q$), and $\nabla_Q\vert_{X\t\{e^{i\th}\}}$ is its restriction to $Q\vert_{X\t\{e^{i\th}\}}\ra X\t\{e^{i\th}\}\cong X$. Since $Q\vert_{X\t\{e^{i\th}\}}\cong P$, we can consider $\nabla_Q\vert_{X\t\{e^{i\th}\}}$ as a connection on $P$, so that $[\nabla_Q\vert_{X\t\{e^{i\th}\}}]\in\B_P$. Here as the definition of $\B_P$ quotients out by $\G_P=\Aut(P)$, the morphism $\Ga_{Q,P}^{e^{i\th}}:\B_Q\ra\B_P$ is independent of the choice of isomorphism~$Q\vert_{X\t\{e^{i\th}\}}\cong P$.

As $\Ga_{Q,P}^{e^{i\th}}$ depends continuously on $e^{i\th}$, we have a morphism of topological stacks $\Ga_{Q,P}:\B_Q\ra\Map_{C^0}(\cS^1,\B_P)$, taking $\Ga_{Q,P}:[\nabla_Q]\mapsto\bigl(e^{i\th}\mapsto\Ga_{Q,P}^{e^{i\th}}([\nabla_Q])\bigr)$. Here $\Map_{C^0}(\cS^1,\B_P)$ is the {\it free loop space\/} of $\B_P$.

Also $P\t\cS^1\ra X\t\cS^1$ is a principal $G$-bundle, and there is a natural morphism $\Pi_P:\B_P\ra\B_{P\t\cS^1}$ mapping $\Pi_P:[\nabla_P]\mapsto[\pi_X^*(\nabla_P)]$. Hence $\Nu_{Q,P}=\Pi_P\ci \Ga_{Q,P}^1:\B_Q\ra\B_{P\t\cS^1}$ is a morphism of topological stacks.
\end{dfn}

\begin{dfn}
\label{os5def2}	
Let $X$ be a compact, connected manifold, and suppose that $F_\bu=(F_0,F_1,D)$ is a first order antilinear self-adjoint linear elliptic operator on $X$ as in Definition \ref{os3def1}, so that $F_0=\bar F_1$. Write $F\ra X$ for the real vector bundle underlying both $F_0$ and $F_1$, and write $i_{F_0},i_{F_1}:\Ga^\iy(F)\ra\Ga^\iy(F)$ for the complex structures on $F_0,F_1$, so that $i_{F_0}=-i_{F_1}$. Define a real vector bundle $E\ra X\t\cS^1$ by $E=\pi_X^*(F)$, where $\pi_X:X\t\cS^1\ra X$ is the projection. 

Define a partial differential operator $\ti D:\Ga^\iy(E)\ra\Ga^\iy(E)$ on $X\t\cS^1$ by
\e
\ti D=\pi_X^*(D)+\pi_X^*(i_{F_0})\frac{\pd}{\pd\th}.
\label{os5eq1}
\e
We claim that $\ti D$ is a self-adjoint real linear elliptic operator. Self-adjointness holds since $D$ is self-adjoint as a real operator, and $\pi_X^*(i_{F_0})$, $\frac{\pd}{\pd\th}$ are both anti-self-adjoint and commute. 
For ellipticity, note that for $\xi \in T^*X$ and $x \in \R$
\begin{equation*}
 \si_{\xi+xd\th}(\ti{D}^*\ti{D}) = x^2 + \si_\xi(D)^*\si_\xi(D)
\end{equation*}


Write $E_\bu=(E,E,\ti D)$ as in Definition \ref{os2def2}. Now let $G,P,Q,P\t\cS^1$ and $\Ga_{Q,P}:\B_Q\ra\Map_{C^0}(\cS^1,\B_P)$ be as in Definition \ref{os5def1}. Then using $E_\bu$ we have principal $\Z_2$-bundles $\check O_Q^{E_\bu}\ra\B_Q$, $\check O_{P\t\cS^1}^{E_\bu}\ra\B_{P\t\cS^1}$ and notions of n-orientation $\check\om_Q^{E_\bu},\check\om_{P\t\cS^1}^{E_\bu}$ on $\B_Q$ and $\B_{P\t\cS^1}$, as in \S\ref{os22}. Also using $F_\bu$ we have a complex line bundle $K_P^{F_\bu}\ra\B_P$, and a notion of spin structure $\si_P^{F_\bu}=(K_P^{F_\bu})^{1/2}$ on $\B_P$, as in~\S\ref{os31}.

Let $\ga:\cS^1\ra\B_P$ lie in $\Map_{C^0}(\cS^1,\B_P)$. Then $\ga^*(K_P^{F_\bu})\ra\cS^1$ is a complex line bundle, so we can consider its square roots. Every complex line bundle $L\ra\cS^1$ admits square roots. Two square roots of $L$ differ up to isomorphism by a principal $\Z_2$-bundle over $\cS^1$. As there are two principal $\Z_2$-bundles on $\cS^1$ up to isomorphism, with monodromy $1$ and $-1$, any complex line bundle $L\ra\cS^1$ has exactly two square roots $L^{1/2}$ up to isomorphism.

Define a principal $\Z_2$-bundle $M^{F_\bu}_P\ra \Map_{C^0}(\cS^1,\B_P)$ to have fibre $M^{F_\bu}_P\vert_\ga$ over each $\ga\in\Map_{C^0}(\cS^1,\B_P)$ the set of two isomorphism classes of square roots of $\ga^*(K_P^{F_\bu})\ra\cS^1$, where the $\Z_2$-action exchanges the two isomorphism classes. Since $\ga^*(K_P^{F_\bu})$ varies continuously with $\ga$, this is well-defined.

Now suppose $P_1,P_2\ra X$ are principal $\U(m_1)$- and $\U(m_2)$-bundles, so that $P_1\op P_2\ra X$ is a principal $\U(m_1+m_2)$-bundle. In \eq{os2eq21} we defined a morphism $\Phi_{P_1,P_2}:\B_{P_1}\t\B_{P_2}\ra\B_{P_1\op P_2}$. Define a morphism
\e
\begin{split}
&\Chi_{P_1,P_2}:\Map_{C^0}(\cS^1,\B_{P_1})\t\Map_{C^0}(\cS^1,\B_{P_2})\longra
\Map_{C^0}(\cS^1,\B_{P_1\op P_2}) \\
&\text{by}\qquad \Chi_{P_1,P_2}:(\ga_1,\ga_2)\longmapsto \Phi_{P_1,P_2}\ci(\ga_1,\ga_2).
\end{split}
\label{os5eq2}
\e
Let the isomorphism $\phi^{F_\bu}_{P_1,P_2}$ of line bundles on $\B_{P_1}\t\B_{P_2}$ be as in \eq{os4eq10}. Define an isomorphism of principal $\Z_2$-bundles on $\Map_{C^0}(\cS^1,\B_{P_1})\t\Map_{C^0}(\cS^1,\B_{P_2})$:
\e
\chi_{P_1,P_2}^{F_\bu}:M^{F_\bu}_{P_1}\bt_{\Z_2}M^{F_\bu}_{P_2}\longra\Chi_{P_1,P_2}^*(M^{F_\bu}_{P_1\op P_2})
\label{os5eq3}
\e
by, for each $(\ga_1,\ga_2)\in \Map_{C^0}(\cS^1,\B_{P_1})\t\Map_{C^0}(\cS^1,\B_{P_2})$, using the isomorphism
\begin{align*}
&(\ga_1,\ga_2)^*(\phi^{F_\bu}_{P_1,P_2}):\ga_1^*(K_{P_1}^{F_\bu})\ot_\C\ga_2^*(K_{P_2}^{F_\bu})\ot_\C (\La_{P_1,P_2}\ci(\ga_1,\ga_2))^*(L_{P_1\t_XP_2,\rho_\C}^{F_\bu}))^{\ot^2}\\
&\qquad\qquad{\buildrel\cong\over\longra}\,
(\Phi_{P_1,P_2}\ci(\ga_1,\ga_2))^*(K_{P_1\op P_2}^{F_\bu})=\Chi_{P_1,P_2}(\ga_1,\ga_2)^*(K_{P_1\op P_2}^{F_\bu})\end{align*}
of complex line bundles on $\cS^1$ to map square roots $\ga_1^*(K_{P_1}^{F_\bu})^{1/2}$, $\ga_2^*(K_{P_2}^{F_\bu})$ to a square root~$\Chi_{P_1,P_2}(\ga_1,\ga_2)^*(K_{P_1\op P_2}^{F_\bu})^{1/2}$.
\end{dfn}
	
The following theorem will be proved in~\S\ref{os7}.

\begin{thm}
\label{os5thm2}
In the situation of Definitions\/ {\rm\ref{os5def1}} and\/ {\rm\ref{os5def2},} there is a natural isomorphism of principal\/ $\Z_2$-bundles over\/ $\B_Q\!:$
\e
\ga_{Q,P}^{F_\bu}:\check O_Q^{E_\bu}\ot_{\Z_2}\Nu_{Q,P}^*(\check O_{P\t\cS^1}^{E_\bu})\longra\Ga_{Q,P}^*(M^{F_\bu}_P).
\label{os5eq4}
\e
This has the following properties:
\begin{itemize}
\setlength{\itemsep}{0pt}
\setlength{\parsep}{0pt}
\item[{\bf(i)}] Take\/ $Q=P\t\cS^1,$ and pull back\/ \eq{os5eq4} by\/ $\Pi_P:\B_P\ra\B_{P\t\cS^1},$ giving
\e
\begin{split}
&\Pi_P^*(\ga_{P\t\cS^1,P}^{F_\bu}):
\Pi_P^*(\check O_{P\t\cS^1}^{E_\bu})\ot_{\Z_2}(\Nu_{P\t\cS^1,P}\ci\Pi_P)^*(\check O_{P\t\cS^1}^{E_\bu})\\
&\qquad\qquad\longra(\Ga_{P\t\cS^1,P}\ci\Pi_P)^*(M^{F_\bu}_P).
\end{split}
\label{os5eq5}
\e
As\/ $\Nu_{P\t\cS^1,P}\ci\Pi_P=\Pi_P,$ the left hand side is the square of\/ $\Pi_P^*(\check O_{P\t\cS^1}^{E_\bu}),$ and so is canonically trivial. On the right,\/ $\Ga_{P\t\cS^1,P}\ci\Pi_P$ maps each point\/ $[\nabla_P]$ to the constant loop\/ $\ga_{[\nabla_P]}$ at\/ $[\nabla_P]$. Now\/ $\ga_{[\nabla_P]}^*(K_P^{F_\bu})$ is the trivial line bundle on\/ $\cS^1$ with fibre\/ $K_P^{F_\bu}\vert_{[\nabla_P]},$ and so it has a canonical trivial square root up to isomorphism, which is a point of\/ $(\Ga_{P\t\cS^1,P}\ci\Pi_P)^*\ab(M^{F_\bu}_P)\vert_{[\nabla_P]}$. Thus the right hand side of\/ \eq{os5eq5} is also canonically trivial. Equation \eq{os5eq5} identifies these canonical trivializations.
\item[{\bf(ii)}] Suppose\/ $Q_1,Q_2\ra X\t\cS^1$ and\/ $P_1,P_2\ra X$ are principal\/ {\rm$\U(m_1)$-, $\U(m_2)$-}bundles with\/ $P_a\cong Q_a\vert_{X\t\{1\}}$ for\/ $a=1,2,$ so that\/ $Q_1\op Q_2\ra X\t\cS^1$ and\/ $P_1\op P_2\ra X$ are principal\/ $\U(m_1+m_2)$-bundles. Then the following diagram of principal\/ $\Z_2$-bundles over\/ $\B_{Q_1}\t\B_{Q_2}$ commutes:
\e
\begin{gathered}
\!\!\!\!\!\!\!\!
\xymatrix@!0@C=270pt@R=47pt{
& *+[l]{\begin{subarray}{l}\ts \Ga_{Q_1,P_1}^*(M^{F_\bu}_{P_1})\bt \Ga_{Q_2,P_2}^*(M^{F_\bu}_{P_2})\\
\ts =(\Ga_{Q_1,P_1}\t\Ga_{Q_2,P_2})^*(M^{F_\bu}_{P_1}\bt M^{F_\bu}_{P_2}) \end{subarray}}
 \ar[dd]_{(\Ga_{Q_1,P_1}\t\Ga_{Q_2,P_2})^*(\chi_{P_1,P_2}^{F_\bu})} 
\\
*+[r]{\begin{subarray}{l}\ts
 \bigl(\check O_{Q_1}^{E_\bu}\!\ot\!\Nu_{Q_1,P_1}^*(\check O_{P_1\t\cS^1}^{E_\bu})\bigr)\bt \\ \ts\;\> \bigl(\check O_{Q_2}^{E_\bu}\!\ot\!\Nu_{Q_2,P_2}^*(\check O_{P_2\t\cS^1}^{E_\bu})\bigr)\end{subarray}} \ar[dd]^{\begin{subarray}{l}\check\phi_{Q_1,Q_2}^{E_\bu}\ot 
(\Nu_{Q_1,P_1}\t\Nu_{Q_2,P_2})^*\\ (\check\phi^{E_\bu}_{P_1\t\cS^1,P_2\t\cS^1})\end{subarray}}
\ar@/^1pc/[ur]^(0.4){\ga_{Q_1,P_1}^{F_\bu}\bt\ga_{Q_2,P_2}^{F_\bu}\qquad}
\\
& *+[l]{\begin{subarray}{l}
\ts (\Ga_{Q_1,P_1}\t\Ga_{Q_2,P_2})^*\ci\Chi_{P_1,P_2}^*(M^{F_\bu}_{P_1\op P_2})\\
\ts =(\Ga_{Q_1\op Q_2,P_1\op P_2}\ci\Phi_{Q_1,Q_2})^*(M^{F_\bu}_{P_1\op P_2})
\end{subarray}}
\\
*+[r]{\begin{subarray}{l}\ts\Phi_{Q_1,Q_2}^*\bigl(\check O_{Q_1\op Q_2}^{E_\bu}\!\ot\! \Nu_{Q_1\op Q_2,P_1\op P_2}^*(\check O_{(P_1\op P_2)\t\cS^1}^{E_\bu})\bigr) \\
\ts =\Phi_{Q_1,Q_2}^*(\check O_{Q_1\op Q_2}^{E_\bu})\!\ot\! (\Nu_{Q_1,P_1}\!\t\!\Nu_{Q_2,P_2})^*\Phi_{P_1\t\cS^1,P_2\t\cS^1}^*(\check O_{(P_1\op P_2)\t\cS^1}^{E_\bu}).\end{subarray}} 
\ar@/_1pc/[ur]_(0.77){\qquad\Phi_{Q_1,Q_2}^*(\ga_{Q_1\op Q_2,P_1\op P_2}^{F_\bu})} }\!\!\!\!\!\!\!
\end{gathered}
\label{os5eq6}
\e
Here\/ $\check\phi_{Q_1,Q_2}^{E_\bu},\check\phi^{E_\bu}_{P_1\t\cS^1,P_2\t\cS^1}$ and\/ $\chi_{P_1,P_2}^{F_\bu}$ are as in \eq{os2eq22} and\/~\eq{os5eq3}.
\end{itemize}
\end{thm}

\begin{rem}
\label{os5rem1}	
{\bf(a)} The isomorphisms $\ga_{Q,P}^{F_\bu}$ satisfy more consistency conditions than Theorem \ref{os5thm2}(i),(ii) for Lie groups $G$ other than $\U(m)$. In particular, for any fixed $G$ there are complicated conditions on $\ga_{Q,P}^{F_\bu}$ relating to composing loops in $\Map_{C^0}(\cS^1,\B_P)$ with a common base point, which we will not explain, but which we would need if we wanted to prove an analogue of Theorem \ref{os5thm3} below for $G$-bundles. For $G=\U(m)$, these are implied by Theorem~\ref{os5thm2}(i),(ii).

\smallskip

\noindent{\bf(b)} There is also an analogue of Theorem \ref{os5thm2}(ii) for $\SU(m)$-bundles rather than $\U(m)$-bundles, involving analogues of Examples \ref{os2ex3} and \ref{os4ex3} for the inclusion $\SU(m_1)\t\SU(m_2)\hookra\SU(m_1+m_2)$, but we will not state it. Along these lines, there is also an analogue of Theorem \ref{os5thm3} below for $\SU(m)$-bundles.
\end{rem}

The next theorem will be proved in \S\ref{os8}. It is a useful tool for constructing spin structures on moduli spaces $\B_P$ for $P\ra X$ a $\U(m)$- or $\SU(m)$-bundle.

\begin{thm}
\label{os5thm3}
In the situation of Definitions\/ {\rm\ref{os5def1}} and\/ {\rm\ref{os5def2},} there are natural\/ {\rm 1-1} correspondences between the following choices of data:
\begin{itemize}
\setlength{\itemsep}{0pt}
\setlength{\parsep}{0pt}
\item[{\bf(a)}] A spin structure\/ $\hat\si_{X\t\U(N)}^{F_\bu}$ on\/ $\B_{X\t\U(N)}$ for the trivial principal\/ $\U(N)$-bundle\/ $X\t\U(N)\ra X,$ for some fixed\/ $N$ with\/~$2N\ge n+2$.
\item[{\bf(b)}] Spin structures\/ $\si_P^{F_\bu}$ on\/ $\B_P$ for all principal\/ $\U(m)$-bundles\/ $P\ra X$ for all\/ $m\ge 0,$ such that if\/ $P_1\ra X,$ $P_2\ra X$ are principal\/ $\U(m_1)$- and\/ $\U(m_2)$-bundles, then Example\/ {\rm\ref{os4ex3}} maps from spin structures on\/ $\B_{P_1\op P_2}$ to pairs of spin structures on\/ $\B_{P_1},\B_{P_2},$ and this maps\/ $\si_{P_1\op P_2}^{F_\bu}\mapsto(\si_{P_1}^{F_\bu},\si_{P_2}^{F_\bu})$.
\item[{\bf(c)}] Trivializations\/ $\check\Om^{E_\bu}_Q:\check O_Q^{E_\bu}\ot_{\Z_2}\Nu_{Q,P}^*(\check O_{P\t\cS^1}^{E_\bu})\cong\B_Q\t\Z_2$ whenever\/ $P\ra X,$ $Q\ra X\t\cS^1$ are principal\/ $\U(m)$-bundles with an isomorphism\/ $Q\vert_{X\t\{1\}}\cong P,$ for all\/ $m\ge 0,$ such that:
\begin{itemize}
\setlength{\itemsep}{0pt}
\setlength{\parsep}{0pt}
\item[{\bf(i)}] Take\/ $Q=P\t\cS^1,$ and pull back\/ $\check\Om^{E_\bu}_{P\t\cS^1}$ by\/ $\Pi_P:\B_P\ra\B_{P\t\cS^1},$ giving
\e
\begin{split}
&\Pi_P^*(\check\Om^{E_\bu}_{P\t\cS^1}):
\Pi_P^*(\check O_{P\t\cS^1}^{E_\bu})\ot_{\Z_2}(\Nu_{P\t\cS^1,P}\ci\Pi_P)^*(\check O_{P\t\cS^1}^{E_\bu})\\
&\qquad\qquad\longra\Pi_P^*(\B_{P\t\cS^1}\t\Z_2)\cong \B_P\t\Z_2.
\end{split}
\label{os5eq7}
\e
As\/ $\Nu_{P\t\cS^1,P}\ci\Pi_P=\Pi_P,$ the left hand side is\/ $\Pi_P^*(\check O_{P\t\cS^1}^{E_\bu})^{\ot^2},$ and so is canonically trivial, as is the right hand side. Then\/ $\Pi_P^*(\check\Om^{E_\bu}_{P\t\cS^1})$ identifies the canonical trivializations.
\item[{\bf(ii)}] Given\/ $P_1,P_2\ra X,$ $Q_1,Q_2\ra X\t\cS^1,$ the isomorphism\/ $\check\phi_{Q_1,Q_2}^{E_\bu}\ot (\Nu_{Q_1,P_1}\t\Nu_{Q_2,P_2})^*(\check\phi^{E_\bu}_{P_1\t\cS^1,P_2\t\cS^1})$ in \eq{os5eq6} identifies the trivializations induced by\/ $\check\Om^{E_\bu}_{Q_1},\check\Om^{E_\bu}_{Q_2}$ and\/~$\check\Om^{E_\bu}_{Q_1\op Q_2}$.
\end{itemize}
\item[{\bf(d)}] The same as\/ {\bf(c)\rm,} but restricted to\/ $P,Q$ such that\/ $P=X\t\U(m)$ is trivial.
\end{itemize}

Note that n-orientations on\/ $\B_Q,\B_{P\t\cS^1}$ induce trivializations\/ $\check\Om^{E_\bu}_Q$ in\/ {\bf(c)}. Thus, if we can construct canonical n-orientations on\/ $\B_Q$ for all\/ $\U(m)$-bundles\/ $Q\ra X\t\cS^1$ compatible with\/ $\check\phi_{Q_1,Q_2}^{E_\bu}$ in\/ \eq{os2eq22} \textup(see Theorem\/ {\rm\ref{os2thm1}} and\/ {\rm\cite{JTU,JoUp1}} for results of this kind\textup) then by\/ {\bf(b)} we obtain canonical spin structures\/ $\si_P^{F_\bu}$ on\/ $\B_P$ for all principal\/ $\U(m)$-bundles\/ $P\ra X$. Example\/ {\rm\ref{os4ex1}} then also gives canonical spin structures\/ $\si_Q^{F_\bu}$ on\/ $\B_Q$ for all principal\/ $\SU(m)$-bundles\/ $Q\ra X$.
\end{thm}

Here (a),(b) are equivalent by Theorem \ref{os5thm1}. We can map from data (b) to data (c) using Theorem \ref{os5thm2}. If we are given spin structures $\si_P^{F_\bu}$ as in (b), then for each $\ga\in\Map_{C^0}(\cS^1,\B_P)$, $\ga^*(\si_P^{F_\bu})$ is a square root of $\ga^*(K_P^{F_\bu})$, and thus a point of $M^{F_\bu}_P\vert_\ga$. This defines a trivialization of $M^{F_\bu}_P\ra \Map_{C^0}(\cS^1,\B_P)$. Hence \eq{os5eq4} induces a trivialization of $\check O_Q^{E_\bu}\ot_{\Z_2}\Nu_{Q,P}^*(\check O_{P\t\cS^1}^{E_\bu})$, as in (c). Theorem \ref{os5thm2}(i),(ii) imply (c)(i),(ii). Clearly data (c) restricts to data (d). The point of the proof of Theorem \ref{os5thm3} is to show these maps (b) $\Ra$ (c) $\Ra$ (d) are bijections.

\subsection{Background on flag structures on 7-manifolds}
\label{os53}

We recall some material from the authors \cite{Joyc3,JoUp1}, starting with~\cite[\S 3.1]{Joyc3}. 

\begin{dfn}
\label{os5def3}
Let $Y$ be an oriented 7-manifold, and consider pairs $(Z,s)$ of a compact, oriented, immersed 3-submanifold $Z\hookra Y$, and a non-vanishing section $s$ of the normal bundle $N_Z$ of $Z$ in $Y$.
We call $(Z,s)$ a {\it flagged submanifold\/} in $Y$. For non-vanishing sections $s, s'$ of $N_Z$ define
\begin{equation*}
d(s,s') \coloneqq Z \bullet \bigl\{ t\cdot s(y) + (1-t)\cdot s'(y) \enskip\big|\enskip t\in [0,1],\enskip y\in Z\bigr\} \in \Z,
\end{equation*}
using the intersection product `$\bu$' between a $3$-cycle and a $4$-chain
whose boundary does not meet the cycle, see Dold~\cite[(13.20)]{Dold}.

Let $(Z_1,s_1),(Z_2,s_2)$ be disjoint flagged submanifolds with $[Z_1]=[Z_2]$ in $H_3(Y,\Z)$. Choose an integral 4-chain $C$ with $\pd C=Z_2-Z_1$. Let $Z_1',Z_2'$ be small perturbations of $Z_1,Z_2$ in the normal directions $s_1,s_2$. Then $Z_1'\cap Z_1=Z_2'\cap Z_2=\es$ as $s_1,s_2$ are non-vanishing, and $Z_1'\cap Z_2=Z_2'\cap Z_1=\es$ as $Z_1,Z_2$ are disjoint and $Z_1',Z_2'$ are close to $Z_1,Z_2$. Define $D((Z_1,s_1),(Z_2,s_2))$ to be the intersection number $(Z_2'-Z_1')\bu C$ in homology over $\Z$. Here we regard
\begin{equation*}
[C] \in H_4(Y,Z_1\cup Z_2;\Z),\qquad
[Z'_1], [Z'_2] \in H_3(Z'_1\cup Z'_2,\emptyset;\Z).
\end{equation*}
Note that since $Z'_1, Z'_2$ are small perturbations and $Z_1, Z_2$ are disjoint we have $(Z_1 \cup Z_2) \cap (Z'_1 \cup Z'_2) = \emptyset$.
This is independent of the choices of $C$ and~$Z_1',Z_2'$.
\end{dfn}

In \cite[Prop.s~3.3 \& 3.4]{Joyc3} we show that if $(Z_1,s_1),(Z_2,s_2),\ab(Z_3,s_3)$ are disjoint flagged submanifolds with $[Z_1]=[Z_2]=[Z_3]$ in $H_3(Y,\Z)$ then
\e
\label{os5eq8}
\begin{split}
D((Z_1,s_1),(Z_3,s_3))\!\equiv\! D((Z_1,s_1),(Z_2,s_2))\!+\!D((Z_2,s_2),(Z_3,s_3))\!\mod\! 2,
\end{split}
\e
and if $(Z',s')$ is any small deformation of $(Z,s)$ with $Z,Z'$ disjoint then
\e
D((Z,s),(Z',s'))\equiv 0\mod 2.
\label{os5eq9}
\e

\begin{dfn}
\label{os5def4}
A {\it flag structure\/} on $Y$ is a map
\begin{equation*}
F:\bigl\{\text{flagged submanifolds $(Z,s)$ in $Y$}\bigr\}\longra\{\pm 1\},
\end{equation*}
satisfying:
\begin{itemize}
\setlength{\itemsep}{0pt}
\setlength{\parsep}{0pt}
\item[(i)] $F(Z,s) = F(Z,s')\cdot (-1)^{d(s,s')}$.
\item[(ii)] If $(Z_1,s_1),(Z_2,s_2)$ are disjoint flagged submanifolds in $Y$ with $[Z_1]=[Z_2]$ in $H_3(Y,\Z)$ then
\e
F(Z_2,s_2)=F(Z_1,s_1)\cdot (-1)^{D((Z_1,s_1),(Z_2,s_2))}.
\label{os5eq10}
\e
This is a well behaved condition by \eq{os5eq8}--\eq{os5eq9}.
\item[(iii)] If $(Z_1,s_1),(Z_2,s_2)$ are disjoint flagged submanifolds then
\end{itemize}
\e
F(Z_1\amalg Z_2,s_1\amalg s_2)=F(Z_1,s_1)\cdot F(Z_2,s_2).
\label{os5eq11}
\e
\end{dfn}

Here is \cite[Prop.~3.6]{Joyc3}:

\begin{prop}
\label{os5prop1}
Let\/ $Y$ be an oriented\/ $7$-manifold. Then:
\begin{itemize}
\setlength{\itemsep}{0pt}
\setlength{\parsep}{0pt}
\item[{\bf(a)}] There exists a flag structure\/ $F$ on\/ $Y$.
\item[{\bf(b)}] If\/ $F,F'$ are flag structures on\/ $Y$ then there exists a unique group morphism\/ $H_3(Y,\Z)\ra\{\pm 1\},$ denoted\/ $F'/F,$ such that
\e
F'(Z,s)=F(Z,s)\cdot (F'/F)[Z] \qquad\text{for all\/ $(Z,s)$.}
\label{os5eq12}
\e
\item[{\bf(c)}] Let\/ $F$ be a flag structure on\/ $Y$ and\/ $\ep:H_3(Y,\Z)\ra\{\pm 1\}$ a morphism, and define\/ $F'$ by \eq{os5eq12} with\/ $F'/F=\varepsilon$. Then\/ $F'$ is a flag structure on\/ $Y$.
\end{itemize}
Hence the set of flag structures on\/ $Y$ is a torsor over\/ $\Hom\bigl(H_3(Y,\Z),\Z_2\bigr)$.
\end{prop}

The first author introduced flag structures to define canonical orientations on moduli spaces of associative 3-folds in compact 7-manifolds $(Y,\vp,g)$ with holonomy $G_2$, \cite[\S 3.2]{Joyc3}. He also conjectured \cite[Conj.~8.3]{Joyc3} that flag structures would be necessary to define orientations on moduli spaces of $G_2$-instantons on $(Y,\vp,g)$. This was proved by the authors in \cite{JoUp1}. 
Here is a version of~\cite[Th.~1.2]{JoUp1}:

\begin{thm}
\label{os5thm4}
Suppose\/ $(Y,g)$ is a compact, oriented, spin Riemannian\/ $7$-manifold, and take\/ $E_\bu$ to be the Dirac operator\/ $\slashed{D}:\Ga^\iy(S)\ra\Ga^\iy(S)$ on\/ $Y$. Fix a flag structure on\/ $Y,$ as in Definition\/ {\rm\ref{os5def4}}. Then for any principal\/ $\U(m)$- or $\SU(m)$-bundle\/ $P\ra Y,$ we can construct a canonical n-orientation\/ $\check\om_P^{E_\bu}$ on\/ $\B_P$.

Here if\/ $P_1,P_2\ra Y$ are\/ {\rm$\SU(m_1)$-, $\SU(m_2)$-}bundles then the n-orientations\/ $\check\om_{P_1}^{E_\bu},\check\om_{P_2}^{E_\bu},\check\om_{P_1\op P_2}^{E_\bu}$ on\/ $\B_{P_1},\B_{P_2},\B_{P_1\op P_2}$ are compatible by the analogue of Example\/ {\rm\ref{os2ex3}}. The same holds if\/ $P_1,P_2\ra Y$ are {\rm$\U(m_1)$-, $\U(m_2)$-}bundles with\/ $c_1(P_1)=c_1(P_2)=0$ in\/ $H^2(Y,\Z)$. However, for general\/ {\rm$\U(m_1)$-, $\U(m_2)$-}bundles\/ $P_1,P_2,$\/ $\check\om_{P_1}^{E_\bu},\check\om_{P_2}^{E_\bu},\check\om_{P_1\op P_2}^{E_\bu}$ may not be compatible under Example\/~{\rm\ref{os2ex3}}.
\end{thm}

The reason for the last part is that the hard work in \cite{JoUp1} is to construct  n-orientations $\check\om_P^{E_\bu}$ for all $\SU(m)$-bundles $P\ra Y$, which are compatible with direct sums. Then n-orientations $\check\om_Q^{E_\bu}$ for $\U(m)$-bundles $Q\ra Y$ are induced by Example \ref{os2ex2}. However, Example \ref{os2ex2} does not commute with direct sums.

\subsection{\texorpdfstring{Canonical spin structures for $\slashed{D}_+$ on a 6-manifold}{Canonical spin structures for Dᕀ on a 6-manifold}}
\label{os54}

\begin{prop}
\label{os5prop2}
For every oriented\/ $6$-manifold\/ $X$ we can define a canonical
flag structure\/ $F_{\rm can}$ on\/ $Y=X\t\cS^1$.
\end{prop}

\begin{proof} We first prescribe $F_{\rm can}(Z,s)$ for flagged submanifolds $(Z,s)$ in $X\t\cS^1$ of two special kinds:
\begin{itemize}
\setlength{\itemsep}{0pt}
\setlength{\parsep}{0pt}
\item[(a)] Let $N\hookra X$ be a compact, oriented, immersed 3-submanifold of $X$. Then $Z=N\t\{1\}$ is a 3-submanifold in $X\t\cS^1$, and $s=\frac{\pd}{\pd\th}$ is a normal vector field to $Z$ in the $\cS^1$-direction in $X\t\cS^1$, where $\th$ is the local coordinate on $\cS^1\ni e^{i\th}$. We require that~$F_{\rm can}(N\t\{1\},\frac{\pd}{\pd\th})=1$.
\item[(b)] Let $\Si\hookra X$ be a compact, oriented, immersed 2-submanifold of $X$. Then $Z=\Si\t\cS^1$ is a 3-submanifold in $X\t\cS^1$. Let $t\in\Ga^\iy(N_\Si)$ be a nonvanishing normal vector field to $\Si$ in $X$. Then $s=\pi_\Si^*(t)$ is a nonvanishing normal vector field to $\Si\t\cS^1$ in $X\t\cS^1$, and $Z,s$ are invariant under the obvious action of $\U(1)=\cS^1$ on $X\t\cS^1$. We require that~$F_{\rm can}(\Si\t\cS^1,\pi_\Si^*(t))=1$.
\end{itemize}

We claim that there is a unique flag structure $F_{\rm can}$ on $Y=X\t\cS^1$ satisfying (a),(b). To see this, note as in the proof of \cite[Prop.~3.6]{Joyc3} that a flag structure on $Y$ is determined by its values on a set of flags $(Z,s)$ whose homology classes $[Z]\in H_3(Y,\Z)$ generate $H_3(Y,\Z)$. For $Y=X\t\cS^1$ we have $H_3(Y,\Z)\cong H_3(X,\Z)\op H_2(X,\Z)$, where the homology classes of flags $(Z,s)$ of types (a) and (b) generate $H_3(X,\Z)$ and $H_2(X,\Z)$ respectively. Therefore there exists a unique flag structure $F_{\rm can}$ satisfying (a),(b) if and only if the prescribed values (a),(b) are consistent with Definition \ref{os5def4}(i)--(iii) for flags of type~(a),(b).

For consistency with Definition \ref{os5def4}(i), in (b) suppose $t,t'\in\Ga^\iy(N_\Si)$ are nonvanishing normal vector fields to $\Si$ in $X$, and $s=\pi_\Si^*(t)$, $s'=\pi_\Si^*(t')$. As $\dim\Si=2$ and $\rank N_\Si=4$, we can choose a smooth family $t_a:a\in[0,1]$ of nonvanishing normal vector fields to $\Si$ in $X$ interpolating between $t_0=t$ and $t_1=t'$. Using this we can show that $d(s,s')=0$, so $(\Si\t\cS^1,\pi_\Si^*(t))$ and $(\Si\t\cS^1,\pi_\Si^*(t'))$ satisfy~(i).

For consistency with (ii), if $(Z_1,s_1)$, $(Z_2,s_2)$ are homologous flags of type (a), or type (b), by choosing a chain in $X$ with boundary $N_1-N_2$ or $\Si_1-\Si_2$ we can show that $D((Z_1,s_1),(Z_2,s_2))=0$, so \eq{os5eq10} holds. Consistency with (iii) is automatic as if $(Z_1,s_1),(Z_2,s_2)$ and $(Z_1\amalg Z_2,s_1\amalg s_2)$ are of type (a), or (b), then all three terms $F(\cdots)$ in \eq{os5eq11} are 1. The proposition follows.
\end{proof}

Combining Theorem \ref{os5thm4} and Proposition \ref{os5prop2} shows that if $(X,g)$ is a compact, oriented, spin Riemannian 6-manifold and $E_\bu$ is the Dirac operator on $X\t\cS^1$ then we have canonical n-orientations on $\B_Q$ for all $\U(m)$- or $\SU(m)$-bundles $Q\ra X\t\cS^1$. We will use this and Theorem \ref{os5thm3} to construct canonical spin structures $\B_P$ for all $\U(m)$- or $\SU(m)$-bundles~$P\ra X$.

We will use Theorem \ref{os5thm5} in the sequel \cite{JoUp2} to construct canonical `orientation data' in the sense of Kontsevich and Soibelman \cite[\S 5.2]{KoSo1} for any Calabi--Yau 3-fold $X$, solving a long-standing problem in Donaldson--Thomas theory.

\begin{thm}
\label{os5thm5}
Suppose\/ $(X,g)$ is a compact, oriented, spin Riemannian\/ $6$-manifold, and take\/ $F_\bu$ to be the positive Dirac operator\/ $\slashed{D}_+:\Ga^\iy(S_+)\ra\Ga^\iy(S_-),$ an antilinear self-adjoint complex linear elliptic operator. 

Then we can construct canonical choices of spin structures\/ $\si_P^{F_\bu}$ on\/ $\B_P$ for all principal\/ $\U(m)$-bundles\/ $P\ra X$ for all\/ $m\ge 0,$ such that if\/ $P_1\ra X,$\/ $P_2\ra X$ are principal\/ $\U(m_1)$- and\/ $\U(m_2)$-bundles, then Example\/ {\rm\ref{os4ex3}} maps from spin structures on\/ $\B_{P_1\op P_2}$ to pairs of spin structures on\/ $\B_{P_1},\B_{P_2},$ and this maps\/~$\si_{P_1\op P_2}^{F_\bu}\mapsto(\si_{P_1}^{F_\bu},\si_{P_2}^{F_\bu})$. 

We can also construct canonical spin structures\/ $\si_Q^{F_\bu}$ on\/ $\B_Q$ for all principal\/ $\SU(m)$-bundles\/ $Q\ra X$ for all\/ $m\ge 0,$ such that if\/ $P=(Q\t\U(m))/\SU(m)$ is the corresponding\/ $\U(m)$-bundle\/ $P\ra X$ then Example\/ {\rm\ref{os4ex1}} maps\/~$\si_P^{F_\bu}\mapsto\si_Q^{F_\bu}$.
\end{thm}

\begin{proof} Observe that when $F_\bu$ is the positive Dirac operator $\slashed{D}_+$ on $X$, the real elliptic operator $E_\bu$ on $X\t\cS^1$ in Definition \ref{os5def2} is naturally isomorphic to the Dirac operator $\slashed{D}$ on $X\t\cS^1$. Hence Theorem \ref{os5thm4} and Proposition \ref{os5prop2} together yield canonical n-orientations $\check\om_Q^{E_\bu}$ on $\B_Q$ for all principal $\U(m)$- or $\SU(m)$-bundles $Q\ra X\t\cS^1$. As in Theorem \ref{os5thm3}(c),(d), whenever $P\ra X,$ $Q\ra X\t\cS^1$ are principal $\U(m)$-bundles with $Q\vert_{X\t\{1\}}\cong P,$ define a trivialization $\check\Om^{E_\bu}_Q:\check O_Q^{E_\bu}\ot_{\Z_2}\Nu_{Q,P}^*(\check O_{P\t\cS^1}^{E_\bu})\cong\B_Q\t\Z_2$ by~$\check\Om^{E_\bu}_Q=\check\om_Q^{E_\bu}\ot\Nu_{Q,P}^*(\check\om_{P\t\cS^1}^{E_\bu})$. 

We would like to say that these $\check\Om^{E_\bu}_Q$ satisfy the conditions of Theorem \ref{os5thm3}(c), and so the theorem follows from Theorem \ref{os5thm3}. However, there is a problem. Theorem \ref{os5thm3}(c) requires the $\check\Om^{E_\bu}_Q$ to have a compatibility under direct sums for all $\U(m_1)$-, $\U(m_2)$-bundles $Q_1,Q_2\ra X\t\cS^1$. Theorem \ref{os5thm4} shows that the analogous compatibility under direct sums holds for $\SU(m_1)$-, $\SU(m_2)$-bundles, but not for $\U(m_1)$-, $\U(m_2)$-bundles.

We offer two solutions to this problem. Firstly, as in Remark \ref{os5rem1}(b) there are analogues of Examples \ref{os2ex3}, \ref{os4ex3} and Theorems \ref{os5thm2}(ii), \ref{os5thm3} for $\SU(m)$-bundles rather than $\U(m)$-bundles, with essentially identical proofs. By working with $\SU(m)$-bundles rather than $\U(m)$-bundles, the problem goes away, so the analogue of Theorem \ref{os5thm3} gives canonical spin structures $\si_P^{F_\bu}$ on $\B_P$ for all principal $\SU(m)$-bundles $P\ra X$. In particular this works when $P=X\t\SU(5)$. Hence Example \ref{os4ex2} gives a canonical spin structure $\hat\si_{X\t\U(4)}^{F_\bu}$ for the trivial $\U(4)$-bundle $X\t\U(4)\ra X$. Then Theorem \ref{os5thm1} defines spin structures $\smash{\si_P^{F_\bu}}$ on $\B_P$ for all $\U(m)$-bundles $P\ra X$ with the properties we want.

Secondly, we can prove that the $\check\Om^{E_\bu}_Q$ for $\U(m)$-bundles $Q\ra X\t\cS^1$ defined above do actually satisfy Theorem \ref{os5thm3}(c), even if the $\smash{\check\om_Q^{E_\bu}}$ are not compatible with direct sums. Since we can apply Theorem \ref{os5thm3}(d), it is enough to do this when $P=Q\vert_{X\t\{1\}}$ is trivial, $P\cong X\t\U(m)$. So suppose $Q_1,Q_2\ra X\t\cS^1$ are principal $\U(m_1)$-, $\U(m_2)$-bundles with $Q_a\vert_{X\t\{1\}}\cong X\t\U(m_a)$ for $a=1,2$. Then we can choose trivializations of $Q_a$ outside $X\t I_a$ for $a=1,2$, where $I_1,I_2\subset\cS^1$ are small, disjoint open intervals. We can then argue using the Excision Theorem of \cite[Th.~2.13]{Upme}, \cite[Th.~3.1]{JTU} that because $Q_1,Q_2$ are trivial outside disjoint open sets in $X\t\cS^1$, the n-orientations $\check\om_{Q_1}^{E_\bu},\check\om_{Q_2}^{E_\bu},\check\om_{Q_1\op Q_2}^{E_\bu}$ in Theorem \ref{os5thm4} are compatible under Example \ref{os2ex3} in this case. So Theorem \ref{os5thm3}(d) holds, and the theorem follows from Theorem~\ref{os5thm3}.	
\end{proof}

\section{Proof of Theorem \ref{os5thm1}}
\label{os6}

Work in the situation of Theorem \ref{os5thm1}. We first define a spin structure $\hat\si_{X\t\U(m)}^{F_\bu}$ on $\B_{X\t\U(m)}$ for all $m\ge N$, extending $\hat\si_{X\t\U(N)}^{F_\bu}$. Apply Example \ref{os4ex3} with $P_1=X\t\U(N)$ and $P_2=X\t\U(m-N)$, so that $P_1\op P_2=X\t\U(m)$ is the trivial $\U(m)$-bundle. This defines a map from spin structures on $\B_{X\t\U(m)}=\B_{P_1\op P_2}$ to $\B_{X\t\U(N)}=\B_{P_1}$. As $2N\ge n+2$, by Proposition \ref{os4prop3} this map is a 1-1 correspondence. Hence there is a unique $\hat\si_{X\t\U(m)}^{F_\bu}$ on $\B_{X\t\U(m)}$ mapped to $\hat\si_{X\t\U(N)}^{F_\bu}$ by this correspondence. If $m=N$ the correspondence is the identity so we recover~$\hat\si_{X\t\U(m)}^{F_\bu}=\hat\si_{X\t\U(N)}^{F_\bu}$.

Now let $P_1\ra X$ be a principal $\U(m_1)$-bundle for some $m_1\ge 0$. We construct a spin structure $\si_{P_1}^{F_\bu}$ on $\B_{P_1}$. Choose $m_2\ge 0$ with $m=m_1+m_2\ge N$ and $2m_2\ge n$. We claim there exists a principal $\U(m_2)$-bundle $P_2\ra X$, unique up to isomorphism, such that $P_1\op P_2\cong X\t\U(m)$ is the trivial $\U(m)$-bundle. 

To see this, let $E_1=(P_1\t\C^{m_1})/\U(m_1)\ra X$ be the associated rank $m_1$ complex vector bundle, write $\ul\C^m\ra X$ for the trivial rank $m$ complex vector bundle, and let $\al:E_1\ra\ul\C^m$ be a generic vector bundle morphism. As $\dim X=n$, dimension counting shows that $\al$ is injective provided $2m_2\ge n-1:$ for $m_1=0$ injectivity is trivial, and for $m_1\ge1$ this is because non-injectivity means vanishing of all $m_1\t m_1$ minors, so $\al$ is non-injective on the zero set of the section $\La^{m_1}(\al)\colon X\to\Hom_\C(\La^{m_1}(E_1),\La^{m_1}(\C^m))$ which, by transversality, is a smooth manifold of dimension $n-2\binom{m}{m_1}\leqslant 2m_2-2\binom{m}{m_1}+1\leqslant -1.$ Hence it is empty. Then we have an orthogonal splitting $\ul\C^m=\al(E_1)\op E_2$ for $E_2\ra X$ a rank $m_2$ vector bundle. Since $2m_2\ge n$, we can repeat the same argument for $X\t[0,1]$ and a generic interpolation vector bundle morphism between two choices for $\al$ to prove that $E_2$ is independent of $\al$ up to isomorphism. Choose a Hermitian metric on $E_2$, and let $P_2\ra X$ be the associated principal $\U(m_2)$-bundle. The claim easily follows.

We apply Theorem \ref{os5thm1}(c). As $P_1\op P_2\cong X\t\U(m)$ with $m\ge N$, we constructed a spin structure $\hat\si_{X\t\U(m)}^{F_\bu}$ on $\B_{X\t\U(m)}$ above, so Example \ref{os4ex1} maps $\hat\si_{X\t\U(m)}^{F_\bu}$ to a spin structure $\si_{P_1}^{F_\bu}$ on $\B_{P_1}$, as we wanted.

We claim this $\si_{P_1}^{F_\bu}$ is independent of the choice of $m_2$. To prove this, let $m_2,\dot m_2$ be alternative choices yielding $\si_{P_1}^{F_\bu},\dot\si_{P_1}^{F_\bu}$, and let $m,E_1,\al,E_2,P_2$ and $\dot m,\dot E_1,\dot\al,\dot E_2,\dot P_2$ be the intermediate data. Without loss of generality $\dot m_2>m_2$. Then $\ul\C^{\dot m}=\ul\C^m\op\ul\C^{\dot m_2-m_2}$, and we may take $\dot\al=\al\op 0_{\C^{\dot m_2-m_2}}$, and $\dot E_2=E_2\op \ul\C^{\dot m_2-m_2}$, and $\dot P_2=P_2\op (X\t\U(\dot m_2-m_2))$. Then we have matching diagrams of principal $\U(k)$-bundles $P$, and of spin structures $\si_P^{F_\bu}$ on $\B_P$, where arrows `$\dashra$' indicate mapping $P_1\op P_2\dashra P_1$, $\si_{P_1\op P_2}^{F_\bu}\dashra \si_{P_1}^{F_\bu}$ in Example~\ref{os4ex3}:
\begin{align*}
&\xymatrix@!0@C=78pt@R=30pt{
*+[r]{(X\!\t\!\U(m))\!\op\!(X\!\t\!\U(\dot m_2\!-\!m_2))} \ar[d]^\cong \ar@{.>}[rr] && *+[r]{X\t\U(m)} \ar[rr]_\cong && *+[l]{P_1\op P_2} \ar@{.>}[d] \\
*+[r]{X\t\U(\dot m)} \ar[rr]^(0.57)\cong && *+[r]{P_1\op \dot P_2} \ar@{.>}[r] & *+[r]{P_1} \ar@{=}[r] & *+[l]{P_1,}} \\
&\xymatrix@!0@C=77pt@R=30pt{
*+[r]{\hat\si_{X\t\U(\dot m)}^{F_\bu}} \ar@{=}[d] \ar@{.>}[rr] && *+[r]{\hat\si_{X\t\U(m)}^{F_\bu}} \ar@{=}[rr] && *+[l]{\hat\si_{X\t\U(m)}^{F_\bu}} \ar@{.>}[d] \\
*+[r]{\hat\si_{X\t\U(\dot m)}^{F_\bu}} \ar@{=}[rr]  && *+[r]{\hat\si_{X\t\U(\dot m)}^{F_\bu}} \ar@{.>}[r] & *+[r]{\dot\si_{P_1}^{F_\bu}} \ar@{=}[r]^? & *+[l]{\si_{P_1}^{F_\bu}.}}
\end{align*}

Here using an associativity property of Example \ref{os4ex3} in the decomposition $(X\t\U(N))\op (X\t\U(m_2-N))\op (X\t\U(\dot m_2-m_2))$ and the definitions of $\hat\si_{X\t\U(m)}^{F_\bu},\hat\si_{X\t\U(\dot m)}^{F_\bu}$, we see that the top left arrow above maps $\hat\si_{X\t\U(\dot m)}^{F_\bu}\dashra\hat\si_{X\t\U(m)}^{F_\bu}$. The same associativity property implies that the rectangle must commute, forcing $\dot\si_{P_1}^{F_\bu}=\si_{P_1}^{F_\bu}$. Hence $\si_{P_1}^{F_\bu}$ is independent of~$m_2$.

Next let $P\ra X$ be a principal $\SU(m)$-bundle for $m\ge 0$, and $Q=(P\t\U(1))/\Z_m$ be the associated principal $\U(m)$-bundle. We defined a spin structure $\si_Q^{F_\bu}$ on $\B_Q$ above. As in Theorem \ref{os5thm1}(d), define $\si_P^{F_\bu}$ to be the unique spin structure on $\B_P$ which is the image of $\si_Q^{F_\bu}$ under Example~\ref{os4ex1}.

We have now constructed spin structures $\si_P^{F_\bu}$ on $\B_P$ for all principal $\U(m)$- and $\SU(m)$-bundles $P\ra X$. We claim they satisfy Theorem \ref{os5thm1}(a)--(d). For (a), the proof above that $\hat\si_{X\t\U(\dot m)}^{F_\bu}\dashra\hat\si_{X\t\U(m)}^{F_\bu}$ in the case that $m_2=\dot m-m$ satisfies $m+m_2\ge N$ and $2m_2\ge n$ also has $\hat\si_{X\t\U(\dot m)}^{F_\bu}\dashra\si_{X\t\U(m)}^{F_\bu}$ by definition of $\si_{X\t\U(m)}^{F_\bu}$. Thus $\si_{X\t\U(m)}^{F_\bu}=\hat\si_{X\t\U(m)}^{F_\bu}$ for all $m\ge N$, giving (a) when~$m=N$.

For (b), the isomorphisms $\B_P\cong\B_{P'}$ and $K_P^{F_\bu}\cong K_{P'}^{F_\bu}$ are independent of $\rho$ as the definition of $\B_P$ involves dividing out by $\Aut(P)$. It is obvious from the construction that these isomorphisms identify~$\si_P^{F_\bu}\cong\si_{P'}^{F_\bu}$.

For (c), let $P_1,P_2\ra X$ be principal $\U(m_1)$-, $\U(m_2)$-bundles. Construct $\si_{P_1}^{F_\bu},\si_{P_2}^{F_\bu}$ using $\U(k_1)$-, $\U(k_2)$-bundles $Q_1,Q_2\ra X$ with $P_a\op Q_a\cong X\t\U(m_a+k_a)$ for $a=1,2$, for $k_1,k_2$ sufficiently large. Then we may construct $\si_{P_1\op P_2}^{F_\bu}$ using the $\U(k_1+k_2)$-bundle $Q_1\op Q_2\ra X$ with $(P_1\op P_2)\op (Q_1\op Q_2)\cong X\t\U(m_1+m_2+k_1+k_2)$. In a similar way to the diagrams above, we have 
\begin{align*}
 \xymatrix@!0@C=78pt@R=35pt{
*+[r]{X\!\t\!\U(m_1\!+\!m_2\!+\!k_1\!+\!k_2)} \ar[d]^(0.4)\cong \ar[rr]_(0.8)\cong && *+[r]{P_1\op P_2\op Q_1\op Q_2} \ar@{.>}[rr] && *+[l]{P_1\op P_2} \ar@{.>}[d] \\
*+[r]{\;\>\begin{subarray}{l}\ts (X\!\t\!\U(m_1\!+\!k_1))\op{} \\ \ts(X\!\t\!\U(m_2\!+\!k_2))\end{subarray}} \ar@{.>}[rr] && {X\!\t\!\U(m_1\!+\!k_1)} \ar[r]^(0.58)\cong & {P_1\op Q_1} \ar@{.>}[r] & *+[l]{P_1,}} \\
\xymatrix@!0@C=77pt@R=30pt{
*+[r]{\hat\si_{X\t\U(m_1+m_2+k_1+k_2)}^{F_\bu}} \ar@{=}[d] \ar@{=}[rr] && *+[r]{\hat\si_{X\t\U(m_1+m_2+k_1+k_2)}^{F_\bu}} \ar@{.>}[rr] && *+[l]{\si_{P_1\op P_2}^{F_\bu}} \ar@{.>}[d] \\
*+[r]{\hat\si_{X\t\U(m_1+k_1+m_2+k_2)}^{F_\bu}} \ar@{.>}[rr]  && {\hat\si_{X\t\U(m_1+k_1)}^{F_\bu}} \ar@{=}[r] & {\hat\si_{X\t\U(m_1+k_1)}^{F_\bu}} \ar@{.>}[r] & *+[l]{\si_{P_1}^{F_\bu},}}
\end{align*} 
so a similar argument to that above shows that Example \ref{os4ex3} maps $\si_{P_1\op P_2}^{F_\bu}\mapsto\si_{P_1}^{F_\bu}$ on $\B_{P_1}$. Similarly it maps $\si_{P_1\op P_2}^{F_\bu}\mapsto\si_{P_2}^{F_\bu}$ on $\B_{P_2}$, proving (c). Part (d) is immediate from the construction.

This proves there exists a family of spin structures $\si_P^{F_\bu}$ satisfying Theorem \ref{os5thm1}(a)--(d). As the construction of these $\si_P^{F_\bu}$, with $\si_{X\t\U(m)}^{F_\bu}=\hat\si_{X\t\U(m)}^{F_\bu}$, uses (a)--(d) at each step, this family is unique. The last part of the theorem follows from Proposition~\ref{os4prop4}.

\section{Proof of Theorem \ref{os5thm2}}
\label{os7}

Throughout this section we work in the situation of Definitions \ref{os5def1} and \ref{os5def2} and Theorem \ref{os5thm2}, with $X,F_\bu$ and $X\t\cS^1, E_\bu$ fixed. We take $G$ to be a Lie group, and $P\ra X$, $Q\ra X\t\cS^1$ to be principal $G$-bundles with an isomorphism $Q\vert_{X\t\{1\}}\cong P$ over $X\t\{1\}\cong X$. Then we also have (non-canonical) isomorphisms $Q\vert_{X\t\{e^{i\th}\}}\cong P$ for all~$e^{i\th}\in\cS^1$.

Let $\nabla_Q\in\A_Q$ be a connection on $Q\ra X\t\cS^1$, so that $[\nabla_Q]$ is a point of $\B_Q$. For simplicity we write much of the proof below as though $\nabla_Q$ is fixed, but in fact all our constructions will be gauge-invariant/equivariant and depend continuously on $\nabla_Q$, and so yield vector bundles, (isomorphisms of) principal $\Z_2$-bundles, etc., over the specified open substacks of $\B_Q$ where they are defined.

We begin by setting up some notation we will need during the proof. We explain the strategy of the proof after Example~\ref{os7ex1}.

\begin{dfn}
\label{os7def1}
Let $X,F_\bu,G$, $Q\ra X\t\cS^1$ and $\nabla_Q$ be as above. Then for each $e^{i\th}\in\cS^1$, we have a principal $G$-bundle $Q\vert_{X\t\{e^{i\th}\}}$ over $X\t\{e^{i\th}\}\cong X$, with associated vector bundle $\Ad(Q)\vert_{X\t\{e^{i\th}\}}$ and connection $\nabla_{\Ad(Q)}\vert_{X\t\{e^{i\th}\}}$. By assumption $F_\bu=(F_0,F_1,D)$ is antilinear self-adjoint, so that $F_0,F_1\ra X$ are complex vector bundles with the same underlying real vector bundle $F\ra X$. Thus we may form the self-adjoint real elliptic twisted operator
\begin{equation*}
D^{\nabla_{\Ad(Q)}\vert_{X\t \{e^{i\th}\}}}:\Ga^\iy\bigl(\Ad(Q)\vert_{X\t\{e^{i\th}\}}\ot F\bigr)\longra\Ga^\iy\bigl(\Ad(Q)\vert_{X\t\{e^{i\th}\}}\ot F\bigr).
\end{equation*}
This has discrete spectrum $\Spec\bigl(D^{\nabla_{\Ad(Q)}\vert_{X\t \{e^{i\th}\}}}\bigr)\subset\R$, with finite-dimensional real eigenspaces $\Eig_\la\bigl(D^{\nabla_{\Ad(Q)}\vert_{X\t \{e^{i\th}\}}}\bigr)$ for $\la\in \Spec\bigl(D^{\nabla_{\Ad(Q)}\vert_{X\t \{e^{i\th}\}}}\bigr)$. 

If $s\in \Eig_\la\bigl(D^{\nabla_{\Ad(Q)}\vert_{X\t \{e^{i\th}\}}}\bigr)$ then $i_{F_0}(s)\in \Eig_{-\la}\bigl(D^{\nabla_{\Ad(Q)}\vert_{X\t \{e^{i\th}\}}}\bigr)$, with $i_{F_0}$ the complex structure on $F_0=F$. Hence $\Spec\bigl(D^{\nabla_{\Ad(Q)}\vert_{X\t \{e^{i\th}\}}}\bigr)$ is preserved by multiplication by $\pm 1$, and $\Eig_0\bigl(D^{\nabla_{\Ad(Q)}\vert_{X\t \{e^{i\th}\}}}\bigr)$ is a complex vector space.

The spectrum $\Spec\bigl(D^{\nabla_{\Ad(Q)}\vert_{X\t \{e^{i\th}\}}}\bigr)$ varies continuously with $[\nabla_Q]\in\B_Q$ and $e^{i\th}\in\cS^1$, in an appropriate sense. Thus, for fixed $e^{i\th}\in\cS^1$ and $\la\in\R$ the condition $\la\notin\Spec\bigl(D^{\nabla_{\Ad(Q)}\vert_{X\t\{e^{i\th}\}}}\bigr)$ is open on $[\nabla_Q]\in\B_Q$. Write $U_{e^{i\th}}^\la(Q)\subset\B_Q$ for the open substack of $[\nabla_Q]\in\B_Q$ with~$\la\notin\Spec\bigl(D^{\nabla_{\Ad(Q)}\vert_{X\t \{e^{i\th}\}}}\bigr)$.
For $e^{i\th_1},\ldots,e^{i\th_k}\in\cS^1$ and $\la_1,\ldots,\la_k\in\R$, $k\ge 1$, we write $U_{e^{i\th_1},\ldots,e^{i\th_k}}^{\la_1,\ldots,\la_k}(Q)$ for~$U_{e^{i\th_1}}^{\la_1}(Q)\cap\cdots\cap U_{e^{i\th_k}}^{\la_k}(Q)\subset\B_Q$.
More generally, if $I\subset\cS^1$ is a compact interval or $I=\cS^1$, and $\la\in\R$, we write $U_I^\la(Q)\subset\B_Q$ for the open substack of $[\nabla_Q]\in\B_Q$ with $\la\notin\Spec\bigl(D^{\nabla_{\Ad(Q)}\vert_{X\t \{e^{i\th}\}}}\bigr)$ for all $e^{i\th}\in I$, again an open condition. If $\th_1,\th_2\in\R$ with $\th_1\le\th_2\le\th_1+2\pi$ we will write $[\th_1,\th_2]$ as a shorthand for the compact interval $\bigl\{e^{i\th}:\th\in[\th_1,\th_2]\subset\R\bigr\}\subset\cS^1$, so we have~$U_{[\th_1,\th_2]}^\la(Q)\subset\B_Q$.
For compact intervals $I_1,\ldots,I_k\subseteq\cS^1$ and $\la_1,\ldots,\la_k\in\R$, $k\ge 1$, we write $U_{I_1,\ldots,I_k}^{\la_1,\ldots,\la_k}(Q)$ for the open substack~$U_{I_1}^{\la_1}(Q)\cap\cdots\cap U_{I_k}^{\la_k}(Q)\subset\B_Q$.

Now let $e^{i\th}\in\cS^1$ and $\mu,\nu\in\R\cup\{\pm\iy\}$ with $\mu,\nu\notin \Spec\bigl(D^{\nabla_{\Ad(Q)}\vert_{X\t \{e^{i\th}\}}}\bigr)$, that is, $[\nabla_Q]\in U_{e^{i\th},e^{i\th}}^{\mu,\nu}(Q)=U_{e^{i\th}}^\mu(Q)\cap U_{e^{i\th}}^\nu(Q)\subset\B_Q$. Write
\begin{equation*}
\Eig^{\mu,\nu}\bigl(D^{\nabla_{\Ad(Q)}\vert_{X\t \{e^{i\th}\}}}\bigr)
\subset L^2\bigl(\Ad(Q)\vert_{X\t\{e^{i\th}\}}\ot F\bigr)
\end{equation*}
for the $L^2$-closure of the direct sum of all eigenspaces $\Eig_\la\bigl(D^{\nabla_{\Ad(Q)}\vert_{X\t \{e^{i\th}\}}}\bigr)$ with eigenvalue $\la$ between $\mu$ and $\nu$, that is, $\mu<\la<\nu$ or~$\mu>\la>\nu$.

Suppose that $\mu,\nu\in\R$. Then $\Eig^{\mu,\nu}\bigl(D^{\nabla_{\Ad(Q)}\vert_{X\t \{e^{i\th}\}}}\bigr)$ is a finite-dimensional real vector subspace of $\Ga^\iy\bigl(\Ad(Q)\vert_{X\t\{e^{i\th}\}}\ot F\bigr)$, which varies continuously with $[\nabla_Q]\in U_{e^{i\th},e^{i\th}}^{\mu,\nu}(Q)$, and so forms a real vector bundle on the topological stack $U_{e^{i\th},e^{i\th}}^{\mu,\nu}(Q)$, which may have different ranks on different connected components of $U_{e^{i\th},e^{i\th}}^{\mu,\nu}(Q)$. Write $O_{Q\vert_{e^{i\th}}}^{\mu,\nu} \ra U_{e^{i\th},e^{i\th}}^{\mu,\nu}(Q)$ for the orientation bundle of this vector bundle, a principal $\Z_2$-bundle over~$U_{e^{i\th},e^{i\th}}^{\mu,\nu}(Q)$.
\end{dfn}

\begin{dfn}
\label{os7def2}
Continue in the situation of Definition \ref{os7def1}. Let $\th_0,\ldots,\th_{k+1}\in\R$ with $\th_0<\th_2<\cdots<\th_{k+1}$ and $\th_0 = \th_k-2\pi$, $\th_{k+1}=\th_1+2\pi$. Then writing $[\th_j,\th_{j+1}]$ as a shorthand for $\bigl\{e^{i\th}:\th\in[\th_j,\th_{j+1}]\subset\R\bigr\}\subset\cS^1$ as above, we cut $\cS^1$ into $k$ intervals $[\th_1,\th_2]$, $[\th_2,\th_3]$, \ldots, $[\th_k,\th_{k+1}]$ intersecting pairwise in their $k$ endpoints $e^{i\th_{k+1}}=e^{i\th_1},e^{i\th_2},\ldots,e^{i\th_k}$. Let $\la_0,\ldots,\la_{k+1}>0$ with $\la_0=\la_k$, $\la_1=\la_{k+1}$. A $2k$-tuple $(\th_1,\ldots,\th_k,\la_1,\ldots,\la_k)$ is a {\it cut\/} for $[\nabla_Q]\in\B_Q$ if
\begin{equation*}
[\nabla_Q]\in U_{[\th_1,\th_2],[\th_2,\th_3],\ldots,[\th_k,\th_{k+1}]}^{\la_1,\ldots,\la_k}(Q)\subset\B_Q,
\end{equation*}
that is, if $\la_j\notin\Spec\bigl(D^{\nabla_{\Ad(Q)}\vert_{X\t \{e^{i\th}\}}}\bigr)$ for all $j=1,\ldots,k$ and~$\th_j\le\th\le\th_{j+1}$.

Observe that $U_{[\th_1,\th_2],[\th_2,\th_3],\ldots,[\th_k,\th_{k+1}]}^{\la_1,\ldots,\la_k}(Q)\subset U_{e^{i\th_j},e^{i\th_j}}^{\la_{j-1},\la_j}(Q)$ for $j=1,\ldots,k$. Thus we have a principal $\Z_2$-bundle
\e
\bigot\nolimits_{j=1}^kO_{Q\vert_{e^{i\th_j}}}^{\la_{j-1},\la_j}\longra
U_{[\th_1,\th_2],[\th_2,\th_3],\ldots,[\th_k,\th_{k+1}]}^{\la_1,\ldots,\la_k}(Q).
\label{os7eq1}
\e

We consider a cut $(\th_1,\ldots,\th_k,\la_1,\ldots,\la_k)$ as cutting $\cS^1$ into the $k$ closed intervals $[\th_1,\th_2],\ab\ldots,\ab[\th_k,\th_{k+1}]$, so we cut $Y=X\t\cS^1$ into $k$ compact manifolds with boundary $Y_1'=X\t[\th_1,\th_2]$, \ldots, $Y_k'=X\t[\th_k,\th_{k+1}]$, which we consider as one disconnected compact manifold with boundary $Y'=Y_1'\amalg\cdots\amalg Y_k'$. We write the boundary $\pd Y_j'=\pd (X\t[\th_j,\th_{j+1}])$ as $X\t\{e^{\th_j}_+,e^{\th_{j+1}}_-\}$. Here $+,-$ keep track of orientations, since in $\pd[\th_j,\th_{j+1}]=\{e^{\th_j},e^{\th_{j+1}}\}$ we give $e^{\th_j}$ the positive orientation, and $e^{\th_{j+1}}$ the negative orientation. Hence the boundary $\pd Y'=X\t\{e^{\th_1}_+,e^{\th_1}_-,\ldots,e^{\th_k}_+,e^{\th_k}_-\}$, where $X\t\{e^{\th_j}_+\}$ and $X\t\{e^{\th_j}_-\}$ are the boundary components to be `glued together' at $X\t\{e^{\th_j}\}$ to make $Y'$ into~$Y$.

For a fixed $[\nabla_Q]\in\B_Q$, we will want to compare different choices of cuts $(\th_1,\ldots,\th_k,\la_1,\ldots,\la_k)$ for $[\nabla_Q]$. Our primary way to do this is to insert an extra point: if $j=1,\ldots,k$ and $\th_{j-1}<\th_*<\th_j$ and $\la_*>0$ we can consider
\e
(\th_1,\ldots,\th_{j-1},\th_*,\th_j,\ldots,\th_k,\la_1,\ldots,\la_{j-1},\la_*,\la_j,\ldots,\la_k).
\label{os7eq2}
\e
If both $(\th_1,\ldots,\th_k,\la_1,\ldots,\la_k)$ and \eq{os7eq2} are cuts for $[\nabla_Q]$, we will call passing from $(\th_1,\ldots,\la_k)$ to \eq{os7eq2} an {\it insertion}, and passing the other way a {\it deletion}.
\end{dfn}

We illustrate the notion of cut for $[\nabla_Q]$ in Figure \ref{os7fig1}. The horizontal axis is the coordinate $\th$ on the circle $\cS^1\ni e^{i\th}$. On the vertical axis we plot the spectrum 
$\Spec\bigl(D^{\nabla_{\Ad(Q)}\vert_{X\t\{e^{i\th}\}}}\bigr)$, a discrete collection of smooth closed curves periodic in $\th$. To choose a cut $(\th_1,\ldots,\th_k,\la_1,\ldots,\la_k)$ we divide the circle into $k$ intervals $[\th_j,\th_{j+1}]$, and choose $\la_j>0$ such that the horizontal level $\la_j$ over $[\th_j,\th_{j+1}]$ does not intersect the graph of the spectrum.
 
\begin{figure}[ht]
\centering
\begin{tikzpicture}
\draw[->] (-0.5,0) -- (10.5,0) node[right]{$\th$};
\draw (0,-3.6) -- (0,3.6);
\draw (0,0) node[below left] {$\th_1$};
\draw[dashed] (0,0.75) node[left]{$\la_1$} -- (2,0.75);
\draw[dashed] (0,-0.75) node[left]{$-\la_1$} -- (2,-0.75);
\draw[dashed] (2,2.75) node[left]{$\la_2$} -- (2.65,2.75);
\draw[dashed] (2,-2.75) node[left]{$-\la_2$} -- (2.65,-2.75);
\draw (2,-3.6) -- (2,3.6);
\draw (2,0) node[below left] {$\th_2$};
\draw (3.75,-0.25) node {\ldots};
\draw (5.5,-3.6) -- (5.5,3.6);
\draw (5.5,0) node[below left] {$\th_{k-1}$};
\draw[dashed] (5.5,3.3) node[left]{$\la_{k-1}$} -- (8,3.3);
\draw[dashed] (5.5,-3.3) node[left]{$-\la_{k-1}$} -- (8,-3.3);
\draw (8,-3.6) -- (8,3.6);
\draw (8,0) node[below left] {$\th_k$};
\draw[dashed] (8,2.25) -- (10,2.25) node[right]{$\begin{subarray}{l}\ts\la_k= \\ \ts \la_0\end{subarray}$};
\draw[dashed] (8,-2.25) -- (10,-2.25)  node[right]{$-\la_k$};
\draw (10,-3.6) -- (10,3.6);
\draw (10,0) node[below left] {$\begin{subarray}{l}\ts \th_{k+1}= \\ \ts\th_1+2\pi\end{subarray}$};
\draw[gray] (0,1) to [out=0, in=180] (2,1.75) to [out=0, in=180] (4,0.5) to [out=0, in=180] (6.5,2.75) to [out=0, in=180] (8,0.25) to [out=0, in=180] (10,1);
\draw[gray] (0,0.5) to [out=0, in=180] (2,0.25) to [out=0, in=180] (4,1) to [out=0, in=180] (6.5,-1) to [out=0, in=180] (8,-1.5) to [out=0, in=180] (9,-1.1) to [out=0, in=180] (10,-0.5); 
\begin{scope}[yscale=-1]
\draw[gray] (0,1) to [out=0, in=180] (2,1.75) to [out=0, in=180] (4,0.5) to [out=0, in=180] (6.5,2.75) to [out=0, in=180] (8,0.25) to [out=0, in=180] (10,1);
\draw[gray] (0,0.5) to [out=0, in=180] (2,0.25) to [out=0, in=180] (4,1) to [out=0, in=180] (6.5,-1) to [out=0, in=180] (8,-1.5) to [out=0, in=180] (9,-1.1) to [out=0, in=180] (10,-0.5); 
\end{scope}
\end{tikzpicture}
\caption{Graph of spectrum $e^{i\th}\mapsto\Spec\bigl(D^{\nabla_{\Ad(Q)}\vert_{X\t\{e^{i\th}\}}}\bigr)$ (curved lines), \newline and a choice of cut $(\th_1,\ldots,\th_k,\la_1,\ldots,\la_k)$ for $[\nabla_Q]$ (dotted lines).}
\label{os7fig1}
\end{figure}

By making $k$ large and the intervals $[\th_j,\th_{j+1}]$ short, it is easy to prove:

\begin{lem}
\label{os7lem1}
Any\/ $[\nabla_Q]\in\B_Q$ admits a cut\/ $(\th_1,\ldots,\th_k,\la_1,\ldots,\la_k)$. Any two cuts for\/ $[\nabla_Q]$ may be linked by a finite chain of insertions and deletions of cuts.
\end{lem}

\begin{ex}
\label{os7ex1}
Let $\nabla_P$ be a connection on $P\ra X$. Take $Q=P\t\cS^1=\pi_X^*(P)$, where $\pi_X:X\t\cS^1\ra X$ is the projection, and set $\nabla_Q=\pi_X^*(\nabla_P)$. Then $D^{\nabla_{\Ad(Q)}\vert_{X\t \{e^{i\th}\}}}=D^{\nabla_{\Ad(P)}}$, so $\Spec\bigl(D^{\nabla_{\Ad(Q)}\vert_{X\t \{e^{i\th}\}}}\bigr)=\Spec\bigl(D^{\nabla_{\Ad(P)}}\bigr)$ is independent of $e^{i\th}\in\cS^1$. Let $\la\in(0,\iy)\sm\Spec\bigl(D^{\nabla_{\Ad(P)}}\bigr)$. Then $(0,\la)$ is a cut for $[\nabla_Q]$, with $k=1$, $\th_1=0$, $\th_2=2\pi$ and~$\la_0=\la_1=\la_2=\la$.
\end{ex}

We can now explain our strategy for the proof of Theorem \ref{os5thm2}. First observe that by \eq{os2eq5} we have a canonical isomorphism of principal $\Z_2$-bundles on~$\B_Q$:
\e
\check{O}^{E_\bu}_Q \ot_{\Z_2} \Nu_{Q,P}^*(\check{O}^{E_\bu}_{P\t\cS^1})\cong
O^{E_\bu}_Q \ot_{\Z_2} \Nu_{Q,P}^*(O^{E_\bu}_{P\t\cS^1}).
\label{os7eq3}
\e
So we will prove the analogue of Theorem \ref{os5thm2} with $\check{O}^{E_\bu}_Q \ot_{\Z_2} \Nu_{Q,P}^*(\check{O}^{E_\bu}_{P\t\cS^1})$ replaced by $O^{E_\bu}_Q \ot_{\Z_2} \Nu_{Q,P}^*(O^{E_\bu}_{P\t\cS^1})$. Let $[\nabla_Q]\in\B_Q$, and choose a cut $(\th_1,\ab\ldots,\ab\th_k,\ab\la_1,\ab\ldots,\ab\la_k)$ for $[\nabla_Q]$. We will establish canonical 1-1 correspondences between the following $\Z_2$-torsors:
\begin{itemize}
\setlength{\itemsep}{0pt}
\setlength{\parsep}{0pt}
\item[(A)] $O^{E_\bu}_Q \ot_{\Z_2} \Nu_{Q,P}^*(O^{E_\bu}_{P\t\cS^1})\vert_{[\nabla_Q]}$, the right hand side of \eq{os7eq3} at $[\nabla_Q]$;
\item[(B)] $\bigot_{j=1}^kO_{Q\vert_{e^{i\th_j}}}^{\la_{j-1},\la_j}\vert_{[\nabla_Q]}$, the fibre of \eq{os7eq1} at $[\nabla_Q]$; and
\item[(C)] $\Ga_{Q,P}^*(M_P^{F_\bu})\vert_{[\nabla_Q]}$, the right hand side of \eq{os5eq4} at $[\nabla_Q]$.
\end{itemize}
Here Propositions \ref{os7prop1} and \ref{os7prop2} below define the 1-1 correspondences (A) $\Leftrightarrow$ (B) and (B) $\Leftrightarrow$ (C). Composing these gives a correspondence (A) $\Leftrightarrow$ (C), which is continuous in $[\nabla_Q]$ and unchanged by insertions and deletions. Hence by Lemma~\ref{os7lem1} it is independent of the choice of cut $(\th_1,\ldots,\th_k,\la_1,\ldots,\la_k)$. Combining this with \eq{os7eq3} defines an isomorphism of principal $\Z_2$-bundles $\ga_{Q,P}^{F_\bu}$ in \eq{os5eq4}, which we prove satisfies Theorem~\ref{os5thm2}(i),(ii).

We prove these 1-1 correspondences by cutting $Y=X\t\cS^1$ into the compact manifold with boundary $Y'=(X\t[\th_1,\th_2])\amalg\cdots\amalg(X\t[\th_k,\th_{k+1}])$ as above. To prove (A) $\Leftrightarrow$ (B) we compare elliptic operators on $Y$ with elliptic operators on $Y'$, with suitable elliptic boundary conditions at $\pd Y'$. An important part of the proof is the freedom to continuously deform these boundary conditions, leading to a corresponding deformation of the bundle (A). Combined with a technique to change boundary conditions, which picks up the terms $\bigot_{j=1}^kO_{Q\vert_{e^{i\th_j}}}^{\la_{j-1},\la_j}\vert_{[\nabla_Q]}$ appearing in (B), we end up with a Fredholm problem in which $Q'$ may be deformed to be constant on each component $Y'_j$, where the corresponding version of (A) becomes trivial. To prove (B) $\Leftrightarrow$ (C) we need to construct a square root of the complex determinant line bundle $\det_\C\bigl(D^{\nabla_{\Ad(Q)}\vert_{X\t \{e^{i\th}\}}}\bigr)$ over $\cS^1$. A cut determines a natural square root over each $[\th_j,\th_{j+1}]$. While gluing these square roots together one has two choices at each boundary point $e^{i\th_j}$ and these choices are precisely parameterized by (B).

A first order elliptic operator $\ti{D}' \colon \Ga^\iy(E'_0) \ra \Ga^\iy(E'_1)$ on $Y'$ with self-adjoint boundary operator becomes a Fredholm operator upon imposing appropriate boundary conditions. A boundary condition is a closed subspace $B\subset L^2_{1/2}(E'_0|_{\partial Y'})$ of the fractional Sobolev space. We shall use the \emph{elliptic boundary conditions} of B\"ar--Ballmann \cite[Def.~7.5]{BaBa}, reviewed in Appendix \ref{osA}. The main result on elliptic boundary conditions, see Theorem \ref{osAthm1} and \cite[Th.~8.5]{BaBa}, is that they determine Fredholm operators
\e
\label{os7eq4}
f \longmapsto \ti{D}'(f), \qquad
f\vert_{\partial Y'} \in B.
\e

We set up some notation for Proposition \ref{os7prop1}:

\begin{dfn}
\label{os7def3}
Work in the situation of Definitions \ref{os7def1} and \ref{os7def2}, with $\nabla_Q$ a fixed connection on $Q\ra X\t\cS^1$, and $(\th_1,\ldots,\th_k,\la_1,\ldots,\la_k)$ a cut for $[\nabla_Q]$, and $Y=X\t\cS^1$, $Y'=(X\t[\th_1,\th_2])\amalg\cdots\amalg(X\t[\th_k,\th_{k+1}])$, with boundary $\pd Y'=X\t\{e^{\th_1}_+,e^{\th_1}_-,\ldots,e^{\th_k}_+,e^{\th_k}_-\}$. Write $E'_\bu=(E',E',\ti{D}')$ and $Q',\Ad(Q'),\nabla_{Q'}$ for the pullbacks of $E_\bu$ and $Q,\Ad(Q),\nabla_Q$ along the canonical projection $Y'\ra Y$. For $f\in\Ga^\iy(\Ad(Q')\ot E')$ the restriction $f\vert_{\pd Y'}$ divides into components $\{f_{\th_{j-}}, f_{\th_{j+}}\}_{j=1}^k$, where
\e
f_{\th_{j\pm}} = f\vert_{X\t\{e^{i\th_j}_\pm\}} \in \Ga^\iy\big(\Ad(Q)\vert_{X\t\{e^{i\th_j}\}}\ot F\big),
\label{os7eq5}
\e
using the identifications $E'\vert_{X\t\{e^{i\th_j}_\pm\}} = F$ and $Q'\vert_{X\t\{e^{i\th_j}_\pm\}} = Q\vert_{X\t \{e^{i\th_j}\}}$.

To perform deformations, and for proving continuity, we must also work in continuous families. To do this, let $T$ a paracompact Hausdorff topological space, and we replace $P\ra X$, $Q\ra Y=X\t\cS^1$, $Q'\ra Y'$ by topological principal $G$-bundles $P\ra T\t X$, $Q\ra T\t Y$, $Q'\ra T\t Y'$ which have smooth structures in the manifold directions $X,Y,Y'$, and we replace $\nabla_P,\nabla_Q,\nabla_{Q'}$ by partial connections in the manifold directions $X,Y,Y'$. Then for each $t\in T$, we have smooth principal $G$-bundles $P\vert_{\{t\}\t X}\ra X$, $Q\vert_{\{t\}\t Y}\ra Y$, $Q'\vert_{\{t\}\t Y'}\ra Y'$, and smooth connections $\nabla_{P\vert_{\{t\}\t X}}$, $\nabla_{Q\vert_{\{t\}\t Y}}$, $\nabla_{Q'\vert_{\{t\}\t Y'}}$, which all vary continuously with $t\in T$. Let $\pi_{Y'}:T\t Y' \ra Y'$ denote the projection. 

A {\it $T$-family of elliptic boundary conditions for $Q'$} is a Hilbert subbundle
\begin{equation*}
B \subset L^2_{1/2}\big(\Ad(Q')\vert_{\partial Y'}\ot\pi_{Y'}^*(E'_0)\big)
\end{equation*}
such that every $B\vert_t$, $t\in T$, is an elliptic boundary condition for $\ti{D}'{}^{\nabla_{Q'\vert \{t\}\t Y'}}$ in the sense of Definition \ref{osAdef1}.
In this case, \eq{os7eq4} becomes a continuous family of Fredholm operators by Proposition \ref{osAprop1}. Extending Definition \ref{os2def3}, we have an \emph{orientation bundle} $O^{E'_\bu}_{Q'}(B) \ra T$ of the corresponding family of Fredholm operators \eq{os7eq4}. We will usually define $T$-families of elliptic boundary conditions by specifying for each parameter $t \in T$ the conditions for $\{f_{\th_{j-}}, f_{\th_{j+}}\}_{j=1}^k$ to belong to the subspace $B\vert_t$, using the notation \eq{os7eq5} for sections over the boundary.
\end{dfn}

\begin{lem}
\label{os7lem2}
Work in the situation and notation of Definition\/ {\rm\ref{os7def1}}. 
\begin{itemize}
\setlength{\itemsep}{0pt}
\setlength{\parsep}{0pt}
\item[{\rm(i)}] Let\/ $I\subset\cS^1$ be a compact interval and\/ $\mu,\nu\in\R$. Then we have an open substack\/ $U_{I,I}^{\mu,\nu}(Q)\subset\B_Q$. For each\/ $e^{i\th}\in I$ we have\/ $U_{I,I}^{\mu,\nu}(Q)\subseteq U_{e^{i\th},e^{i\th}}^{\mu,\nu}(Q)$ and an orientation bundle\/ $O_{Q\vert_{e^{i\th}}}^{\mu,\nu} \ra U_{e^{i\th},e^{i\th}}^{\mu,\nu}(Q)$. The restrictions\/ $O_{Q\vert_{e^{i\th}}}^{\mu,\nu}\vert_{U_{I,I}^{\mu,\nu}(Q)}$ vary continuously with\/ $e^{i\th}\in I$. Thus we have fibre transport isomorphisms
\e
O^{\mu,\nu}_{Q\vert_{e^{i\th_1}}}\Big\vert_{U_{I,I}^{\mu,\nu}(Q)} \overset{\cong}{\longra} O^{\mu,\nu}_{Q\vert_{e^{i\th_2}}}\Big\vert_{U_{I,I}^{\mu,\nu}(Q)}
\quad \text{for all\/ $e^{i\th_1},e^{i\th_2} \in I$.}
\label{os7eq6}
\e
\item[{\rm(ii)}] For\/ $\la,\mu,\nu \in \R$ and\/ $e^{i\th}\in\cS^1$ we have a canonical isomorphism
\e
O_{Q\vert_{e^{i\th}}}^{\la,\nu}\Big\vert_{U_{e^{i\th},e^{i\th},e^{i\th}}^{\la,\mu,\nu}(Q)} \overset{\cong}{\longra} O_{Q\vert_{e^{i\th}}}^{\la,\mu}\ot_{\Z_2} O_{Q\vert_{e^{i\th}}}^{\mu,\nu}\Big\vert_{U_{e^{i\th},e^{i\th},e^{i\th}}^{\la,\mu,\nu}(Q)}.
\label{os7eq7}
\e
\item[{\rm(iii)}] Let\/ $Q_1, Q_2 \ra X\t\cS^1$ be principal\/ {\rm$\U(m_1)$-,\/ $\U(m_2)$-}bundles,\/ $\mu,\nu\in\R$, and $e^{i\th}\in\cS^1$. Define an open substack\/ $U_{e^{i\th},e^{i\th}}^{\mu,\nu}(Q_1,Q_2)\subset\B_{Q_1}\t\B_{Q_2}$ by
\begin{equation*}
\!\!\! U_{e^{i\th},e^{i\th}}^{\mu,\nu}(Q_1,Q_2)\!=\!\bigl(U_{e^{i\th},e^{i\th}}^{\mu,\nu}(Q_1)\t U_{e^{i\th},e^{i\th}}^{\mu,\nu}(Q_2)\bigr)\cap\Phi_{Q_1,Q_2}^{-1}\bigl(U_{e^{i\th},e^{i\th}}^{\mu,\nu}(Q_1\!\op Q_2)\bigr),
\end{equation*}
for\/ $\Phi_{Q_1,Q_2}:\B_{Q_1}\t\B_{Q_2}\ra \B_{Q_1\op Q_2}$ as in \eq{os2eq21}. Then there is a canonical isomorphism of principal\/ $\Z_2$-bundles on\/ $U_{e^{i\th},e^{i\th}}^{\mu,\nu}(Q_1,Q_2)\!:$
\end{itemize}
\e
(O^{\mu,\nu}_{Q_1\vert_{e^{i\th}}}\bt
O^{\mu,\nu}_{Q_2\vert_{e^{i\th}}})\vert_{U_{e^{i\th},e^{i\th}}^{\mu,\nu}(Q_1,Q_2)}\overset{\cong}{\longra}\Phi_{Q_1,Q_2}^*(O^{\mu,\nu}_{Q_1 \op Q_2\vert_{e^{i\th}}})\vert_{U_{e^{i\th},e^{i\th}}^{\mu,\nu}(Q_1,Q_2)}
\label{os7eq8}
\e
\end{lem}

\begin{proof} Part (i) is immediate. For (ii), consider the case $\la \leq \mu \leq \nu$. Then $\Eig^{\la,\mu}\op \Eig^{\mu,\nu}=\Eig^{\la,\nu}$, and taking determinants is compatible with direct sums, so \eq{os7eq7} follows. The other orders for $\la,\mu,\nu$ are similar. For (iii), the operator $D^{\nabla_{\bar Q_1 \ot_\C Q_2}\vert_{X\t\{e^{i\th}\}}}$ is complex linear, so its eigenspaces are canonically oriented. As $\nabla_{\Ad(Q_1\op Q_2)}=\nabla_{\Ad(Q_1)}\op\nabla_{\bar Q_1 \ot_\C Q_2}\op\nabla_{\Ad(Q_2)}$, the result follows from the compatibility of determinants with direct sums.
\end{proof}

\begin{prop}
\label{os7prop1}
Work in the situation of Definition\/ {\rm\ref{os7def2},} with\/ $(\th_1,\ldots,\th_k,\ab\la_1,\ab\ldots,\la_k)$ as defined there. Over\/ $U=U^{\la_1,\ldots,\la_k}_{[\th_1,\th_2],\ldots,[\th_k,\th_{k+1}]}(Q) \subset \B_Q$ we have a canonical isomorphism
\e
O^{E_\bu}_Q \ot_{\Z_2} \Nu_{Q,P}^*(O^{E_\bu}_{P\t\cS^1})\big\vert_U \longra \bigot\nolimits_{j=1}^k O^{\la_{j-1},\la_j}_{Q\vert_{e^{i\th_j}}}\big\vert_U.
\label{os7eq9}
\e
Moreover, \eq{os7eq9} has the following properties:
\begin{itemize}
\setlength{\itemsep}{0pt}
\setlength{\parsep}{0pt}
\item[{\rm(i)}] Suppose\/ $Q=P\t\cS^1,$ and\/ $\nabla_Q=\pi_X^*(\nabla_P),$ and\/ $k=1,$ as in Example\/ {\rm\ref{os7ex1}}. Then both sides of \eq{os7eq9} at\/ $[\nabla_Q]$ are canonically trivial, and\/ \eq{os7eq9} identifies these trivializations.
\item[{\rm(ii)}] Let\/ $Q_1, Q_2 \ra X\t\cS^1$ and\/ $P_1, P_2 \ra X$ be principal\/ {\rm$\U(m_1)$-, $\U(m_2)$}-bundles with\/ $P_a \cong Q_a\vert_{X\t\{1\}}$ for\/ $a=1,2$. Over the open substack\/
\e
\begin{split}
&U^{\la_1,\ldots,\la_k}_{[\th_1,\th_2],\ldots,[\th_k,\th_{k+1}]}(Q_1,Q_2)\\
&=
\bigl(U^{\la_1,\ldots,\la_k}_{[\th_1,\th_2],\ldots,[\th_k,\th_{k+1}]}(Q_1)\t U^{\la_1,\ldots,\la_k}_{[\th_1,\th_2],\ldots,[\th_k,\th_{k+1}]}(Q_2)\bigr)\\
&\qquad \cap\Phi_{Q_1,Q_2}^{-1}\bigl(U^{\la_1,\ldots,\la_k}_{[\th_1,\th_2],\ldots,[\th_k,\th_{k+1}]}(Q_1\!\op Q_2)\bigr)\subset\B_{Q_1}\t\B_{Q_2}
\end{split}
\label{os7eq10}
\e
we have a commutative diagram:
\begin{equation*}
\xymatrix@!0@C=265pt@R=40pt{
*+[r]{O^{E_\bu}_{Q_1} \ot \Nu_{Q_1,P_1}^*(O^{E_\bu}_{P_1\t\cS^1})
\boxtimes
O^{E_\bu}_{Q_2} \ot \Nu_{Q_2,P_2}^*(O^{E_\bu}_{P_2\t\cS^1})}\ar[dd]^{
\begin{subarray}{l}
\check\phi^{E_\bu}_{Q_1,Q_2}\ot(\Nu_{Q_1,P_1}\t \Nu_{Q_2,P_2})^*\\
\check\phi^{E_\bu}_{P_1\t\cS^1,P_2\t\cS^1}
\end{subarray}
}
\ar[rd]^(0.6){\eq{os7eq9}\bt\eq{os7eq9}}
\\
& *+[l]{\bigot_{j=1}^k O^{\la_{j-1},\la_j}_{Q_1\vert_{e^{i\th_j}}}
\boxtimes O^{\la_{j-1},\la_j}_{Q_2\vert_{e^{i\th_j}}}}\ar[dd]_{\bigot_{j=1}^k\eq{os7eq8}}
\\
*+[r]{O^{E_\bu}_{Q_1\op Q_2} \ot \Nu_{Q_1\op Q_2,P_1\op P_2}^*(O^{E_\bu}_{P_1\op P_2\t\cS^1})}
\ar[rd]_{\eq{os7eq9}}
\\
& *+[l]{\bigot_{j=1}^k O^{\la_{j-1},\la_j}_{Q_1\op Q_2\vert_{e^{i\th_j}}}.}
}
\end{equation*}
\item[{\rm(iii)}] If\/ $j=1,\ldots,k$ and\/ $\th_{j-1}<\th_*<\th_j,$ $\la_*>0$ we may replace\/ $(\th_1,\ab\ldots,\ab\th_k,\ab\la_1,\ab\ldots,\ab\la_k)$ by\/ {\rm\eq{os7eq2},} an \begin{bfseries}insertion\end{bfseries} in the sense of Definition\/ {\rm\ref{os7def2}}. This changes the right hand side of\/ \eq{os7eq9} over\/ $U\cap U^{\la_*}_{[\th_*,\th_j]}$ by replacing\/ $O^{\la_{j-1},\la_j}_{Q\vert_{e^{i\th_j}}}$ with\/ $O^{\la_{j-1},\la_*}_{Q\vert_{e^{i\th_*}}}\ot_{\Z_2}O^{\la_*,\la_j}_{Q\vert_{e^{i\th_j}}}$. The corresponding maps \eq{os7eq9} are then connected via the isomorphism
\e
\!\!\!\!\xymatrix@C=27pt{
O^{\la_{j-1},\la_j}_{Q\vert_{e^{i\th_j}}}
\ar[r]^(0.35){\eq{os7eq7}} &
O^{\la_{j-1},\la_*}_{Q\vert_{e^{i\th_j}}}\!\ot_{\Z_2}\!O^{\la_*,\la_j}_{Q\vert_{e^{i\th_j}}}
\ar[r]^{\eq{os7eq6}\ot\id} &
O^{\la_{j-1},\la_*}_{Q\vert_{e^{i\th_*}}}\!\ot_{\Z_2}\!O^{\la_*,\la_j}_{Q\vert_{e^{i\th_j}}}
}
\label{os7eq11}
\e
in the sense of a commutative triangle.
\end{itemize}
\end{prop}

\begin{proof} The isomorphism \eq{os7eq9} will be constructed by cutting $Y=X\t\cS^1$ into $Y'$ as in Definitions \ref{os7def2} and \ref{os7def3} and then deforming through elliptic boundary value problems on $Y'$. To set this up, we need more notation. For $\mu,\nu \in \R\cup\{\pm\iy\}$, $\nabla_Q \in \A_Q$, and $e^{i\th} \in\cS^1$, let $\Pi^{\mu,\nu}_{\nabla_{Q}\vert_{e^{i\th}}}$ be the $L^2$-orthogonal projection onto $\Eig^{\mu,\nu}\big(D^{\nabla_{\Ad(Q)}\vert_{X\t \{e^{i\th}\}}} \circ i_{F_0}\big)$ in Definition \ref{os7def1}. Here we include $i_{F_0}$, since comparing \eq{os5eq1} and \eq{osAeq1} shows that the boundary operator of $\ti{D}'$ is $-i_{F_0}\circ D = D\circ i_{F_0}$. However, we may identify all eigenspaces of $D^{\nabla_{\Ad(Q)}\vert_{X\t \{e^{i\th}\}}} \circ i_{F_0}$ with those of $D^{\nabla_{\Ad(Q)}\vert_{X\t \{e^{i\th}\}}}$ via the isomorphism $f \mapsto \ha(f-i_{F_0}f)$, which will be used to identify the corresponding orientation bundles of Definition \ref{os7def1}. Moreover, this proves that the eigenvalues agree so that $\Pi^{\mu,\nu}_{\nabla_{Q}\vert_{e^{i\th}}}$ becomes a continuous family of bounded operators over $U^{\mu,\nu}_{\th,\th}(Q)$. 

Let $Y', E'_\bu, Q'$ be as in Definitions \ref{os7def2} and \ref{os7def3}. Then we have a family of twisted elliptic operators $\ti D^{\prime\nabla_{\Ad(Q')}}$ on the compact manifold with boundary $Y'$ parametrized by $[\nabla_Q]$ in the topological stack $U$, where $\nabla_Q$ on $Y$ lifts to $\nabla_{Q'}$ on $Y'$ as in Definition \ref{os7def3}. Define a $U\t[0,1]$-family $B$ of elliptic boundary conditions $B_{([\nabla_Q],t)}$ for this $U$-family of operators as follows: at $([\nabla_Q],t)$ in $U\t[0,1]$ restrict the boundary values $\{f_{\th_{j-}}, f_{\th_{j+}}\}_{j=1}^k$ to belong to the subspace
\begin{equation*}
\left.
\begin{aligned}
\Pi^{\la_j,\iy}_{\nabla_{Q}\vert_{e^{i\th_j}}}(f_{\th_{j+}})
&=(1-t)\cdot \Pi^{\la_j,\iy}_{\nabla_{Q}\vert_{e^{i\th_j}}}(f_{\th_{j-}}),\\
\Pi^{-\la_j,\la_j}_{\nabla_{Q}\vert_{e^{i\th_j}}}(f_{\th_{j+}})
&=\Pi^{-\la_j,\la_j}_{\nabla_{Q}\vert_{e^{i\th_j}}}(f_{\th_{j-}}),\\
(1-t)\cdot \Pi^{-\iy,-\la_j}_{\nabla_{Q}\vert_{e^{i\th_j}}}(f_{\th_{j+}})
&=\Pi^{-\iy,-\la_j}_{\nabla_{Q}\vert_{e^{i\th_j}}}(f_{\th_{j-}}).
\end{aligned}
\enskip\right\}B(t)
\end{equation*}

Let $B(t)$ denote the restriction of $B$ to $U \t\{t\}$ for $t\in[0,1]$. The determinant line of \eq{os7eq4} constructed from $\ti{D}'^{\nabla_{\Ad(Q')}}$ with the boundary condition $B(0)$ can be identified with the determinant line of $\ti{D}^{\nabla_{\Ad(Q)}}$, using \eq{osAeq3}. Hence
\e
O^{E_\bu}_Q\vert_U \overset{\cong}{\longra}
O^{E'_\bu}_{Q'}\big(B(0)\big).
\label{os7eq12}
\e
Fibre transport in $O_{Q'}^{E'_\bu}(B) \ra U\t[0,1]$ along $[0,1]$ determines an isomorphism
\e
O^{E'_\bu}_{Q'}\big(B(0)\big)
\overset{\cong}{\longra}
O^{E'_\bu}_{Q'}\big(B(1)\big).
\label{os7eq13}
\e

Define a $U$-family of elliptic boundary conditions $B_{\th_1,\ldots, \th_k}^{\la_1, \ldots, \la_k}$ for the pullback of the $U$-family $[\nabla_Q]\mapsto\ti D^{\prime\nabla_{\Ad(Q')}}$ of elliptic operators on $Y'$ by restricting the boundary values $\{f_{\th_{j-}}, f_{\th_{j+}}\}_{j=1}^k$ at $[\nabla_Q] \in U$ to belong to the subspace
\[
\left.
\begin{aligned}
f_{\th_{j+}} &\in \Eig^{-\iy,\la_j}\Big(D^{\nabla_{\Ad(Q)}\vert_{X\t \{e^{i\th_j}_+\}}}\circ i_{F_0}\Big),\\
f_{\th_{j-}} &\in \Eig^{\la_j,\iy}\Big(D^{\nabla_{\Ad(Q)}\vert_{X\t \{e^{i\th_j}_-\}}}\circ i_{F_0}\Big).
\end{aligned}
\enskip\right\}B_{\th_1,\ldots, \th_k}^{\la_1, \ldots, \la_k}
\]

We wish to compare $B(1)$ to $B_{\th_1,\ldots, \th_k}^{\la_1, \ldots, \la_k}$. As neither is contained in the other, we introduce an intermediate $U$-family of elliptic boundary conditions $B'$ by
\[
\left.
\begin{aligned}
f_{\th_{j+}} &\in \Eig^{-\iy,\la_j}\Big(D^{\nabla_{\Ad(Q)}\vert_{X\t \{e^{i\th_j}_+\}}}\circ i_{F_0}\Big),\\
f_{\th_{j-}} &\in \Eig^{-\la_j,\iy}\Big(D^{\nabla_{\Ad(Q)}\vert_{X\t \{e^{i\th_j}_-\}}}\circ i_{F_0}\Big).
\end{aligned}
\enskip\right\}B'
\]
For $t=1$ the boundary condition $B(t)$ requires $f_{\th_{j+}}$ to be contained in the $(-\iy,\la_j)$-eigenspaces, $f_{\th_{j-}}$ to be in the $(-\la_j,+\iy)$-eigenspaces, and the matching condition $\Pi^{-\la_j,\la_j}_{\nabla_{Q}\vert_{e^{i\th_j}}}(f_{\th_{j+}})=\Pi^{-\la_j,\la_j}_{\nabla_{Q}\vert_{e^{i\th_j}}}(f_{\th_{j-}})$ on the overlap $(-\la_j,\la_j)$ of the spectrum where both $f_{\th_{j\pm}}$ may have a component. This last condition was not required for $B',$ so $B'/B\cong \Eig^{-\la_j,\la_j}\Big(D^{\nabla_{\Ad(Q)}\vert_{X\t \{e^{i\th_j}_+\}}}\circ i_{F_0}\Big).$ Similarly, the difference between $B_{\th_1,\ldots, \th_k}^{\la_1, \ldots, \la_k}$ and $B'$ is that $f_{\th_{j-}}$ in $B'$ is allowed to have a component in $\Eig^{-\la_j,\la_j}\Big(D^{\nabla_{\Ad(Q)}\vert_{X\t \{e^{i\th_j}_-\}}}\circ i_{F_0}\Big).$ Hence $B' / B(1) \cong B' / B_{\th_1,\ldots, \th_k}^{\la_1, \ldots, \la_k}$, and applying Proposition \ref{osAprop2} twice gives
\e
O^{E'_\bu}_{Q'}\big(B(1)\big)
\cong
O^{E'_\bu}_{Q'}\big(B_{\th_1,\ldots, \th_k}^{\la_1, \ldots, \la_k}\big).
\label{os7eq14}
\e
Combining \eq{os7eq12}--\eq{os7eq14} we conclude
\e
O^{E_\bu}_Q\vert_U
\overset{\cong}{\longra}
O^{E'_\bu}_{Q'}(B_{\th_1,\ldots, \th_k}^{\la_1, \ldots, \la_k}).
\label{os7eq15}
\e
For $1 \leq \ell \leq k$ define (recall the notation $Y_j'$ from Definition~\ref{os7def2})
\begin{align*}
\vartheta_\ell \colon [0,1]\t Y' &\longra Y',\\
\vartheta_\ell|_{[0,1]\t Y_j'}(t,(x,e^{i\th})) &=
\begin{cases}
(x,e^{i\th}), & j<\ell,\\
(x,e^{i[(1-t)\th + t\th_\ell]}), &j=\ell,\\
(x,e^{i[(1-t)\th_{\ell+1} + t\th_\ell]}), &j>\ell.\\
\end{cases}
\end{align*}
Write also $\vartheta_\ell(t) \colon Y' \ra Y'$, $(x,e^{i\th})\mapsto \vartheta_l(t,(x,e^{i\th}))$, $\vartheta_\ell(t,e^{i\th})\colon X \ra Y'$, $x\mapsto \vartheta_\ell(t,(x,e^{i\th}))$, and set $\vartheta_{k+1}(1)\coloneqq \id_{Y'}$. Note $\vartheta_\ell(1) = \vartheta_{\ell-1}(0)$ for $2\leq \ell \leq k+1$. For $1\leq \ell \leq k$ define a $(U \t [0,1])$-family of elliptic boundary conditions $B_\ell$ for the pullback of the $\B_{Q'}$-family of twisted elliptic operators $[\nabla_{Q'}]\mapsto\ti D^{\prime\nabla_{\Ad(Q')}}$ along $U\t[0,1]\t Y' \xrightarrow{\id_{U}\t \vartheta_\ell} U\t Y' \xrightarrow{[\nabla_Q]\mapsto[\nabla_{Q'}]} \B_{Q'}\t Y'$ by restricting the boundary values $\{f_{\th_{j-}}, f_{\th_{j+}}\}_{j=1}^k$ at the parameter $([\nabla_Q],t)$ in $U \t [0,1]$ to belong to the subspace
\begin{align*}
f_{\th_{j+}} &\in \Eig^{-\iy,\la_j}
\big(D^{\vartheta_\ell(t,e^{i\th_j}_+)^*\nabla_{\Ad(Q')}}\circ i_{F_0}\big),
&&
1\leq j \leq \ell,\\
f_{\th_{j+}} &\in \Eig^{-\iy,\la_\ell}\big(D^{\vartheta_\ell(t,e^{i\th_j}_+)^*\nabla_{\Ad(Q')}}\circ i_{F_0}\big),
&&
\ell < j \leq k,\\
f_{\th_{1-}} &\in \Eig^{\la_\ell,\iy}\big(D^{\vartheta_\ell(t,e^{i\th_{1-}})^*\nabla_{\Ad(Q')}}\circ i_{F_0}\big),
&& j=1,\\
f_{\th_{j-}} &\in \Eig^{\la_j,\iy}\big(D^{\vartheta_\ell(t,e^{i\th_j}_-)^*\nabla_{\Ad(Q')}}\circ i_{F_0}\big),
&&
2\leq j \leq \ell,\\
f_{\th_{j-}} &\in \Eig^{\la_\ell,\iy}\big(D^{\vartheta_\ell(t,e^{i\th_j}_-)^*\nabla_{\Ad(Q')}}\circ i_{F_0}\big),
&&
\ell < j \leq k.
\end{align*}
Write $B_\ell(t)$ for the restriction of $B_\ell$ to $U\t\{t\}$. Also, set $B_{k+1}(1)=B_{\th_1,\ldots, \th_k}^{\la_1, \ldots, \la_k}$. For $\ell=2,\ldots,k+1$ the boundary conditions $B_\ell(1)$ and $B_{\ell - 1}(0)$ differ by $\Eig^{\la_{\ell-1}, \la_\ell}\big(D^{\nabla_{\Ad(Q)}\vert_{X\t \{e^{i\th_\ell}\}}}\circ i_{F_0}\big)$, so as in \eq{os7eq14} applying \eq{osAeq2} twice gives
\begin{equation*}
O^{E'_\bu}_{Q'}\big( B_\ell(1) \big)
\cong
O^{E'_\bu}_{Q'}\big( B_{\ell-1}(0) \big)
\ot
O^{\la_{\ell-1},\la_\ell}_{Q\vert_{e^{i\th_\ell}}}\big\vert_U.
\end{equation*}
Using this and deforming for each $\ell = k, \ldots, 1$ using fibre transport through $t\in [0,1]$ as in \eq{os7eq13} gives the following isomorphisms over $U$:
\e
\begin{aligned}
O^{E_\bu}_Q\big\vert_U \hspace{1ex}&\overset{\mathclap{\eq{os7eq15}}}{\cong}\hspace{1ex} O^{E'_\bu}_{Q'}\big(B_{\th_1,\ldots, \th_k}^{\la_1, \ldots, \la_k}\big)\\
&=\hspace{1ex} O^{E'_\bu}_{Q'}(B_{k+1}(1))\\
&\cong\hspace{1ex} O^{E'_\bu}_{Q'}(B_k(0)) \ot O^{\la_k,\la_{k+1}}_{Q\vert_{e^{i\th_{k+1}}}}\big\vert_U\\
&\cong\hspace{1ex} O^{E'_\bu}_{\vartheta_k(1)^*Q'}(B_k(1)) \ot O^{\la_0,\la_1}_{Q\vert_{e^{i\th_1}}}\big\vert_U\\
&\cong\hspace{1ex} O^{E'_\bu}_{\vartheta_{k-1}(0)^*Q'}(B_{k-1}(0)) \ot O^{\la_0,\la_1}_{Q\vert_{e^{i\th_1}}}\big\vert_U \ot O^{\la_{k-1},\la_k}_{Q\vert_{e^{i\th_k}}}\big\vert_U\\
&\cong\hspace{1ex} O^{E'_\bu}_{\vartheta_{k-1}(1)^*Q'}(B_{k-1}(1)) \ot O^{\la_0,\la_1}_{Q\vert_{e^{i\th_1}}}\big\vert_U \ot O^{\la_{k-1},\la_k}_{Q\vert_{e^{i\th_k}}}\big\vert_U\\
&\enskip\vdots\\
&\cong\hspace{1ex} O^{E'_\bu}_{\vartheta_1(1)^*Q'}(B_1(1)) \ot \ts\bigot_{j=1}^k O^{\la_{j-1},\la_j}_{Q\vert_{e^{i\th_j}}}\big\vert_U\\
&\overset{\mathclap{\eq{os7eq15}}}{\cong}\hspace{1ex} \Nu_{Q,P}^*(O^{E_\bu}_{P\t\cS^1}) \ot\ts\bigot_{j=1}^k O^{\la_{j-1},\la_j}_{Q\vert_{e^{i\th_j}}}\big\vert_U.
\end{aligned}
\label{os7eq16}
\e
In the last step we apply the inverse isomorphism of \eq{os7eq15} with $P\t\cS^1$ in place of $Q$. The claimed isomorphism \eq{os7eq9} follows.

Property (i) for $Q=P\t\cS^1$ and $\nabla_Q=\pi_X^*(\nabla_P)$ follows by inspection of \eq{os7eq16}: as the family of connections is constant, the isomorphisms in the third and fourth line of \eq{os7eq16} reduce to the equalities $O^{E'_\bu}_{Q'}(B_{2}(1))=O^{E'_\bu}_{Q'}(B_{1}(0))=O^{E'_\bu}_{Q'}(B_{1}(1))$ and then we apply the inverse isomorphism of \eq{os7eq15}, so that \eq{os7eq16} reduces to the identity map $O^{E_\bu}_{P\t\cS^1}\big\vert_U = \Nu_{Q,P}^*(O^{E_\bu}_{P\t\cS^1})$. Property (ii) holds since all of our constructions are compatible with direct sums. Finally (iii) follows by direct inspection of \eq{os7eq15} and \eq{os7eq16}.
\end{proof}

Square roots are connected to orientations by the following lemma.

\begin{lem}
\label{os7lem3}
Let\/ $V$ be a complex vector space equipped with a real structure $V\to\bar{V},$ $v\mapsto\bar{v}.$ Let\/ ${\om \in \left(\det_\C V\right)^{\ot 2}}$ be real, non-zero, and suppose a real square root of\/ $\om$ exists \textup(meaning there is \/ $\eta\in\det_\C V$ with\/ $\eta\ot\eta=\om$ and\/ $\bar\eta=\eta,$ using the induced real structure on $\det_\C V$\textup). Then the set of solutions\/ $\eta \in \det_\C V$ to\/ $\eta^{\ot 2} = \omega$ is canonically identified with the\/ $\Z_2$-torsor\/ $O(V_\R)$ of orientations of the real form\/ $V_\R=\{v\in V\mid v=\bar{v}\}.$
\end{lem}

\begin{proof}
As $\det_\C V$ is one-dimensional, the two solutions $\eta^{\ot2}=\om$ are $\pm\eta,$ which by assumption are automatically non-zero and real. Using the isomorphism $(\det_\C V)_\R\cong\det_\R(V_\R)$ we get a map $\{\eta\mid \eta^{\ot 2}=\om\}\to(\det_\R(V_\R)\setminus\{0\})/\R_{>0}=O(V_\R)$ of $\Z_2$-torsors, and such maps are always bijective.
\end{proof}

\begin{prop}
\label{os7prop2}
Work in the situation of Definition\/ {\rm\ref{os7def2},} with\/ $(\th_1,\ldots,\th_k,\ab\la_1,\ab\ldots,\la_k)$ as defined there. Over\/ $U=U^{\la_1,\ldots,\la_k}_{[\th_1,\th_2],\ldots,[\th_k,\th_{k+1}]}(Q) \subset \B_Q$ we have a canonical isomorphism
\e
\bigot\nolimits_{j=1}^k O^{\la_j,\la_{j+1}}_{Q\vert_{e^{i\th_j}}}\big\vert_U
\longra\Ga_{Q,P}^*(M_P^{F_\bu})\big\vert_U.
\label{os7eq17}
\e
Moreover, \eq{os7eq17} has the following properties:
\begin{itemize}
\setlength{\itemsep}{0pt}
\setlength{\parsep}{0pt}
\item[{\rm(i)}] Suppose\/ $Q=P\t\cS^1,$ and\/ $\nabla_Q=\pi_X^*(\nabla_P),$  and\/ $k=1,$ as in Example\/ {\rm\ref{os7ex1}}. Then both sides of\/ \eq{os7eq17} are canonically trivial, and\/ \eq{os7eq17} identifies these trivializations.
\item[{\rm(ii)}] Let\/ $Q_1, Q_2 \ra X\t\cS^1$ and\/ $P_1, P_2 \ra X$ be principal\/ {\rm$\U(m_1)$-, $\U(m_2)$}-bundles with\/ $P_a \cong Q_a\vert_{X\t\{1\}}$ for\/ $a=1,2$. Over\/ $U^{\la_1,\ldots,\la_k}_{[\th_1,\th_2],\ldots,[\th_k,\th_{k+1}]}(Q_1,Q_2)\subset\B_{Q_1}\t\B_{Q_2}$ in \eq{os7eq10} we have a commutative diagram:
\begin{equation*}
\!\!\!\!\xymatrix@!0@C=295pt@R=40pt{
*+[r]{\Ga_{Q_1,P_1}^*(M_{P_1}^{F_\bu}) \boxtimes \Ga_{Q_2,P_2}^*(M_{P_2}^{F_\bu})} \ar[d]^{\eq{os5eq3}}\ar[r]_(0.52){\eq{os7eq17}\bt\eq{os7eq17}}
&
*+[l]{\bigot_{j=1}^k O^{\la_j,\la_{j+1}}_{Q_1\vert_{e^{i\th_j}}}
\boxtimes O^{\la_j,\la_{j+1}}_{Q_2\vert_{e^{i\th_j}}}}\ar[d]_{\bigot_{j=1}^k\eq{os7eq8}}
\\
*+[r]{\Ga_{Q_1\op Q_2,P_1\op P_2}^*(M_{P_1\op P_2}^{F_\bu})}\ar[r]^(0.52){\eq{os7eq17}}
&
*+[l]{\bigot_{j=1}^k O^{\la_j,\la_{j+1}}_{Q_1\op Q_2\vert_{e^{i\th_j}}}.\!}
}
\end{equation*}
\item[{\rm(iii)}] If\/ $j=1,\ldots,k$ and\/ $\th_{j-1}<\th_*<\th_j,$ $\la_*>0$ we may replace\/ $(\th_1,\ab\ldots,\ab\th_k,\ab\la_1,\ab\ldots,\ab\la_k)$ by {\rm\eq{os7eq2},} an \begin{bfseries}insertion\end{bfseries} in the sense of Definition\/ {\rm\ref{os7def2}}. This changes the left hand side of\/ \eq{os7eq17} over\/ $U\cap U^{\la_*}_{[\th_*,\th_j]}$ by replacing\/ $O^{\la_{j-1},\la_j}_{Q\vert_{e^{i\th_j}}}$ with\/ $O^{\la_{j-1},\la_*}_{Q\vert_{e^{i\th_*}}}\ot_{\Z_2}O^{\la_*,\la_j}_{Q\vert_{e^{i\th_j}}}$. The corresponding maps \eq{os7eq17} are then connected via\/ {\rm\eq{os7eq11},} in the sense of a commutative triangle.
\end{itemize}
\end{prop}

\begin{proof} Recall from Definition \ref{os5def2} that $D$ is antilinear self-adjoint, $D^*=\bar D$. Define an open substack $V^\la =\bigl\{[\nabla_P] \in \B_P : \la^2 \notin \Spec(\bar D{}^{\nabla_{\Ad(P)}}D^{\nabla_{\Ad(P)}})\bigr\}\subset\B_P$ for $\la>0$. By definition of the topology for determinant line bundles (e.g.\ \cite[Def.~3.4]{Upme}), for any $\la > 0$ and $[\nabla_P] \in V^\la$, we have canonical isomorphisms
\e
\begin{aligned}
& K^{F_\bu}_P\vert_{[\nabla_P]}\\
\overset{\cong}{\longra}&\det_\C \Eig^{0,\la^2}\big(
\bar D{}^{\nabla_{\Ad(P)}}D^{\nabla_{\Ad(P)}}
\big)\ot\det_\C \Eig^{0,\la^2}\big(
D^{\nabla_{\Ad(P)}}\bar D{}^{\nabla_{\Ad(P)}}
\big)^*\\
\overset{\cong}{\longra}&\det_\C \Eig^{0,\la^2}\big(
\bar D{}^{\nabla_{\Ad(P)}}D^{\nabla_{\Ad(P)}}
\big)^{\ot 2},
\end{aligned}
\label{os7eq18}
\e
using the Hermitian metrics on the eigenbundles in the second step. These isomorphisms are continuous over $V^\la$, so we have a canonical square root $\big(K^{F_\bu}_P\big)^{1/2}_\la$ of the complex determinant line bundle over the open substack $V^\la\subset\B_P$. Define
\begin{align*}
\Ga : U \t \cS^1 &\longra \B_P, & \Ga: ([\nabla_Q],e^{i\th}) &\longmapsto
[\nabla_Q\vert_{X\t\{e^{i\th}\}}]
\end{align*}
and let $\Ga_j$ be the restriction of $\Ga$ to $U \t \cS^1_{[\th_j, \th_{j+1}]} \ra V^{\la_j}$ for $j=1,\ldots,k$. We shall construct a square root of $\Ga^*(K^{F_\bu}_P)$ by patching together the square roots $\Ga_j^*\bigl((K^{F_\bu}_P){}^{1/2}_{\la_j}\bigr) \ra U \t \cS^1_{[\th_j, \th_{j+1}]}$ over the intersection points $e^{i\th_1},\ldots,e^{i\th_k}$, compatible with \eq{os7eq18}. By construction of $\big(K^{F_\bu}_P\big){}^{1/2}_{\la_j}$ and $\big(K^{F_\bu}_P\big){}^{1/2}_{\la_{j+1}}$ this requires isomorphisms
\e
\begin{aligned}
&\det_\C \Eig^{0,\la_j^2}\big(\bar D{}^{\nabla_{\Ad(Q)}\vert_{X\t \{e^{i\th_j}\}}}D^{\nabla_{\Ad(Q)}\vert_{X\t \{e^{i\th_j}\}}}\big)\\
&\longra\det_\C \Eig^{0,\la_{j+1}^2}\big(\bar D{}^{\nabla_{\Ad(Q)}\vert_{ X\t \{e^{i\th_j}\}}}D^{\nabla_{\Ad(Q)}\vert_{X\t \{e^{i\th_j}\}}}\big).
\end{aligned}
\label{os7eq19}
\e
For compatibility with \eq{os7eq18}, and by the way \eq{os7eq18} is constructed (see e.g.\ \cite[(3.5)]{Upme}), these isomorphisms must square to multiplication by
\begin{align*}
\om_j =&\; \Big(v_1 \wedge \cdots \wedge v_\ell\Big) \ot \Big(D^{\nabla_{\Ad(Q)}\vert_{X\t \{e^{i\th_j}\}}}v_1 \wedge\cdots\wedge D^{\nabla_{\Ad(Q)}\vert_{X\t \{e^{i\th_j}\}}}v_\ell \Big)\\
&\in \det_\C \Eig^{\la_j^2,\la_{j+1}^2}\big(\bar D{}^{\nabla_{\Ad(Q)}\vert_{X\t \{e^{i\th_j}\}}}D^{\nabla_{\Ad(Q)}\vert_{X\t \{e^{i\th_j}\}}}\big)^{\ot 2}.
\end{align*}
Here $v_1,\ldots, v_\ell \in \Eig^{\la_j^2,\la_{j+1}^2}\big(\bar D{}^{\nabla_{\Ad(Q)}\vert_{X\t \{e^{i\th_j}\}}}D^{\nabla_{\Ad(Q)}\vert_{X\t \{e^{i\th_j}\}}}\big)$ denotes an arbitrary {\it complex}\/ orthonormal basis with respect to $i_{F_0}$. For example, one may take $v_1,\ldots, v_\ell$ to be a real orthonormal basis of eigenfunctions of the corresponding real form $\Eig^{\la_j,\la_{j+1}}(D^{\nabla_{\Ad(Q)}\vert_{X\t \{e^{i\th_j}\}}})$, where $D$ is now viewed as a self-adjoint operator over $\R$. This example shows that $\om_j$ is real and, since $\la_j, \la_{j+1} >0$, that $\om_j$ has real square roots. Therefore, Lemma \ref{os7lem3} canonically parametrizes the two square roots of $\omega_j$ by $O^{\la_j,\la_{j+1}}_{Q\vert_{e^{i\th_j}}}$.  From $\bigotimes_{j=1}^k O^{\la_j,\la_{j+1}}_{Q\vert_{e^{i\th_j}}}\big\vert_U$ we may therefore construct a square root of $\Ga^*(K^{F_\bu}_P) \ra U \t \cS^1$, an element of $\Ga_{Q,P}^*(M_P^{F_\bu})\big\vert_U$, by patching with multiplication by the square root of $\om_j$ fixed by the orientations. As flipping the orientation twists the resulting square root once, this map factors through an isomorphism~\eq{os7eq17}.

For (i) note that \eq{os7eq19} is the identity map in this case (recall that $\la_{k+1}=\la_1$), and the canonical trivialization of $O_{Q\vert_{e^{i\th_1}}}^{\la_1,\la_1}$ corresponds to the identity square root. Property (ii) is obvious, as all our constructions are compatible with direct sums. Finally, in (iii) one factors the patching of $[\th_{j-1},\th_j]$ to $[\th_j,\th_{j+1}]$ and the isomorphism \eq{os7eq19} into one additional step at $e^{i\th_*}$, where the square root gets twisted further according to the orientation.
\end{proof}

We can now complete the proof of Theorem \ref{os5thm2}. Let $X,F_\bu,G$, $P\ra X$ and $Q\ra X\t\cS^1$ be given, as in Definitions \ref{os5def1} and \ref{os5def2}. Suppose $(\th_1,\ldots,\th_k,\la_1,\ab\ldots,\la_k)$ is as in Definition \ref{os7def2}, and write $U=U^{\la_1,\ldots,\la_k}_{[\th_1,\th_2],\ldots,[\th_k,\th_{k+1}]}(Q) \subset \B_Q$ as in Propositions \ref{os7prop1} and \ref{os7prop2}. Consider the commutative diagram of principal $\Z_2$-bundles on~$U$:
\e
\begin{gathered}
\xymatrix@C=180pt@R=15pt{
*+[r]{\check O_Q^{E_\bu}\ot_{\Z_2}\Nu_{Q,P}^*(\check O_{P\t\cS^1}^{E_\bu})\vert_U} \ar[r]_(0.55){\ga_{Q,P}^{F_\bu}\vert_U} \ar[d]^{\eq{os7eq3}} & *+[l]{\Ga_{Q,P}^*(M^{F_\bu}_P)\vert_U} \\
*+[r]{O^{E_\bu}_Q \ot_{\Z_2} \Nu_{Q,P}^*(O^{E_\bu}_{P\t\cS^1})\vert_U} \ar[r]^(0.55){\eq{os7eq9}} & *+[l]{\bigot_{j=1}^k O^{\la_j,\la_{j+1}}_{Q\vert_{e^{i\th_j}}}\vert_U.\!} \ar[u]^{\eq{os7eq17}} }	
\end{gathered}
\label{os7eq20}
\e
We claim that there exists a unique morphism $\ga_{Q,P}^{F_\bu}$ in \eq{os5eq4} such that \eq{os7eq20} commutes for all such~$(\th_1,\ldots,\th_k,\la_1,\ab\ldots,\la_k)$.

To prove this, first observe that the first part of Lemma \ref{os7lem1} implies that the open substacks $U\subset\B_Q$ from different choices of $(\th_1,\ldots,\la_k)$ cover $\B_Q$. Hence the claim holds if and only if whenever $(\th_1,\ldots,\la_k)$ and $(\th'_1,\ldots,\la'_{k'})$ are alternative choices yielding open substacks $U,U'\subset\B_Q$, then the values of $\ga_{Q,P}^{F_\bu}\vert_U,\ga_{Q,P}^{F_\bu}\vert_{U'}$ determined by \eq{os7eq20} for $(\th_1,\ldots,\la_k)$ and $(\th'_1,\ldots,\la'_{k'})$ agree on the overlap~$U\cap U'\subset\B_Q$.

So let $[\nabla_Q]\in U\cap U'$. Then $(\th_1,\ldots,\la_k)$ and $(\th'_1,\ldots,\la'_{k'})$ are cuts for $[\nabla_Q]$, as in Definition \ref{os7def2}, and the second part of Lemma \ref{os7lem1} says that we can join $(\th_1,\ldots,\la_k)$ and $(\th'_1,\ldots,\la'_{k'})$ in cuts for $[\nabla_Q]$ by a finite chain of insertions and deletions. Therefore, to show that $\ga_{Q,P}^{F_\bu}\vert_U,\ga_{Q,P}^{F_\bu}\vert_{U'}$ agree near $[\nabla_Q]$, it is enough to show that they agree near $[\nabla_Q]$ when $(\th_1,\ldots,\la_k)$ and $(\th'_1,\ldots,\la'_{k'})$ are related by an insertion or a deletion, say if $(\th'_1,\ldots,\la'_{k'})$ is \eq{os7eq2}. This follows from Propositions \ref{os7prop1}(iii) and \ref{os7prop2}(iii), and thus the claim holds.

Now that $\ga_{Q,P}^{F_\bu}$ in \eq{os5eq4} is well-defined, Theorem \ref{os5thm2}(i) is a direct consequence of Propositions \ref{os7prop1}(i) and \ref{os7prop2}(i), since without loss we may take $k=1$, $\th_1=0$ and any $\la_1\in(0,\iy)\sm \Spec D^{\nabla_{\Ad(P)}}$. Similarly, Theorem \ref{os5thm2}(ii) follows from Propositions \ref{os7prop1}(ii) and \ref{os7prop2}(ii), as one can show in a similar way to Lemma \ref{os7lem1} that the open substacks $U^{\la_1,\ldots,\la_k}_{[\th_1,\th_2],\ldots,[\th_k,\th_{k+1}]}(Q_1,Q_2)\subset\B_{Q_1}\t\B_{Q_2}$ in \eq{os7eq10} for all $(\th_1,\ldots,\la_k)$ cover $\B_{Q_1}\t\B_{Q_2}$. This completes the proof of Theorem~\ref{os5thm2}.

\section{Proof of Theorem \ref{os5thm3}}
\label{os8}

Work in the situation of Definitions \ref{os5def1} and \ref{os5def2} and Theorem \ref{os5thm3}, with $X,F_\bu$ and $E_\bu$ on $X\t\cS^1$ fixed. We must establish 1-1 correspondences between choices of data (a),(b),(c) and (d) in Theorem \ref{os5thm3}. We will prove:
\begin{itemize}
\setlength{\itemsep}{0pt}
\setlength{\parsep}{0pt}
\item[(i)] Data (a) is in natural 1-1 correspondence with data (b);
\item[(ii)] Data (b) is in natural 1-1 correspondence with data (c); and
\item[(iii)] Data (c) is in natural 1-1 correspondence with data (d).
\end{itemize}
The theorem follows. For (i), Theorem \ref{os5thm1} gives a 1-1 correspondence between data (a),(b), by showing that any $\hat\si_{X\t\U(N)}^{F_\bu}$ in (a) extends uniquely to $\si_P^{F_\bu}$ in (b) for all $\U(m)$-bundles $P\ra X$ such that~$\si_{X\t\U(N)}^{F_\bu}=\hat\si_{X\t\U(N)}^{F_\bu}$.

For (ii), here is how to map from data (b) to data (c). Suppose we are given spin structures $\si_P^{F_\bu}$ as in (b). Then for each $\ga\in\Map_{C^0}(\cS^1,\B_P)$, $\ga^*(\si_P^{F_\bu})$ is a square root of $\ga^*(K_P^{F_\bu})$, and thus a point of $M^{F_\bu}_P\vert_\ga$. This defines a trivialization of $M^{F_\bu}_P\ra \Map_{C^0}(\cS^1,\B_P)$. Hence \eq{os5eq4} in Theorem \ref{os5thm2} induces a trivialization of $\check O_Q^{E_\bu}\ot_{\Z_2}\Nu_{Q,P}^*(\check O_{P\t\cS^1}^{E_\bu})$, as in (c). Theorem \ref{os5thm2}(i),(ii) imply~(c)(i),(ii).

We will prove that this map (b) $\Ra$ (c) is a bijection. Let $P\ra X$ be a principal $\U(m)$-bundle, and consider spin structures on $\B_P$, that is, equivalence classes $[L,\io]$ of a line bundle $L\ra\B_P$ and isomorphism $\io:L^{\ot^2}\ra K_P^{F_\bu}$. Write $\dot L=L\sm 0(\B_P),$ $\dot K_P^{F_\bu}=K_P^{F_\bu}\sm 0(\B_P)$ for the complements of the zero sections, which are principal $\C^*$-bundles over $\B_P$. Define a morphism $\jmath:\dot L\ra\dot K_P^{F_\bu}$ by $\jmath(l)=\io(l\ot l)$. Then $\jmath$ is a double cover of $\dot K_P^{F_\bu}$, so we can make it into a principal $\Z_2$-bundle over $\dot K_P^{F_\bu}$, with $\Z_2$-action multiplication by $\pm 1$ on the fibres of~$\dot L\subset L$.

As $\B_P,\C^*$ are connected, $\dot K_P^{F_\bu}$ is connected, so fixing a base-point $x_0\in\dot K_P^{F_\bu}$ to define the fundamental group $\pi_1(\dot K_P^{F_\bu})$, principal $\Z_2$-bundles on $\dot K_P^{F_\bu}$ up to isomorphism are in 1-1 correspondence with group morphisms $\la_P:\pi_1(\dot K_P^{F_\bu})\ra\Z_2$. Write $[\nabla_P]=\pi(x_0)\in\B_P$. The inclusion $\C^*\cdot x_0=\dot K_P^{F_\bu}\vert_{[\nabla_P]}\hookra\dot K_P^{F_\bu}$ induces a morphism $\mu_P:\Z=\pi_1(\C^*)\ra\pi_1(\dot K_P^{F_\bu})$. Over this $\C^*$-fibre $\dot K_P^{F_\bu}\vert_{[\nabla_P]}$, $\jmath:\dot L\ra\dot K_P^{F_\bu}$ is the nontrivial double cover $\C^*\ra\C^*$ mapping $z\mapsto z^2$, so $\la_P\ci\mu_P:\Z\ra\Z_2$ maps~$n\mapsto (-1)^n$.

A morphism $\la_P:\pi_1(\dot K_P^{F_\bu})\ra\Z_2$ corresponds to a principal $\Z_2$-bundle $\jmath:\dot L\ra\dot K_P^{F_\bu}$ coming from a spin structure $[L,\io]$ on $\B_P$ if and only if $\la_P\ci\mu_P(n)=(-1)^n$ for $n\in\Z$. This gives a 1-1 correspondence between spin structures $\si_P^{F_\bu}=[L,\io]$ on $\B_P$, as in (b), with group morphisms $\la_P:\pi_1(\dot K_P^{F_\bu})\ra\Z_2$ such that $\la_P\ci\mu_P(n)=(-1)^n$ for~$n\in\Z$.

Let $[\de]\in \pi_1(\dot K_P^{F_\bu})$ be the equivalence class of a continuous loop $\de:\cS^1\ra\dot K_P^{F_\bu}$. We may take $\de$ to be smooth in the appropriate sense. Then $\bar\de=\pi\ci\de:\cS^1\ra\B_P$ is a (smooth) loop in $\B_P$ based at $[\nabla_P]$, so $\bar\de\in\Map_{C^0}(\cS^1,\B_P)$. This determines a (smooth) principal $\U(m)$-bundle $Q_{\bar\de}\ra X\t\cS^1$, unique up to isomorphism, with $Q_{\bar\de}\vert_{X\t\{1\}}\cong P$, and a partial connection on $Q_{\bar\de}$ in the $X$ directions on $X\t\cS^1$. We may lift this to a connection $\nabla_{Q_{\bar\de}}$ on $Q_{\bar\de}$, and hence a point $[\nabla_{Q_{\bar\de}}]\in\B_{Q_{\bar\de}}$. Definition \ref{os5def1} implies that~$\Ga_{Q_{\bar\de},P}([\nabla_{Q_{\bar\de}}])=\bar\de$.

Then $\de:\cS^1\ra\dot K_P^{F_\bu}$ is a nonvanishing section of $\bar\de^*(K_P^{F_\bu})\ra\cS^1$, and so determines a trivialization of $\bar\de^*(K_P^{F_\bu})\ra\cS^1$, and hence a square root of $\bar\de^*(K_P^{F_\bu})\ra\cS^1$, that is, a point $m_\de$ of $M_P^{F_\bu}\vert_{\bar\de}$. From \eq{os5eq4} we have
\begin{equation*}
\ga_{Q_{\bar\de},P}^{F_\bu}\vert_{[\nabla_{Q_{\bar\de}}]}:(\check O_{Q_{\bar\de}}^{E_\bu}\ot_{\Z_2}\Nu_{Q_{\bar\de},P}^*(\check O_{P\t\cS^1}^{E_\bu}))\vert_{[\nabla_{Q_{\bar\de}}]}\,{\buildrel\cong\over\longra}\,M^{F_\bu}_P\vert_{\bar\de},
\end{equation*}
since $\Ga_{Q_{\bar\de},P}([\nabla_Q])=\bar\de$. Thus $\ga_{Q_{\bar\de},P}^{F_\bu}\vert^{-1}_{[\nabla_{Q_{\bar\de}}]}(m_\de)\in (\check O_{Q_{\bar\de}}^{E_\bu}\ot_{\Z_2}\Nu_{Q_{\bar\de},P}^*(\check O_{P\t\cS^1}^{E_\bu}))\vert_{[\nabla_{Q_{\bar\de}}]}$.

Suppose now we are given data $\check\Om^{E_\bu}_Q$ for all $\U(m)$-bundles $Q\ra X\t\cS^1$ as in Theorem \ref{os5thm3}(c). Define $\la_P:\pi_1(\dot K_P^{F_\bu})\ra\Z_2$ by $\la_P([\de])=\check\Om^{E_\bu}_{Q_{\bar\de}}(\ga_{Q_{\bar\de},P}^{F_\bu}\vert^{-1}_{[\nabla_{Q_{\bar\de}}]}(m_\de))$. As $\check\Om^{E_\bu}_{Q_{\bar\de}}(\ga_{Q_{\bar\de},P}^{F_\bu}\vert^{-1}_{[\nabla_{Q_{\bar\de}}]}(m_\de))$ is unchanged under continuous deformations of $\de:\cS^1\ra\dot K_P^{F_\bu}$, this $\la_P$ is well defined as a map of sets. We will prove that $\la_P:\pi_1(\dot K_P^{F_\bu})\ra\Z_2$ is a group morphism, with~$\la_P\ci\mu_P(n)=(-1)^n$. 

First let $\de_0:\cS^1\ra\dot K_P^{F_\bu}$ be the constant loop at $x_0$, so $[\de_0]$ is the identity in $\pi_1(\dot K_P^{F_\bu})$. Then $\bar\de_0=\pi\ci\de_0$ is the constant loop at $[\nabla_P]$, and we may take $[\nabla_Q]$ above to be $[\nabla_Q]=\Pi_P([\nabla_P])$. As in Theorems \ref{os5thm2}(i) and \ref{os5thm3}(c)(i), the fibres $M_P^{F_\bu}\vert_{\bar\de_0}$ and $(\check O_{Q_{\bar\de_0}}^{E_\bu}\ot_{\Z_2}\Nu_{Q_{\bar\de_0},P}^*(\check O_{P\t\cS^1}^{E_\bu}))\vert_{[\nabla_{Q_{\bar\de_0}}]}$ are canonically trivial, and $\ga_{Q_{\bar\de_0},P}^{F_\bu}\vert_{[\nabla_{Q_{\bar\de_0}}]}$ and $\check\Om^{E_\bu}_{Q_{\bar\de_0}}$ preserve these trivializations. Since $m_{\de_0}$ corresponds to $1\in\Z_2$ under the trivialization $M_P^{F_\bu}\vert_{\bar\de_0}\cong\Z_2$, we see that $\la_P([\de_0])=1$, so $\la_P$ preserves identities.

More generally, let $n\in\Z$ and define $\de_n:\cS^1\ra\dot K_P^{F_\bu}$ by $\de_n(e^{i\th})=e^{in\th}\cdot x_0$, so $[\de_n]=\mu_P(n)$. Then $\bar\de_n=\pi\ci\de_n$ is again the constant loop at $[\nabla_P]$, but now $m_{\de_n}$ corresponds to $(-1)^n\in\Z_2$ under the trivialization $M_P^{F_\bu}\vert_{\bar\de_n}\cong\Z_2$. Hence $\la_P\ci\mu_P(n)=\la_P([\de_n])=(-1)^n$, as we have to prove.

Next let $\ep_1,\ep_2,\ep_{12}:\cS^1\ra\dot K_P^{F_\bu}$ be smooth loops based at $x_0$ with $\ep_{12}$ isotopic to the (piecewise smooth) composition $\ep_1*\ep_2$ used to define multiplication `$\cdot$' in $\pi_1(\dot K_P^{F_\bu})$, so that $[\ep_1]\cdot[\ep_2]=[\ep_{12}]=[\de_0]\cdot[\ep_{12}]$ in $\pi_1(\dot K_P^{F_\bu})$. Let $\bar\de_0,\bar\ep_1,\bar\ep_2,\bar\ep_{12}$ be the loops in $\B_P$, and $Q_{\bar\de_0},Q_{\bar\ep_1},Q_{\bar\ep_2},Q_{\bar\ep_{12}}$ the principal $\U(m)$-bundles over $X\t\cS^1$, and $\nabla_{Q_{\bar\de_0}},\nabla_{Q_{\bar\ep_1}},\nabla_{Q_{\bar\ep_2}},\nabla_{Q_{\bar\ep_{12}}}$ the connections, corresponding to $\de_0,\ep_1,\ep_2,\ep_{12}$ as above. Then $Q_{\bar\ep_1}\op Q_{\bar\ep_2}$ and $Q_{\bar\de_0}\op Q_{\bar\ep_{12}}$ are principal $U(2m)$-bundles over~$X\t\cS^1$. 

We claim there is an isomorphism $Q_{\bar\ep_1}\op Q_{\bar\ep_2}\cong Q_{\bar\de_0}\op Q_{\bar\ep_{12}}$. To see this, note that as $Q_{\bar\de_0},Q_{\bar\ep_1},Q_{\bar\ep_2},Q_{\bar\ep_{12}}$ are isomorphic to $P$ on $X\t\{1\}$, we can choose $Q_{\bar\ep_1},\nabla_{Q_{\bar\ep_1}}$ to be isomorphic to $\pi_X^*(P)$, $\pi_X^*(\nabla_P)$ on $X\t\{e^{i\th}:\Im e^{i\th}\le\ha\}$ and $Q_{\bar\ep_2},\nabla_{Q_{\bar\ep_2}}$ to be isomorphic to $\pi_X^*(P)$, $\pi_X^*(\nabla_P)$ on $X\t\{e^{i\th}:-\ha\le\Im e^{i\th}\}$. Similarly we can choose $Q_{\bar\ep_{12}},\nabla_{Q_{\bar\ep_{12}}}$ to be isomorphic to $\pi_X^*(P)$, $\pi_X^*(\nabla_P)$ on $X\t\{e^{i\th}:-\ha\le\Im e^{i\th}\le\ha\}$, and to be isomorphic to $Q_{\bar\ep_1},\nabla_{Q_{\bar\ep_1}}$ on $X\t\{e^{i\th}:\Im e^{i\th}\ge\ha\}$, and to be isomorphic to $Q_{\bar\ep_2},\nabla_{Q_{\bar\ep_2}}$ on~$X\t\{e^{i\th}:-\ha\ge\Im e^{i\th}\}$.

Then the difference between $Q_{\bar\ep_1}\op Q_{\bar\ep_2},\nabla_{Q_{\bar\ep_1}}\op\nabla_{Q_{\bar\ep_2}}$ and $Q_{\bar\de_0}\op Q_{\bar\ep_{12}},\nabla_{Q_{\bar\de_0}}\op\nabla_{Q_{\bar\ep_{12}}}$ is that on $X\t\{e^{i\th}:\Im e^{i\th}\ge\ha\}$, the two $\U(m)$-bundle factors are exchanged. We can continuously deform $Q_{\bar\ep_1}\op Q_{\bar\ep_2},\nabla_{Q_{\bar\ep_1}}\op\nabla_{Q_{\bar\ep_2}}$ to $Q_{\bar\de_0}\op Q_{\bar\ep_{12}},\nabla_{Q_{\bar\de_0}}\op\nabla_{Q_{\bar\ep_{12}}}$, producing an isomorphism $Q_{\bar\ep_1}\op Q_{\bar\ep_2}\cong Q_{\bar\de_0}\op Q_{\bar\ep_{12}}$, by `rotating' the two factors $P\op P$ between themselves on the regions $X\t\{e^{i\th}:-\ha\le\Im e^{i\th}\le\ha\}$, using isomorphisms of the form $\bigl(\begin{smallmatrix} \cos\psi\,\id_P & \sin\psi\,\id_P \\ -\sin\psi\,\id_P & \cos\psi\,\id_P \end{smallmatrix}\bigr)$, where $\psi$ deforms from 0 to $\pi$ on each interval in~$\{e^{i\th}\in\cS^1:-\ha\le\Im e^{i\th}\}$.

We have
\ea
&\la_P([\ep_1])\la_P([\ep_2])=\check\Om^{E_\bu}_{Q_{\bar\ep_1}}(\ga_{Q_{\bar\ep_1},P}^{F_\bu}\vert^{-1}_{[\nabla_{Q_{\bar\ep_1}}]}(m_{\ep_1}))\check\Om^{E_\bu}_{Q_{\bar\ep_2}}(\ga_{Q_{\bar\ep_2},P}^{F_\bu}\vert^{-1}_{[\nabla_{Q_{\bar\ep_2}}]}(m_{\ep_2}))
\nonumber\\
&=\bigl(\check\Om^{E_\bu}_{Q_{\bar\ep_1}}\bt \check\Om^{E_\bu}_{Q_{\bar\ep_2}}\bigr)\bigl(\ga_{Q_{\bar\ep_1},P}^{F_\bu}\bt \ga_{Q_{\bar\ep_2},P}^{F_\bu}\bigr)\big\vert^{-1}_{([\nabla_{Q_{\bar\ep_1}}],[\nabla_{Q_{\bar\ep_2}}])}\bigl(m_{\ep_1}\bt m_{\ep_2}\bigr)
\nonumber\\
&=\bigl(\check\Om^{E_\bu}_{Q_{\bar\ep_1}}\bt \check\Om^{E_\bu}_{Q_{\bar\ep_2}}\bigr)\bigl((\Ga_{Q_{\bar\ep_1},P}\t\Ga_{Q_{\bar\ep_2},P})^*(\chi_{P,P}^{F_\bu})\ci{} 
\nonumber\\
&\qquad\qquad(\ga_{Q_{\bar\ep_1},P}^{F_\bu}\bt \ga_{Q_{\bar\ep_2},P}^{F_\bu})\bigr)\big\vert^{-1}_{([\nabla_{Q_{\bar\ep_1}}],[\nabla_{Q_{\bar\ep_2}}])}\bigl(\chi_{P,P}^{F_\bu}(m_{\ep_1}\bt m_{\ep_2})\bigr)
\nonumber\\
&=\bigl(\check\Om^{E_\bu}_{Q_{\bar\ep_1}}\!\bt\! \check\Om^{E_\bu}_{Q_{\bar\ep_2}}\bigr)\bigl(\Phi_{Q_{\bar\ep_1},Q_{\bar\ep_2}}^*(\ga_{Q_{\bar\ep_1}\op Q_{\bar\ep_2},P\op P}^{F_\bu})\!\ci\!(\check\phi_{Q_{\bar\ep_1},Q_{\bar\ep_2}}^{E_\bu}\ot 
(\Nu_{Q_{\bar\ep_1},P}\t\Nu_{Q_{\bar\ep_2},P})^*{} 
\nonumber\\
&\qquad\qquad(\check\phi^{E_\bu}_{P\t\cS^1,P\t\cS^1}))\bigr)\big\vert^{-1}_{([\nabla_{Q_{\bar\ep_1}}],[\nabla_{Q_{\bar\ep_2}}])}\bigl(\chi_{P,P}^{F_\bu}(m_{\ep_1}\bt m_{\ep_2})\bigr)
\nonumber\\
&=\check\Om^{E_\bu}_{Q_{\bar\ep_1}\op Q_{\bar\ep_2}}\ci\ga_{Q_{\bar\ep_1}\op Q_{\bar\ep_2},P\op P}^{F_\bu}\big\vert^{-1}_{([\nabla_{Q_{\bar\ep_1}}\op\nabla_{Q_{\bar\ep_2}}])}\bigl(\chi_{P,P}^{F_\bu}(m_{\ep_1}\bt m_{\ep_2})\bigr),
\label{os8eq1}
\ea
using the definition of $\la_P$ in the first step, the morphism $\chi_{P,P}^{F_\bu}$ in \eq{os5eq3} in the third, equation \eq{os5eq6} in the fourth, and Theorem \ref{os5thm3}(c)(ii) in the fifth. Similarly
\e
\begin{split}
&\la_P([\de_0])\la_P([\ep_{12}])=\\
&\quad\check\Om^{E_\bu}_{Q_{\bar\de_0}\op Q_{\bar\ep_{12}}}\ci\ga_{Q_{\bar\de_0}\op Q_{\bar\ep_{12}},P\op P}^{F_\bu}\big\vert^{-1}_{([\nabla_{Q_{\bar\de_0}}\op\nabla_{Q_{\bar\ep_{12}}}])}\bigl(\chi_{P,P}^{F_\bu}(m_{\de_0}\bt m_{\ep_{12}})\bigr).
\end{split}
\label{os8eq2}
\e

Now under the isomorphism $Q_{\bar\ep_1}\op Q_{\bar\ep_2}\cong Q_{\bar\de_0}\op Q_{\bar\ep_{12}}$, which identifies $\check\Om^{E_\bu}_{Q_{\bar\ep_1}\op Q_{\bar\ep_2}}$ with $\check\Om^{E_\bu}_{Q_{\bar\de_0}\op Q_{\bar\ep_{12}}}$ and $\ga_{Q_{\bar\ep_1}\op Q_{\bar\ep_2},P\op P}^{F_\bu}$ with $\ga_{Q_{\bar\de_0}\op Q_{\bar\ep_{12}},P\op P}^{F_\bu}$, we can find a continuous deformation from $\nabla_{Q_{\bar\ep_1}}\op\nabla_{Q_{\bar\ep_2}}$ to $\nabla_{Q_{\bar\de_0}}\op\nabla_{Q_{\bar\ep_{12}}}$, and as $\ep_1*\ep_2$ is isotopic to $\de_0*\ep_{12}$, covering this deformation there is a continuous deformation from $\chi_{P,P}^{F_\bu}(m_{\ep_1}\bt m_{\ep_2})$ to $\chi_{P,P}^{F_\bu}(m_{\de_0}\bt m_{\ep_{12}})$. Therefore the right hand sides of \eq{os8eq1}--\eq{os8eq2} are equal, so that $\la_P([\ep_1])\la_P([\ep_2])=\la_P([\de_0])\la_P([\ep_{12}])$. Since $\la_P([\de_0])=1$ from above this gives $\la_P([\ep_1])\la_P([\ep_2])=\la_P([\ep_{12}])=\la_P([\ep_1]\cdot[\ep_2])$, for all $[\ep_1],[\ep_2]\in\pi_1(\dot K_P^{F_\bu})$. Thus $\la_P$ is a group morphism with~$\la_P\ci\mu_P(n)=(-1)^n$. 

Hence as above, $\la_P$ corresponds to a spin structure $\si_P^{F_\bu}$ on $\B_P$. Thus, starting with data $\check\Om^{E_\bu}_Q$ satisfying Theorem \ref{os5thm3}(c), we have constructed a unique spin structure $\si_P^{F_\bu}$ on $\B_P$ for all $\U(m)$-bundles $P\ra X$, as in Theorem \ref{os5thm3}(b). If $P_1\ra X$, $P_2\ra X$ are principal $\U(m_1)$- and $\U(m_2)$-bundles then we can use the relation between $\check\Om^{E_\bu}_{Q_1},\check\Om^{E_\bu}_{Q_2}$ and $\check\Om^{E_\bu}_{Q_1\op Q_2}$ in Theorem \ref{os5thm3}(c)(ii) to prove that $\si_{P_1}^{F_\bu},\si_{P_2}^{F_\bu},\si_{P_1\op P_2}^{F_\bu}$ are compatible as in Theorem \ref{os5thm3}(b). So this defines a map from data (c) to data (b). It is not difficult to show from the definitions that the maps (b) $\Ra$ (c), (c) $\Ra$ (b) above are inverse. This proves~(ii).

For (iii), there is an obvious forgetful map from data (c) to data (d), by restricting from $\check\Om^{E_\bu}_Q$ for all $\U(m)$-bundles $Q\ra X\t\cS^1$ to $\check\Om^{E_\bu}_Q$ for $\U(m)$-bundles $Q\ra X\t\cS^1$ such that $P=Q\vert_{X\t\{1\}}\cong X\t\U(m)$. To see that this forgetful map is a bijection, note that the proof of (b) $\Leftrightarrow$ (c) above may be restricted to bundles $P\ra X$, $Q\ra X\t\cS^1$ with $Q\vert_{X\t\{1\}}\cong P$, such that $P$ is isomorphic to a trivial $\U(m)$-bundle. This gives a 1-1 correspondence between data (b) restricted to trivial bundles $P$, and data (d). But the 1-1 correspondence between (a) and (b) implies that data (b) restricted to trivial bundles $P$ is in 1-1 correspondence with data (b). This completes the proof of Theorem~\ref{os5thm3}.

\appendix
\section{Elliptic boundary value problems}
\label{osA}

We review here the theory of elliptic boundary conditions developed by B\"ar--Ballmann \cite{BaBa}, and generalize it to families. The point is to establish Fredholm results for a large class of boundary conditions. Essentially, a finite-dimensional perturbation of the well-known Atiyah--Patodi--Singer boundary conditions (decay conditions) is allowed. This further generality is crucial for performing the necessary deformations in the proof of Theorem~\ref{os5thm2}.

Let $X$ be a compact Riemannian manifold with boundary and inward normal coordinate $\th \geq 0$. Let $E_\bu = (E_0, E_1, D)$ be a first order real elliptic operator as in Definition \ref{os2def2} with $E_0, E_1$ equipped with Euclidean metrics $h_{E_0}, h_{E_1}$ on the fibres. On a collar $Z_r$, $r>0$, of $\partial X$ we have coordinates $(y,\th) \in \partial X \t [0,r)$. Let $\pi_{\partial X}\colon Z_r \ra \partial X$, $\pi_{\partial X}(y,\th)=y$ be the projection. Set $F=E_0\vert_{\partial X}$ and choose orthogonal isomorphisms $E_0\vert_{Z_r} \cong \pi_{\partial X}^*(F)$, $E_1\vert_{Z_r} \cong \pi_{\partial X}^*(F)$. In terms of these identifications, assume that one can write
\e
D\vert_{Z_r}=J\big(
\partial/\partial \th+A_\th\big),
\label{osAeq1}
\e
where
\begin{itemize}
\setlength{\itemsep}{0pt}
\setlength{\parsep}{0pt}
\item[(i)]
$J=\sigma_{d\th}(D)\colon F \ra F$ is an orthogonal isomorphism; and
\item[(ii)]
${A_\th\colon \Ga^\iy(F) \ra \Ga^\iy(F)}$ is a $[0,r)$-family of self-adjoint elliptic operators.
\end{itemize}
Operators of this form are called \emph{boundary symmetric} and the operator $A_0$ is called a \emph{boundary operator} for $D$.
As $A_0$ is not uniquely determined by $D$, we fix a choice.
Since $A_0$ is self-adjoint and $\partial X$ compact, $L^2(F)$ decomposes into eigenspaces $\Eig^\la(A_0)$ for eigenvalues $\la \in \R$. For $k\geq 0$ let $L^2_k(F) \subset L^2(F)$ denote the $k^{\rm th}$ Sobolev space of sections.

\begin{dfn}[{B\"ar--Ballmann~\cite[Def.~7.5]{BaBa}}]
\label{osAdef1}
Let $D$ be a boundary symmetric first order real elliptic operator as in \eq{osAeq1} with boundary operator $A_0$. An \emph{elliptic boundary condition} for $D$ is a subspace ${B \subset L^2_{1/2}(F)}$ of the following form. There exists an orthogonal decomposition $L^2(F) = V_- \op W_- \op W_+ \op V_+$ and $g\colon V_- \ra V_+$ with
\begin{itemize}
\setlength{\itemsep}{0pt}
\setlength{\parsep}{0pt}
\item[(i)] closed subspaces $V_\pm$ containing all but finitely many $\Eig^\la(A_0)$, $\pm\la > 0$;
\item[(ii)] $W_\pm$ are finite-dimensional and contained in $L^2_{1/2}(F)$;
\item[(iii)] $g$ is a bounded operator with ${g\big(V_- \cap L^2_{1/2}(F)\big) \subset V_+ \cap L^2_{1/2}(F)}$ and for the adjoint ${g^*\big(V_+ \cap L^2_{1/2}(F)\big) \subset V_- \cap L^2_{1/2}(F)}$.
\end{itemize}
For an elliptic boundary condition it is then required that
\begin{equation*}
B = \left\{
\varphi \in L^2_{1/2}(F)
\enskip\middle\vert\enskip
\begin{array}{l}
\varphi = w_+ + v_- + g(v_-)\\
w_+ \in W_+\\
v_- \in V_- \cap L^2_{1/2}(F)
\end{array}
\right\}.
\end{equation*}
\end{dfn}

\begin{thm}[{B\"ar--Ballmann~\cite[Th.~8.5]{BaBa}}]
\label{osAthm1}
An elliptic boundary condition\/ ${B \subset L^2_{1/2}(F)}$ for a boundary symmetric elliptic operator\/ $D$ as in \eq{osAeq1} on a compact manifold\/ $X$ restricts to a Fredholm operator
\begin{equation*}
D_B \colon L^2_D(E_0;B)
\longrightarrow
L^2(E_1)
\end{equation*}
with domain\/ $L^2_D(E_0;B)$ the Hilbert space of weak solutions\/ $e_0 \in L^2(E_0)$ satisfying\/ $e_0\vert_{\partial X} \in B$, equipped with the graph norm\/ $\|e_0\|^2_D = \|e_0\|^2 + \|De_0\|^2$.
\end{thm}

All boundary conditions we use are variations of the following two examples.

\begin{ex}\label{osAex3}
Let $\mu \in \R$. The \emph{Atiyah--Patodi--Singer boundary
conditions}
\begin{equation*}
B_\mu^\mathrm{APS} \coloneqq \Eig^{(-\iy, \mu)}(A_0)
\end{equation*}
are elliptic, as ${V_- = \Eig^{(-\iy, \mu)}(A_0)}$, ${V_+ = \Eig^{[\mu, \iy)}(A_0)}$, ${g=0}$, and ${W_\pm = \{0\}}$.
\end{ex}

\begin{ex}\label{osAex4}
Let $X'$ be obtained from a closed manifold $X$ by cutting along a hypersurface $Y$, so $\partial X' = -Y \sqcup Y$. Let $E'_\bu=(E'_0,E'_1,D')$ be the pullback of $E_\bu$ to $X'$. Then $L^2(\partial X',E'_0\vert_{\partial X'})=L^2(Y,F) \op L^2(Y,F)$, and we have boundary operators $\pm A_0$. \emph{Transmission boundary conditions} are elliptic
\begin{equation*}
B^\mathrm{trans} \coloneqq \bigl\{(\varphi,\varphi) \in L^2_{1/2}(\partial X',E'_0\vert_{\partial X'})\bigr\},
\end{equation*}
as $V_+ = \Eig^{(0,\iy)}(A_0) \op \Eig^{(-\iy,0)}(A_0),$
$V_- = \Eig^{(-\iy,0)}(A_0) \op \Eig^{(0,\iy)}(A_0)$,
$W_\pm = \{(\varphi,\pm \varphi) \in L^2(\partial X',E'_0\vert_{\partial X'}) \mid \varphi \in \Ker A_0\},$
and $g=\operatorname{swap}$.
\end{ex}

We now generalize to families.

\begin{dfn}
\label{osAdef2}
Let $T$ be a paracompact Hausdorff topological space and $E_0, E_1 \ra X \t T$ real vector bundles with Euclidean metrics on the fibres and set $F=E_0\vert_{\partial X \t T}$. Consider a $T$-family $D$ of first order real elliptic differential operators
\begin{equation*}
D\vert_{X\t\{t\}} \colon \Ga^\iy\left(X,E_0\vert_{X\t\{t\}}\right) \longra \Ga^\iy\left(X,E_1\vert_{X\t\{t\}}\right), \qquad t \in T.
\end{equation*}
Suppose each $D\vert_{X\t\{t\}}$ has a decomposition \eq{osAeq1} with a bundle automorphism $J$ of $F$ and a $T\t[0,r)$-family of self-adjoint operators $A$. A \emph{$T$-family of elliptic boundary conditions} for $D$ is a Hilbert subbundle $B$ of the bundle of fibrewise sections $L^2_{1/2}(\partial X, F) \ra T$ with the property that each $B\vert_t$ is an elliptic boundary condition for $D\vert_{X\t\{t\}}$ and $A_0\vert_{\partial X \t \{t\}}$, in the sense of Definition~\ref{osAdef1}.
\end{dfn}

\begin{prop}
\label{osAprop1}
Let\/ $B$ be a\/ $T$-family of elliptic boundary conditions for a\/ $T$-family of operators\/ $D$ as in Definition~{\rm\ref{osAdef2}}. Then\/ $L^2_D(E_0;B) \ra T$ is a Hilbert bundle and\/ $D_B \colon L^2_D(E_0;B) \ra L^2(E_1)$ is a\/ $T$-family of bounded operators.      
\end{prop}

\begin{prop}
\label{osAprop2}
Let\/ $D$ be a\/ $T$-family of elliptic operators as in Definition\/ {\rm\ref{osAdef2}} and let\/ ${B' \subset B \subset L^2_{1/2}(F)}$ be a pair of\/ $T$-families of elliptic boundary conditions for\/ $D$ with\/ $B/B'$ of finite rank. Then we get a canonical isomorphism of determinant line bundles of the corresponding families of Fredholm operators
\e
\label{osAeq2}
\det_{\R}(D_B)
\longrightarrow
\det_{\R}(D_{B'})
\ot 
\det_{\R}(B/B').
\e
If\/ ${B'' \subset B' \subset B}$ are elliptic boundary conditions
for\/ $D$, with\/ $B/B''$ of finite rank, then \eq{osAeq2} for $B/B',B'/B'',B/B''$ are associative in the obvious~way.
\end{prop}

\begin{proof}
Let $K$ be the orthogonal complement of $B$ in $L^2_{1/2}(F)$. Then
\begin{equation*}
D^K \colon L^2_{D}(X) \longrightarrow L^2(X) \op K,
\quad
f \longmapsto
\big(Df, \pi_{K}\operatorname{res}_{\partial X}(f)\big),
\end{equation*}
has the same kernel and cokernel as $D_B$ and so $\det_\R D^K = \det_\R D_B$.
Let $K'$ be the orthogonal complement
of $B'$. Then ${D^K = (\id_{L^2(X)} \op p)\circ D^{K'}}$ for the
projection ${p \colon K' \ra K}$, whose kernel is canonically isomorphic to $B/B'$.
As determinants are compatible with composition and direct sums, this implies
\begin{align*}
\det_\R D^K
&\cong
\det_\R D^{K'} \ot \det_\R (\id_{L^2(X)} \op p)\\
&\cong
\det_\R D^{K'} \ot \det_\R p\\
&\cong
\det_\R D^{K'} \ot \det_\R(B/B'). \qedhere
\end{align*}
\end{proof}

Using \cite[(62)]{BaBa} we find
\e
\label{osAeq3}
\det_\R(D_B) \cong \det_\R \Ker( D_B ) \ot \big( \det_\R \Ker  (D^*)_{B^\mathrm{adj}} \big)^*
\e
for the \emph{adjoint boundary condition} of \cite[(63)]{BaBa}.

\medskip

\noindent{\small\sc The Mathematical Institute, Radcliffe
Observatory Quarter, Woodstock Road, Oxford, OX2 6GG, U.K.

\noindent E-mails: {\tt joyce@maths.ox.ac.uk, 
upmeier@maths.ox.ac.uk.}}

\end{document}